\documentclass[12pt]{amsart}
\usepackage[utf8]{inputenc}
\usepackage{amsthm}
\usepackage{amsmath}
\usepackage{amsfonts,amssymb,latexsym}
\usepackage{mathtools}
\usepackage{enumerate}
\usepackage{bbm}
\usepackage{relsize}
\usepackage{wasysym}
\usepackage{tikz}
\usepackage{mathrsfs}
\usepackage[english]{babel}
\usepackage{forest}
\usepackage{pdflscape}
\usepackage{fullpage}
\usepackage{times}
\usepackage{enumitem}
\usepackage{eurosym}
\usepackage[all,cmtip]{xy}
\usepackage[margin=1.5cm]{geometry}
\usepackage{biblatex}
\usepackage{hyperref}

\hypersetup{colorlinks=true, citecolor=blue, linkcolor=blue, urlcolor=blue, pdfstartview=FitH, pdfauthor=Goldmakher-Martin-Péringuey, pdftitle=Refinements of Artin's primitive root conjecture}

\usepackage{tikz-cd}
\usetikzlibrary{positioning,arrows}
\usetikzlibrary{calc}
\usepackage{subfiles}

\theoremstyle{plain} 
\newtheorem{thm}{Theorem}[section]
\newtheorem{prop}[thm]{Proposition}
\newtheorem{lemma}[thm]{Lemma}
\newtheorem{cor}[thm]{Corollary}

\theoremstyle{remark}

\theoremstyle{definition}
\newtheorem{defn}[thm]{Definition}

\newtheorem{notn}[thm]{Notation}
\newtheorem{rmk}[thm]{Remark}

\renewcommand{\mod}[1]{
{\ifmmode\text{\rm\ (mod~$#1$)}
\else\discretionary{}{}{\hbox{\!\!}}\rm(mod~$#1$)\fi}}

\newcommand{\ord}{\mathop{\rm ord}\nolimits}

\newcommand{\Gal}{\textup{Gal}}

\newcommand{\sgn}{\textup{sgn}}

\newcommand{\uu}{\mathcal U}
\newcommand{\ff}{\mathcal F}

\renewcommand{\aa}{\mathcal A}
\newcommand{\ee}{\mathcal E}

\newcommand{\N}{{\mathbb N}}
\newcommand{\Z}{{\mathbb Z}}

\newcommand{\Q}{{\mathbb Q}}

\newcommand{\C}{{\mathbb C}}

\newcommand{\bo}{{\mathit{O}}}

\newcommand{\Na}{{\mathcal N_{a}}}

\newcommand{\Tr}{{\rm Tr}}

\newcommand{\id}{{\mathbbm{1}}}

\newcommand{\schinzeln}{{\ell}}

\newcommand{\schinzelm}{{m\ell}}

\newcommand{\Palm}{P_{a,\ell,m}}
\newcommand{\Pali}{P_{a,\ell,1}}
\newcommand{\newa}{\alpha}
\newcommand{\newb}{\beta}

\title{Refinements of Artin's primitive root conjecture}

\author{Leo Goldmakher}
\address{Department of Mathematics and Statistics, Williams College, Williamstown, MA, USA 01267}
\email{Leo.Goldmakher@williams.edu}

\author{Greg Martin}
\address{Department of Mathematics, University of British Columbia, Vancouver V6T 1Z2, Canada}
\email{gerg@math.ubc.ca}

\author{Paul Péringuey}
\address{Department of Mathematics, University of British Columbia, Vancouver V6T 1Z2, Canada}
\email{peringuey@math.ubc.ca}

\begin{document}
\selectlanguage{english}

\dedicatory{Dedicated to John Friedlander on the occasion of his $\Bigl(70+O(1)\Bigr)^\text{th}$ birthday.}

\begin{abstract}
A famous conjecture of Artin asserts that any integer~$a$ that is neither $-1$ nor a square should be a primitive root (mod~$p$) for a positive proportion of primes~$p$. Moreover, using a heuristic argument, Artin guessed an explicit formula for the proportion; this formula is well-supported by computations and is known to hold on a generalized Riemann hypothesis, but remains open.
In this paper we propose several conjectures that capture the finer properties of the distribution of the order of~$a$ (mod~$p$) as~$p$ varies over primes; these assertions contain Artin's original conjecture as a special case. We prove these conjectures assuming the generalized Riemann hypothesis, as well as weaker versions unconditionally.
\end{abstract}
\maketitle

\thispagestyle{empty}

\numberwithin{equation}{section}

\section{Introduction}

Given an integer~$a$ and a prime~$p$, let $\ord_p(a)$ denote the multiplicative order of~$a$ modulo~$p$. Exploring the distribution of $\ord_p(a)$, for~$a$ fixed as~$p$ varies, is a classical subject:

in 1927, Emil Artin proposed a heuristic argument that for any fixed integer $a \neq -1$ that is not a perfect square, $\ord_p(a) = p-1$ for a positive proportion of primes~$p$.
His key insight was that this problem is most naturally viewed through the lens of algebraic number theory.
Artin knew that $\ord_p(a) = p-1$ if and only if~$p$ fails to split completely in the cyclotomic-Kummer extension $\Q(\sqrt[q]{a},\zeta_q)$ for all primes~$q$.
On the other hand, the prime ideal theorem implies that the density of those primes that split completely in $\Q(\sqrt[q]{a},\zeta_q)$ is
${1}/{[\Q(\sqrt[q]{a},\zeta_q) : \Q]} = {1}/{q(q-1)}$.
Making the seemingly innocuous assumption that the events ``$p$ does not split completely in $\Q(\sqrt[q]{a},\zeta_q)$'' are independent, Artin predicted that the proportion of primes~$p$ for which~$a$ is a primitive root should be
\[
\prod_{q \text{ prime}}
 \left( 1 - \frac{1}{q(q-1)} \right)
\approx
0.3739558,
\]
which is now known as Artin's constant.

Extensive computational work by Emma Lehmer and Derrick H. Lehmer provided evidence that Artin's conjecture was correct for some values of~$a$ but not for others, prompting Artin to write
``I was careless but the machine caught up with me.''
He realized what was going wrong in his original heuristic: for all prime $a \equiv 1 \mod 4$ we have $\sqrt a \in \Q(\zeta_a)$, whence
$\Q(\sqrt[2]{a},\zeta_2) \subseteq \Q(\sqrt[a]{a},\zeta_a)$. In particular, any $p$ that does not split in $\Q(\sqrt[2]{a},\zeta_2)$ cannot split in $\Q(\sqrt[a]{a},\zeta_a)$ either, so the events ``$p$ does not split completely in $\Q(\sqrt[q]{a},\zeta_q)$'' are not independent after all.

Artin predicted that there were no other dependences among splitting in the fields $\Q(\sqrt[q]{a},\zeta_q)$ whenever~$a$ is prime.
In the preface to his collected works, Artin's former students Lang and Tate wrote down a modified version of Artin's conjecture
(working independently, Heilbronn made the same conjecture in 1968).
To state the corrected version of Artin's conjecture, we require the following notation:

\begin{defn} \label{s and d def}
For any nonzero rational number~$g$, define $s(g)$ to be the {\em squarefree kernel} of $g$, which is the sign of~$g$ times the product of all primes dividing~$g$ to an odd power, or equivalently the unique squarefee integer such that $gs(g)$ is the square of a rational number, or equivalently the unique squarefee integer such that $\Q(\sqrt g) = \Q(\sqrt{s(g)})$. Also define $\mathfrak{d}(g)$ to be the \emph{discriminant} of $\Q(\sqrt g)$, so that
\[
\mathfrak{d}(g)=\begin{cases}
s(g),&\text{if }s(g)\equiv1\mod{4},\\
4s(g),&\text{otherwise.}
\end{cases}
\]
\end{defn}

Artin's conjecture, which can be stated for rational numbers as well (since there are only finitely many primes for which the multiplicative order is not defined), was proved by Hooley~\cite{Hooley} under the assumption of the generalized Riemann hypothesis for the Dedekind zeta functions of the cyclotomic-Kummer extensions mentioned earlier. It takes the following form if we assume that~$a$ is not a perfect power (see Theorem~\ref{definition general Artin itself} below for a more general version):

\begin{thm} \label{OG Artin}
Assume GRH. If $a\in\Q$ is not a perfect power, then
\begin{equation} \label{OG Artin density}
\frac{1}{\pi(x)} \#\{p \leq x \colon a \text{ is a primitive root (mod $p$)}\} \sim F_a \prod_q \biggl( 1-\frac1{q(q-1)} \biggr),
\end{equation}
where
\[
F_a=\begin{cases}
 1,& \text{if }\mathfrak{d}(a) \not\equiv1\mod{4},\\
 \displaystyle 1-\prod\limits_{q \mid \mathfrak{d}(a)} \frac{-1}{q^2-q-1},& \text{if }\mathfrak{d}(a) \equiv1\mod{4}.
 \end{cases}
\]
\end{thm}

\begin{rmk}
A different expression for the density on the right-hand side of equation~\eqref{OG Artin density}, which is simpler to write down but harder to use, is
\begin{equation} \label{looks like degrees}
\sum_{\ell=1}^\infty \frac{\mu(\ell)}{[\Q(\sqrt[\ell]{a},\zeta_\ell)] : \Q]}
\end{equation}
\end{rmk}

Our goal in the present work is to formulate several conjectures which subsume Artin's conjecture and produce predictions about finer aspects of the distribution of $\ord_p(a)$ as~$p$ varies. 
As a prototype of our results, consider the simplest case of Artin's conjecture, namely that the proportion of primes~$p$ for which~$2$ is a primitive root is Artin's constant $\prod_q \Bigl(1-{1}/{q(q-1)}\Bigr)$.
Consider the deformation
$\prod_q \Bigl(1-(1-z)/{q(q-1)}\Bigr)$ of this product, which admits a power series in~$z$ whose constant term is precisely Artin's constant; in other words, the constant term measures the density of primes~$p$ for which 
$(p-1)/{\ord_p(2)} = 1$.
The starting point of our investigation is the observation that the other power series coefficients also have relatively simple interpretations in terms of $\ord_p(2)$: the coefficient of~$z^n$ is the density of those primes~$p$ for which $(p-1)/{\ord_p(2)}$ has precisely~$n$ distinct prime factors.
Below, we will generalize this observation considerably (to rational $a$ and to other possible deformations of the product), prove our generalizations on GRH, and obtain weaker versions of our conjectures unconditionally.

\medskip

Throughout this article,~$p$ and~$q$ denote prime numbers, and~$a$ denotes a rational number not equal to~$-1$, $0$, or~$1$.
We define $\nu_p(a)$ to be the unique integer~$\nu$ such that $a = p^\nu \frac mn$ where~$p$ divides neither~$m$ nor~$n$. In other words, if~$p$ divides the numerator of~$a$ (including~$a$ itself if~$a$ is an integer) then $\nu_p(a)$ is the power to which it divides the numerator, while if~$p$ divides the denominator of~$a$ then $-\nu_p(a)$ is the power to which it divides the denominator.
We use the standard notation $\omega(n)$ for the number of distinct prime factors of~$n$ and $\Omega(n)$ for the number of prime factors of~$n$ counted with multiplicity.

\subsection{Statement of results: generating functions}

We compare $p-1$ with $\ord_p(a)$ in three ways. The first way is to count the number of distinct prime factors of their quotient, namely $\omega\Bigl( (p-1)/\ord_p(a) \Bigr)$, which is the same as the number of distinct prime factors of the index of the subgroup $\langle a\rangle$ in $(\Z/p\Z)^*$. As this quantity is a nonnegative integer, it makes sense to frame its limiting distribution in terms of a generating function; our first theorem gives an exact formula for that generating function.

\begin{thm}\label{theorem generating function omega of quotient h=1}
Assume GRH. If $a\in\Q$ is not a perfect power, then for any $z\in\C$ with $|z|\leq 1$,
\[
\frac{1}{\pi(x)} \sum_{\substack{p\le x \\ \nu_p(a)=0}} z^{\omega((p-1)/\ord_p(a))}
=
F^{\omega/}_a(z) \prod_q \biggl( 1+\frac{z-1}{q(q-1)} \biggr) + \bo\biggl(\frac{x\log\log x}{(\log x)^{2}}\biggr)
\]
where 
\[
F^{\omega/}_a(z)=\begin{cases}
1, & \text{if } \mathfrak{d}(a) \not\equiv 1\mod{4},\\
1+\displaystyle\prod_{q \mid 2\mathfrak{d}(a)} \frac{z-1}{z+q^2-q-1} ,
& \text{if } \mathfrak{d}(a) \equiv 1\mod{4}.
\end{cases}
\] 
\end{thm}

\begin{rmk}
The condition ``$a$ is not a perfect power'' in the statement of Theorem~\ref{theorem generating function omega of quotient h=1} can be relaxed, but the factor in front of the infinite product becomes more complicated in the general case. The same remark applies to Theorems~\ref{theorem generating function Omega of quotient h=1} and~\ref{theorem generating function omega minus omega h=1} below.
The most general versions of these three theorems appear later as Theorems~\ref{theorem generating function omega of quotient}, \ref{theorem generating function Omega of quotient}, and~\ref{theorem generating function omega minus omega}.
\end{rmk}

\begin{rmk}
If we set $z=0$ in Theorem~\ref{theorem generating function omega of quotient h=1}, we recover Artin's conjecture as given in Theorem~\ref{OG Artin} (see Remark~\ref{z=0 remark} below for a longer discussion of this reduction). The same remark applies to Theorem~\ref{theorem generating function Omega of quotient h=1} below.
\end{rmk}

We give concrete examples of these generating functions for three specific values of~$a$. Technically these examples are a corollary of Theorem~\ref{theorem generating function omega of quotient} below, although the first and last cases do follow from Theorem~\ref{theorem generating function omega of quotient h=1}.

\begin{cor} \label{corollary generating function omega of quotient a=3,4,5}
Assuming GRH, we have the following three special cases:
\begin{itemize}[itemindent=8mm]
\item[$a=3{:}$\quad]
 $\displaystyle
 \sum_{\substack{p\le x \\ p\ne3}} z^{\omega((p-1)/\ord_p(3))}=\pi(x)\prod\limits_{\substack{q}}\biggl(1+\frac{z-1}{q(q-1)}\biggr)+\bo\biggl(\frac{x\log\log x}{(\log x)^{2}}\biggr),
 $
\item[$a=4{:}$\quad]
 $\displaystyle
 \sum_{\substack{p\le x \\ p\ne2}} z^{\omega((p-1)/\ord_p(4))}=\pi(x)\frac{2z}{z+1}\prod\limits_{\substack{q}}\biggl(1+\frac{z-1}{q(q-1)}\biggr)+\bo\biggl(\frac{x\log\log x}{(\log x)^{2}}\biggr),
 $
\item[$a=5{:}$\quad]
 $\displaystyle
 \sum_{\substack{p\le x \\ p\ne5}} z^{\omega((p-1)/\ord_p(5))}=\pi(x)\frac{2z^2 + 18z + 20}{z^2 + 20z + 19}\prod\limits_{\substack{q}}\biggl(1+\frac{z-1}{q(q-1)}\biggr)+\bo\biggl(\frac{x\log\log x}{(\log x)^{2}}\biggr).
 $
\end{itemize}
\end{cor}

\begin{rmk}
Note that the case $a=4$ is uninteresting in the original context of Artin's conjecture, since a perfect square is never a primitive root modulo any odd prime. This fact is reflected in the second generating function in Corollary~\ref{corollary generating function omega of quotient a=3,4,5}, since the factor~$2z$ on the right-hand side ensures that there are no terms of the form~$z^0$ in the generating function. However, even in this case it is still interesting to examine the full distribution of the statistic $\omega\Bigl( (p-1)/\ord_p(4) \Bigr)$, and the generating function accomplishes this goal. A similar remark applies to Corollary~\ref{corollary generating function Omega of quotient a=3,4,5} below.
\end{rmk}

Our second way of comparing $p-1$ with $\ord_p(a)$ is the same as the first except that we count the prime factors of $(p-1)/\ord_p(a)$ with multiplicity. The formula becomes more complicated in a way that depends upon the factorization of~$a$; the following notation, which we will use throughout this paper, provides us with the needed specificity.

\begin{notn} \label{notation beginning}
Given $a\in\Q\setminus\{-1,0,1\}$, the symbols $h,a_0,b_0,c_0$ will always have the following meanings:
 \begin{itemize}
 \item We write $a=\pm a_0^h$ where $a_0$ is positive and not a perfect power, so that $h$ is the largest integer for which $|a|$ is an $h$th power.
 \item We write $a_0=b_0c_0^2$ where $b_0 = s(a_0)$ is squarefree.
 \end{itemize}
\end{notn}

\begin{thm}\label{theorem generating function Omega of quotient h=1}
Assume GRH. If $a\in\Q$ is not a perfect power, then for any $z\in\C$ with $|z|\leq 1$,
\begin{align*}\sum_{\substack{p\le x \\ \nu_p(a)=0}} z^{\Omega((p-1)/\ord_p(a))}&=\pi(x)F^{\Omega}_a(z)\prod\limits_{\substack{q}}\biggl(1+\frac{(z-1)q}{(q-1)(q^2-z)}\biggr)+\bo\biggl(\frac{x\log\log x}{(\log x)^{2}}\biggr).
\end{align*}
where 
\[
F^{\Omega}_a(z)=1+\delta(a)(z-1)^{\omega(2b_0)}\prod\limits_{q|2b_0}\left(\frac{q}{z+q^3-q^2-q}\right),
\]
with
\[
\delta(a)=\begin{cases}
 1,& \text{if }\sgn(a)b_0 \equiv 1\mod{4},\\
 z/4,& \text{if }\sgn(a)b_0 \equiv 3\mod{4},\\
 z^2/16,& \text{if }b_0 \equiv 2\mod{4}.
\end{cases}
\]
\end{thm}

As before, we give concrete examples of these generating functions for three specific values of~$a$. 

\begin{cor} \label{corollary generating function Omega of quotient a=3,4,5}
Assuming GRH, we have the following three special cases:
\begin{itemize}[itemindent=8mm]
\item[$a=3{:}$\quad]
 $\displaystyle
 \sum_{\substack{p\le x \\ p\ne3}} z^{\Omega((p-1)/\ord_p(3))}=\pi(x)\frac{3z^3 - 4z^2 + 37z + 60}{2z^2 + 34z + 60}\prod\limits_{\substack{q}}\biggl(1+\frac{(z-1)q}{(q-1)(q^2-z)}\biggr)+\bo\biggl(\frac{x\log\log x}{(\log x)^{2}}\biggr),
 $
\item[$a=4{:}$\quad]
 $\displaystyle
 \sum_{\substack{p\le x \\ p\ne2}} z^{\Omega((p-1)/\ord_p(4))}=\pi(x)\frac{z^3 - z^2 + 12z}{4z + 8}\prod\limits_{\substack{q}}\biggl(1+\frac{(z-1)q}{(q-1)(q^2-z)}\biggr)+\bo\biggl(\frac{x\log\log x}{(\log x)^{2}}\biggr),
 $
\item[$a=5{:}$\quad]
 $\displaystyle
 \sum_{\substack{p\le x \\ p\ne5}} z^{\Omega((p-1)/\ord_p(5))}=\pi(x)\frac{11z^2 + 77z + 200}{z^2 + 97z + 190}\prod\limits_{\substack{q}}\biggl(1+\frac{(z-1)q}{(q-1)(q^2-z)}\biggr)+\bo\biggl(\frac{x\log\log x}{(\log x)^{2}}\biggr).
 $
\end{itemize}
\end{cor}

Our third and final way of comparing $p-1$ with $\ord_p(a)$ is to compute the number of distinct primes dividing one but not the other, namely $\omega(p-1) - \omega(\ord_p(a))$. Note that this third statistic is not directly connected to Artin's conjecture, as $\omega(p-1) - \omega(\ord_p(a))$ can equal~$0$ without~$a$ being a primitive root modulo~$p$ (as the example $p=17$ and $a=2$ shows). However the limiting distribution of $\omega(p-1) - \omega(\ord_p(a))$ remains interesting in its own right.

\begin{thm}\label{theorem generating function omega minus omega h=1}
Assume GRH. If $a\in\Q$ is not a perfect power, then for any $z\in\C$ with $|z|\leq 1$,
\begin{align*}\sum_{\substack{p\le x \\ \nu_p(a)=0}} z^{\omega(p-1)-\omega(\ord_p(a))}=\pi(x)F^{\omega-}_a(z)\prod\limits_{\substack{q}}\biggl(1+\frac{z-1}{q^2-1}\biggr)+\bo\biggl(\frac{x\log\log x}{(\log x)^{2}}\biggr).
\end{align*}
where 
\[
F^{\omega-}_a(z)=1+\delta(a)(z-1)^{\omega(2b_0)}\prod\limits_{q|2b_0}\left(\frac{1}{q^2+z-2}\right),
\]
with
\[\delta(a)=\begin{cases}
 1,& \text{if }\sgn(a)b_0 \equiv 1\mod{4},\\
 -1/2,& \text{if }\sgn(a)b_0 \equiv 3\mod{4},\\
 -1/8,& \text{if }b_0 \equiv 2\mod{4}.
\end{cases}
\]
\end{thm}

We again give concrete examples of these generating functions for the same three specific values of~$a$.

\begin{cor} \label{corollary generating function omega minus omega a=3,4,5}
Assuming GRH, we have the following three special cases:
\begin{itemize}[itemindent=8mm]
\item[$a=3{:}$\quad]
 $\displaystyle
 \sum_{\substack{p\le x \\ p\ne3}} z^{\omega(p-1)-\omega(\ord_p(3))}=\pi(x)\frac{z^2+20z+27}{2z^2+18z+28}\prod\limits_{\substack{q}}\biggl(1+\frac{z-1}{q^2-1}\biggr)+\bo\biggl(\frac{x\log\log x}{(\log x)^{2}}\biggr),
 $
\item[$a=4{:}$\quad]
 $\displaystyle
 \sum_{\substack{p\le x \\ p\ne2}} z^{\omega(p-1)-\omega(\ord_p(4))}=\pi(x)\frac{7z+5}{4(z+2)}\prod\limits_{\substack{q}}\biggl(1+\frac{z-1}{q^2-1}\biggr)+\bo\biggl(\frac{x\log\log x}{(\log x)^{2}}\biggr),
 $
\item[$a=5{:}$\quad]
 $\displaystyle
 \sum_{\substack{p\le x \\ p\ne5}} z^{\omega(p-1)-\omega(\ord_p(5))}=\pi(x)\frac{2z^2+23z+47}{z^2+25z+46}\prod\limits_{\substack{q}}\biggl(1+\frac{z-1}{q^2-1}\biggr)+\bo\biggl(\frac{x\log\log x}{(\log x)^{2}}\biggr).
 $
\end{itemize}
\end{cor}

All of these results and their generalizations are proved together in stages. In
Section~\ref{combinatorics sec} we axiomatize the combinatorial setup used by Hooley in his proof of Artin's conjecture, and then apply that generalization to each of our three statistics. The resulting formulas involve degrees of arbitrary cyclotomic-Kummer extensions of the form $\Bigl[\Q\Big(\sqrt[\ell]{a},\zeta_{m\ell}\Big):\Bbb Q\Bigr]$, which we work out in
Section~\ref{degree section}. Again following Hooley's techniques,
in Section~\ref{error section} we establish adequate bounds for all the error terms that arise during our strategy. Finally, in
Section~\ref{hard proofs section} we convert the main terms in these asymptotic formulas from infinite sums involving the degrees of these extensions (of which equation~\eqref{looks like degrees} is an example) into the infinite products that appear in our theorems.

\subsection{Statement of results: individual densities}

A generating function for a nonnegative integer-valued statistic is valuable if we can extract from it information about its coefficients. We introduce the notation
\begin{align}
\sum_{n=0}^\infty D_a^{\omega/}(n;x) z^n &= \frac1{\pi(x)} \sum_{\substack{p\le x \\ \nu_p(a)=0}} z^{\omega((p-1)/\ord_p(a))} \notag \\
\sum_{n=0}^\infty D_a^{\Omega}(n;x) z^n &= \frac1{\pi(x)} \sum_{\substack{p\le x \\ \nu_p(a)=0}} z^{\Omega((p-1)/\ord_p(a))} \label{defining coefficients}\\
\sum_{n=0}^\infty D_a^{\omega-}(n;x) z^n &= \frac1{\pi(x)} \sum_{\substack{p\le x \\ \nu_p(a)=0}} z^{\omega(p-1)-\omega(\ord_p(a))}. \notag
\end{align}
for the coefficients of the polynomials on the right-hand sides, and their limiting values (assuming they exist)
\begin{align}
 D_a^{\omega/}(n) &= \lim_{x\to\infty} D_a^{\omega/}(n;x) \notag \\
 D_a^{\Omega}(n) &= \lim_{x\to\infty} D_a^{\Omega}(n;x) \label{definition_D_a(n)} \\
 D_a^{\omega-}(n) &= \lim_{x\to\infty} D_a^{\omega-}(n;x). \notag
\end{align}
So for example, $D_a^{\omega/}(n)$ is the density of the set of primes~$p$ for which $\omega((p-1)/\ord_p(a)) = n$. Assuming GRH, our theorems show that these densities do indeed exist for all $a\notin\{-1,0,1\}$.
Moreover, our generating function formulas allow numerical evaluation of the densities $D_a^{\omega/}(n)$, $D_a^{\Omega}(n)$, and $D_a^{\omega-}(n)$ to high precision (see Section~\ref{numerical section} for details and examples).

From their generating functions, we can also employ techniques from complex analysis to deduce decay rates for these densities. The following result is proved in Section~\ref{decay section}):

\begin{cor} \label{decay rate corollary}
Assume GRH. For any $a\in\Q\setminus\{-1,0,1\}$:
\begin{enumerate}
\item[(a)] $D_a^{\Omega}(n) = R_a 4^{-n} + O_a(9^{-n})$, where
\[
R_a = \tfrac32 H_h^\Omega(4) \Bigl( 1+I_h^\Omega(\gamma,4) \Bigr) \prod_{q\ge3} \biggl( 1 + \frac{3q}{(q-1)(q^2-4)} \biggr)
\]
and $H^{\Omega}_h(z)$ and $I^{\Omega}_h(\gamma,z)$ are as in Definition~\ref{definition generating function Omega of quotient} below.
\item[(b)] Both $D_a^{\omega/}(n)\le n^{-(2-o(1))n}$ and $D_a^{\omega-}(n)\le n^{-(2-o(1))n}$ (with the implied constants depending on~$a$), with each of these upper bounds being attained infinitely often.
\end{enumerate}
\end{cor}

These densities represent the limiting probability that the given statistics take a certain value if we choose a prime~$p$ ``at random''; in other words, for each statistic (and each value of~$a$) there is a nonnegative integer-valued random variable whose distribution function is the same as the sequence of densities. Such a random variable has an expectation, which we can interpret as the average value of the statistic when we sample it over randomly chosen primes. We compute exact values and numerical approximations for these expectations in Section~\ref{subsection expectations}.

All of the above results are conditional on GRH (just as Hooley's proof of Artin's conjecture is); however, we are able to obtain some unconditional results. More precisely, we can derive some unconditional inequalities for these densities and for the expectations of the three statistics. The following corollary gives two examples:

\begin{cor} \label{unconditional inequality corollary k=0}
Let $a$ be a rational number that is not a perfect power.
Unconditionally,
\begin{align*}
\#\bigl\{ p\le x, \, \nu_p(a)=0\colon \ord_p(a) = p-1 \bigr\} &\le \biggl( F_a \prod_q \biggl( 1-\frac1{q(q-1)} \biggr) + o(1) \biggr) \pi(x) \\
\#\bigl\{ p\le x, \, \nu_p(a)=0\colon \omega(\ord_p(a)) = \omega(p-1) \bigr\} &\le \biggl( F_a^{\omega-}(0) \prod_q \biggl( 1-\frac1{q^2-1} \biggr) + o(1) \biggr) \pi(x).
\end{align*}
in the notation of Theorems~\ref{OG Artin} and~\ref{theorem generating function omega minus omega h=1}.
All implied constants may depend on~$a$.
\end{cor}

The first statement recovers a result of Vinogradov and Van der Waall from the 1970s, as described in~\cite{Moree} (just before the start of Section~6.1).
We give generalizations of this corollary in Theorems~\ref{welcome both lim sup and tilde notation. sigh} and~\ref{welcome both lim sup and tilde notation. no sigh} below.
We note that assuming GRH, all of these inequalities become asymptotic formulas.

\section{Extending Hooley's combinatorics} \label{combinatorics sec}

In Hooley's conditional proof of Artin's conjecture~\cite{Hooley}, he began with a combinatorial decomposition (similar to an inclusion--exclusion sieve) based on the primes dividing $p-1$ that could obstruct~$a$ being a primitive root modulo~$p$. Our work requires a similar combinatorial setup, with the additional complexity of needing a decomposition involving all powers of primes, rather than just the primes themselves. Moreover, the relationship among the ``obstructions'' indexed by powers of a given prime is important to the combinatorial setup. For these reasons, in Sections~\ref{Cumulative sec} and~\ref{Disjoint sec} we will describe these combinatorial setups in an axiomatic way, with the hope that other authors will find their generality useful. Afterwards, we will apply these setups to our three sets of counting functions in Section~\ref{three applications sec}.

Throughout Section~\ref{combinatorics sec}, when~$h(m)$ is a multiplicative function, we let
\begin{equation} \label{inverses}
g(m) = \sum_{d\mid m} \mu(d) h(m/d) \quad\text{and}\quad g^*(m) = \sum_{\substack{d\mid m \\ (d,m/d) = 1}} (-1)^{\omega(d)} h(m/d)
\end{equation}
denote its Dirichlet inverse and its unitary inverse, respectively. Also, for any positive real number~$\xi$, we define~$Q_\xi$ to be the least common multiple of all positive integers less than or equal to~$\xi$.

Furthermore, throughout Section~\ref{combinatorics sec}, let $\uu\subset\N$ be the ``universe'' of numbers under consideration.
For every prime power $q^k$, let $\ff_{q^k}$ denote a subset of $\uu$.
For later convenience, we also define $\ff_1=\uu$. We can think of $\ff_{q^k}$ as being a property that a natural number might or might not have; the letter $\ff$ is chosen to help us think of the numbers in $\ff_{q^k}$ as having ``failed'' a test administered by~$q^k$. We assume that every infinite intersection of various $\ff_{q^k}$ is empty, so that any given integer $n$ has only finitely many of the properties~$\ff_{q^k}$.

\subsection{Cumulative properties} \label{Cumulative sec}

It might be the case that the collection $\{\ff_{q^k}\}$ has the property that $\ff_{q^k} \subseteq \ff_{q^{k-1}}$ for all prime powers~$q^k$, that is, numbers in our universe~$\uu$ with property $\ff_{q^k}$ automatically have property~$\ff_{q^{k-1}}$ (and thus all properties~$\ff_{q^j}$ with $0\le j\le k$, property~$\ff_1$ being automatic). We call such collections ``cumulative'', and for all positive integers $\ell$ we use the following notation:
\begin{itemize}
\item We define new sets $\aa_\ell = \bigcap_{q^k\mid\ell} \ff_{q^k}$ (with the interpretation $\aa_1=\uu$), so that a number $n$ has property~$\aa_\ell$ if and only if it has property $\ff_{q^k}$ for all prime powers $q^k\mid\ell$.
Note that such integers have {\em at least} those properties (hence the letter $\aa$) but might have further properties as well, whether from higher powers of existing primes or from new primes.
\item We also define $A_\ell$ to be the indicator function of $\aa_\ell$ (that is, the function defined by $A_\ell(n)=1$ if $n\in\aa_\ell$ and $A_\ell(n)=0$ if $n\notin\aa_\ell$).
\item Furthermore, we define $\ee_\ell = \aa_\ell \setminus \bigcup_{q^k\nmid\ell} \ff_{q^k}$, so that a number in $\ee_\ell$ has {\em exactly} the properties~$\ff_{q^k}$ for all prime powers $q^k\mid\ell$ and no other properties (hence the letter~$\ee$).
\item Define $e(n)$ to be the unique number such that $n\in\ee_{e(n)}$, which exists for all $n$ by the infinite-intersection property and the cumulative property of the $\ff_{q^k}$; note that by cumulativeness, $n\in\aa_\ell$ if and only if $\ell\mid e(n)$. Further, define $e_\xi(n) = \gcd(e(n),Q_\xi)$, which we think of as recording which tests up to~$\xi$ the number~$n$ fails.
\item Finally, define
\[
N_\xi(x) = \sum_{\substack{n\le x \\ n\in \uu}} h(e_\xi(n))
\quad\text{and}\quad
P_\ell(x) = \#\{ n\le x \colon n\in \aa_\ell \} = \sum_{\substack{n\le x \\ n\in \uu}} A_\ell(n).
\]
\end{itemize}
In practice, $N_y(x)$ will be the exact quantity we are interested in, while the~$P_\ell(x)$ will be quantities for which asymptotic estimates are available.
The following lemma will allow us to estimate $N_\xi(x)$ in terms of these functions $P_\ell(x)$ as well as the Dirichlet inverse $g(m)$ of $h(m)$.

\begin{lemma}\label{cumulative lemma}
Assume that the collection $\{\ff_{q^k}\}$ is cumulative.
Let~$h(m)$ be a multiplicative function bounded in absolute value by~$1$. For any positive real numbers $\xi<y$,
\begin{equation*} 
N_y(x) = \sum_{\ell\mid Q_\xi} g(\ell) P_\ell(x) + O\bigg(\sum_{\xi < q^k \le y} P_{q^k}(x) \bigg).
\end{equation*}
\end{lemma}

\begin{proof}
Given the definition~\eqref{inverses} of $g(m)$, M\"obius inversion yields
\begin{equation} \label{apres Mobius}
h(e_\xi(n)) = \sum_{\ell\mid e_\xi(n)} g(\ell) = \sum_{\substack{\ell\mid e(n) \\ \ell\mid Q_\xi}} g(\ell) = \sum_{\substack{n\in\aa_\ell \\ \ell\mid Q_\xi}} g(\ell) = \sum_{\ell\mid Q_\xi} g(\ell) A_\ell(n),
\end{equation}
and then summing over $n\le x$ with $n\in\uu$ gives
\begin{equation} \label{apres summing}
N_\xi(x) = \sum_{\ell\mid Q_\xi} g(\ell) P_\ell(x).
\end{equation}

Next suppose that $\eta<\theta$ and that there is exactly one prime power~$q^k$ in the interval $(\eta,\theta]$. In this situation, $h(e_\theta(n)) = h(e_\eta(n))$ when $n\notin \aa_{q^k}$, and thus
\[
N_\theta(x) - N_\eta(x) = \sum_{\substack{n\le x \\ n\in \uu}} \Big( h(e_\theta(n)) - h(e_\eta(n)) \Big) = \sum_{\substack{n\le x \\ n\in \aa_{q^k}}} \Big( h(e_\theta(n)) - h(e_\eta(n)) \Big).
\]
Since the summand is at most~$2$ in absolute value, this identity gives
\[
N_\theta(x) - N_\eta(x) \ll \sum_{\substack{n\le x \\ n\in \aa_{q^k}}} 1 = P_{q^k}(x).
\]

If we now partition $(\xi,y]$ into intervals of the form $(\eta,\theta]$ that each contain exactly one of the prime powers~$q^k$ in $(\xi,y]$, summing over the resulting partition yields
\begin{equation*}
N_y(x) - N_\xi(x) \ll \sum_{\xi < q^k \le y} P_{q^k}(x),
\end{equation*}
which establishes the lemma when combined with the identity~\eqref{apres summing}.
\end{proof}

\subsection{Disjoint properties} \label{Disjoint sec}

On the other hand, it might be the case that the collection~$\{\ff_{q^k}\}$ has the property that $\ff_{q^j} \cap \ff_{q^k} = \emptyset$ for all primes~$p$ and all distinct positive integers~$j$ and~$k$, that is, numbers with the property~$\ff_{q^j}$ cannot also have the property~$\ff_{q^k}$ for any positive $j\ne k$. We call such collections ``disjoint'', and for all positive integers~$\ell$ we use the following notation:
\begin{itemize}
\item We define new sets $\aa^*_\ell = \bigcap_{q^k\parallel \ell} \ff_{q^k}$ (with the interpretation $\aa^*_1=\uu$), so that a number $n$ has property~$\aa^*_\ell$ if and only if it has property~$\ff_{q^k}$ for all prime powers~$q^k$ exactly dividing~$\ell$. (Recall that $q^k\parallel\ell$ means that $q^k\mid\ell$ but $q^{k+1}\nmid\ell$.) Note that such integers have {\em at least} those properties but might have further properties as well from powers of other primes, although not from other powers of primes already dividing~$\ell$ due to disjointness.
\item We also define $A^*_\ell$ to be the indicator function of $\aa^*_\ell$ (that is, the function defined by $A^*_\ell(n)=1$ if $n\in\aa^*_\ell$ and $A^*_\ell(n)=0$ if $n\notin\aa^*_\ell$).
\item Furthermore, we define $\ee^*_\ell = \aa^*_\ell \setminus \bigcup_{q^k\nparallel\ell} \ff_{q^k}$, so that a number in $\ee^*_\ell$ has {\em exactly} the properties~$\ff_{q^k}$ for all prime powers~$q^k\parallel\ell$ and no other properties.
\item Define $e^*(n)$ to be the unique number such that $n\in\ee^*_{e^*(n)}$, which exists for all $n$ by the infinite-intersection property and the disjoint property of the $\ff_{q^k}$; note that by disjointness, $n\in\ee^*_\ell$ if and only if $\ell$ is a unitary divisor of $e^*(n)$, that is, if and only if $\ell\mid e^*(n)$ and $(\ell,e^*(n)/\ell)=1$. Further, define $e^*_\xi(n) = \gcd(e^*(n),Q_\xi)$.
\item Finally, define
\[
N^*_\xi(x) = \sum_{\substack{n\le x \\ n\in \uu}} h(e^*_\xi(n))
\quad\text{and}\quad
P^*_\ell(x) = \#\{ n\le x \colon n\in \aa^*_\ell \} = \sum_{\substack{n\le x \\ n\in \uu}} A^*_\ell(n).
\]
\end{itemize}
The following lemma will allow us to estimate $N^*_\xi(x)$ in terms of these functions~$P^*_\ell(x)$ as well as the unitary inverse $g^*(m)$ of~$h(m)$.

\begin{lemma}\label{disjoint lemma}
Assume that the collection~$\{\ff_{q^k}\}$ is disjoint.
Let $h(m)$ be a multiplicative function bounded by~$1$ in absolute value. For any positive real numbers $\xi<y$,
\begin{equation*}
N^*_y(x) = \sum_{\ell\mid Q_\xi} g^*(\ell) P^*_\ell(x) + O\bigg( \sum_{\xi < q^k \le y} P^*_{q^k}(x) \bigg).
\end{equation*}
\end{lemma}

\begin{proof}
In the context of unitary Dirichlet convolution, the function $(-1)^{\omega(n)}$ is the analogue of the M\"obius function (see~\cite[Theorem 2.5]{Cohen}):
\begin{equation} \label{unitary}
g^*(m) = \sum_{\substack{d\mid m \\ (d,m/d)=1}} (-1)^{\omega(d)} h(m/d) \iff h(m) = \sum_{\substack{d\mid m \\ (d,m/d)=1}} g^*(d).
\end{equation}
By this unitary inversion formula,
\begin{equation} \label{apres unitary}
h(e^*_\xi(n)) = \sum_{\substack{\ell\mid e^*_\xi(n) \\ (\ell,e^*_\xi(n)/\ell)=1}} g^*(\ell) = \sum_{\substack{\ell\mid e^*(n) \\ (\ell,e^*_\xi(n)/\ell)=1 \\\ell\mid Q_\xi}} g^*(\ell) = \sum_{\substack{n\in\aa^*_\ell \\ \ell\mid Q_\xi}} g^*(\ell) = \sum_{\ell\mid Q_\xi} g^*(\ell) A^*_\ell(n);
\end{equation}
and then summing over $n\le x$ with $n\in\uu$ yields
\begin{equation} \label{apres summing disjoint}
N^*_\xi(x) = \sum_{\ell\mid Q_\xi} g^*(\ell) P^*_\ell(x).
\end{equation}

Next suppose that $\eta<\theta$ and that there is exactly one prime power~$q^k$ in the interval $(\eta,\theta]$. In this situation, $h(e^*_\theta(n)) = h(e^*_\eta(n))$ when $n\notin \aa^*_{q^k}$, and thus
\[
N^*_\theta(x) - N^*_\eta(x) = \sum_{\substack{n\le x \\ n\in \uu}} \Big( h(e^*_\theta(n)) - h(e^*_\eta(n)) \Big) = \sum_{\substack{n\le x \\ n\in \aa^*_{q^k}}} \Big( h(e^*_\theta(n)) - h(e^*_\eta(n)) \Big).
\]
Since the summand is at most~$2$ in absolute value, this identity gives
\[
N^*_\theta(x) - N^*_\eta(x) \ll \sum_{\substack{n\le x \\ n\in \aa^*_{q^k}}} 1 = P^*_{q^k}(x).
\]

If we now partition $(\xi,y]$ into intervals of the form $(\eta,\theta]$ that each contain exactly one of the prime powers~$q^k$ in $(\xi,y]$, summing over the resulting partition yields
\begin{equation*}
N_y(x) - N_\xi(x) \ll \sum_{\xi < q^k \le y} P^*_{q^k}(x),
\end{equation*}
which establishes the lemma when combined with the identity~\eqref{apres summing disjoint}.
\end{proof}

We remark that both formulas in Lemmas~\ref{cumulative lemma} and~\ref{disjoint lemma} involve finite sums of finite counting functions (for any given~$\xi$, $x$, and~$y$), and so no convergence arguments are needed.

\subsection{Our three applications} \label{three applications sec}

In Hooley's conditional proof of Artin's conjecture for~$a$, his method amounts to applying Lemma~\ref{cumulative lemma} in the case where $h(m)$ is the indicator function of $m=1$, the universe~$\uu$ is the set of primes not dividing the numerator or denominator of~$a$ (we write $\nu_p(a)=0$ for this condition), and
\begin{equation} \label{Fq def}
\ff_{q} = \{ p\in\uu \colon p\equiv1\mod q,\, a \text{ is a $q$th power}\mod p \}, \quad\text{while } \ff_{q^k} = \emptyset \text{ for } k\ge2.
\end{equation}
For our refinements of Artin's conjecture, we will use $h(m) = z^{\omega(m)}$ or $h(m) = z^{\Omega(m)}$, and these sets $\ff_q$ or variants thereof, in such a way that the functions $N_\xi(x)$ treated in Lemmas~\ref{cumulative lemma} and~\ref{disjoint lemma} become the left-hand sides of the asymptotic formulas in Theorems~\ref{theorem generating function omega of quotient h=1}, \ref{theorem generating function Omega of quotient h=1}, and~\ref{theorem generating function omega minus omega h=1}.

\begin{prop}\label{proposition omega of quotient}
Let $a\in\Q\setminus\{-1,0,1\}$ and let~$z$ be a complex number with $|z|\leq 1$. For any positive real numbers $\xi<x$,
\begin{multline*}
\sum_{\substack{p\le x \\ \nu_p(a)=0}} z^{\omega((p-1)/\ord_p(a))} = \sum_{\ell\mid Q_\xi} \mu^2(\ell) (z-1)^{\omega(\ell)} \#\{ p\le x \colon \nu_p(a)=0,\, \ell \mid (p-1)/\ord_p(a) \} \\
+ O\bigg( \sum_{\xi<q\le x} \#\{ p\le x \colon \nu_p(a)=0,\, q \mid (p-1)/\ord_p(a) \} \bigg).
\end{multline*}
\end{prop}

\begin{proof}
In preparation for applying Lemma~\ref{cumulative lemma}, we use the same universe $\uu = \{p\colon \nu_p(a)=0\}$ and subsets~\eqref{Fq def} as Hooley did (noting that these $\ff_{q^k}$ are indeed cumulative), but we use the strongly multiplicative function $h(m)$ with $h(q^k) = z$ for all prime powers~$p^k$. (In fact, since $\ff_{q^k} = \emptyset$ for $k\ge2$, the values of $h(q^k)$ on those proper prime powers are actually irrelevant.)
We calculate that $g(q) = z-1$ while $g(q^k)=0$ for $k\ge2$, and so $g(m) = \mu^2(m) (z-1)^{\omega(m)}$.

Crucially, since $p\in\ff_q$ if and only if $q\mid (p-1)/\ord_p(a)$, we see that $e(p)$ is the largest squarefree divisor of $(p-1)/\ord_p(a)$. Since $e_x(p) = e(p)$ for all $p\le x$, we deduce that $h(e_x(p)) = z^{\omega((p-1)/\ord_p(a))}$. Moreover, when~$\ell$ is squarefree we see that $p\in \aa_\ell$ if and only if $\ell \mid (p-1)/\ord_p(a)$, and thus
\[
P_\ell(x) = \#\{ p\le x \colon \nu_p(a)=0,\, \ell \mid (p-1)/\ord_p(a) \}
\quad\text{when $\ell$ is squarefree}
\]
and $P_\ell(x) = 0$ when $\ell$ is not squarefree.
The proposition now follows immediately from Lemma~\ref{cumulative lemma}.
\end{proof}

\begin{rmk}
In Proposition~\ref{proposition Omega of quotient}, taking the limit as $z\to0$ (or equivalently setting $z=0$ with the convention that $0^0=1$) recovers Hooley's formula
\begin{multline*}
\#\{ p\le x \colon \nu_p(a)=0,\, (p-1)/\ord_p(a)=1 \} = \sum_{\ell\mid Q_\xi} \mu^2(\ell) (-1)^{\omega(\ell)} \#\{ p\le x \colon \nu_p(a)=0,\, \ell \mid (p-1)/\ord_p(a) \} \\
+ O\bigg( \sum_{\xi<q\le x} \#\{ p\le x \colon \nu_p(a)=0,\, q \mid (p-1)/\ord_p(a) \} \bigg).
\end{multline*}
\end{rmk}

\begin{prop}\label{proposition Omega of quotient}
Let $a\in\Q\setminus\{-1,0,1\}$ and let~$z$ be a complex number with $|z|\leq 1$. For any positive real numbers $\xi<x$,
\begin{multline*}
\sum_{\substack{p\le x \\ \nu_p(a)=0}} z^{\Omega((p-1)/\ord_p(a))} = \sum_{\ell\mid Q_\xi} z^{\Omega(\ell)} (1-z^{-1})^{\omega(\ell)} \#\{ p\le x \colon \nu_p(a)=0,\, \ell \mid (p-1)/\ord_p(a) \} \\
+ O\bigg( \sum_{\xi<q^k\le x} \#\{ p\le x \colon \nu_p(a)=0,\, q^k \mid (p-1)/\ord_p(a) \} \bigg).
\end{multline*}
\end{prop}

\begin{proof}
In preparation once again for applying Lemma~\ref{cumulative lemma}, we use the same universe $\uu = \{p\colon \nu_p(a)=0\}$ but the new sets
\begin{equation*}
\ff_{q^k} = \{ p\in\uu \colon p\equiv1\mod{q^k},\, a \text{ is a $q^k$th power}\mod p \},
\end{equation*}
which we note are again cumulative. We also change to the completely multiplicative function $h(m)$ with $h(q^k) = z^k$, so that $g(q^k) = z^k-z^{k-1} = z^k (1-z^{-1})$ and therefore $g(m) = z^{\Omega(m)} (1-z^{-1})^{\omega(m)}$.

In this case, we see that $e(p) = e_x(p)$ is exactly $(p-1)/\ord_p(a)$, and therefore $h(e_x(p)) = z^{\Omega((p-1)/\ord_p(a))}$. Moreover, we see that
\[
P_\ell(x) = \#\{ p\le x \colon \nu_p(a)=0,\, \ell \mid (p-1)/\ord_p(a) \}
\]
now regardless of whether $\ell$ is squarefree. The proposition now follows immediately from Lemma~\ref{cumulative lemma}.
\end{proof}

In the proof of the next proposition (as well as later in this paper), we use the notation $\id(S)$ to denote the indicator function of the statement~$S$, so that $\id(S)=1$ if~$S$ is true and $\id(S)=0$ if~$S$ is false.

\begin{prop}\label{proposition difference of omega}
Let $a\in\Q\setminus\{-1,0,1\}$, $z\in\C$ with $|z|\leq 1$, $x$ and $\xi$ positive real numbers with $\xi<x$, then
\begin{align*}
\sum_{\substack{p\le x \\ \nu_p(a)=0}} z^{\omega(p-1) - \omega(\ord_p(a))} &= \sum_{\ell\mid Q_\xi} (z-1)^{\omega(\ell)}\sum\limits_{m\mid \ell}\mu(m) \#\Big\{ p\le x \colon \nu_p(a)=0,\, \ell\mid (p-1)/\ord_p(a),\ m\ell\mid (p-1)\Big\}\\
&+ O\bigg( \sum_{\xi<q^k\le x} \#\Big\{ p\le x \colon \nu_p(a)=0,\, q^k \mid (p-1)/\ord_p(a),\, q^k\parallel (p-1) \Big\} \bigg).
\end{align*}
\end{prop}

\begin{proof}
We now prepare to apply Lemma~\ref{disjoint lemma}, using the same universe $\uu = \{p\colon \nu_p(a)=0\}$ as before but now defining
\begin{equation*}
\ff_{q^k} = \{ p\in\uu \colon q^k\parallel (p-1),\, a \text{ is a $q^k$th power}\mod p \},
\end{equation*}
which we note are disjoint. We revert to the strongly multiplicative function $h(m)$ with $h(q^k) = z$ for all $k\ge1$, and we calculate that $g^*(q^k) = z-1$ and thus $g^*(m) = (z-1)^{\omega(m)}$.

Crucially, we see that $e^*(p) = e^*_x(p)$ is the largest divisor of $(p-1)/\ord_p(a)$ that is a unitary divisor of $p-1$. It follows that $\omega(e^*_x(p)) = \omega(p-1) - \omega(\ord_p(a))$ and therefore $h(e^*_x(p)) = z^{\omega(p-1) - \omega(\ord_p(a))}$. Furthermore,
\[
P^*_\ell(x) = \#\Big\{ p\le x \colon \nu_p(a)=0,\, \ell \mid (p-1)/\ord_p(a),\, (\ell, (p-1)/\ell) = 1\Big\}.
\]
However, when $\ell\mid(p-1)$, note that
\[
\id\Bigl( \bigl( \ell,(p-1)/\ell \bigr) = 1 \Bigr) = \sum_{m\mid ( \ell,(p-1)/\ell )} \mu(m) = \sum_{m\mid\ell} \mu(m) \id\Bigl( m\ell\mid(p-1) \Bigr);
\]
consequently,
\[
P^*_\ell(x) = \sum_{m\mid\ell} \mu(m) \#\Big\{ p\le x \colon \nu_p(a)=0,\, \ell \mid (p-1)/\ord_p(a),\, m\ell\mid(p-1) \Big\}.
\]
The proposition now follows immediately from Lemma~\ref{disjoint lemma}, once we note (for the error term) that a prime power~$q^k$ dividing $p-1$ satisfies $(q^k, (p-1)/q^k)$ if and only if $q^k\parallel(p-1)$.
\end{proof}

\section{The connection to cyclotomic-Kummer extensions} \label{degree section}

In light of the propositions in the previous section, it is important to study the quantity
\begin{equation} \label{palm counts}
\Palm(x)=\#\Big\{ p\le x \colon \nu_p(a)=0,\, \ell\mid (p-1)/\ord_p(a),\ m\ell\mid (p-1)\Big\}.
\end{equation}
This is the quantity appearing in the main term of Proposition~\ref{proposition difference of omega}; moreover, the special case $m=1$ is also the quantity appearing in the main terms of Propositions~\ref{proposition omega of quotient} and~\ref{proposition Omega of quotient}). We immediately rewrite this quantity as
\begin{equation} \label{palm fields}
\Palm(x)=\#\Big\{ p\le x \colon \nu_p(a)=0,\, \text{$p$ splits completely in $\Q\bigl( \sqrt[\ell]{a},\zeta_{m\ell} \bigr)$} )\Big\}.
\end{equation}
This approach is the one used by Hooley~\cite{Hooley}, although he needed only the special case where~$\ell$ is squarefree and $m=1$. The fact that these two quantities are identical follows from~\cite[Proposition 8]{Peringuey}; in that work this fact is stated under the assumption that~$\ell$ is squarefree (since that was the only case required for the application), but the proof therein actually works for arbitrary~$\ell$.

The expression~\eqref{palm fields} for $\Palm(x)$ can be estimated using the Chebotarev density theorem (indeed, in this particular case of splitting completely, it can be derived directly from the prime ideal theorem). We take exactly this approach in Section~\ref{palm asymptotic section}; but first we need to understand the degrees (and bound the discriminants) of the cyclotomic-Kummer extensions $\Q\bigl( \sqrt[\ell]{a},\zeta_{m\ell} \bigr)$; we address these two requirements in Sections~\ref{degree section} and~\ref{discriminant section}, respectively.

\subsection{Degrees of cyclotomic-Kummer extensions} \label{degree section}

In the case $m=1$, the degree of the pure Kummer extension $\Q\bigl( \sqrt[\ell]{a},\zeta_{\ell} \bigr)$ was computed by Wagstaff (\cite[Propostion 4.1]{Wagstaff}). For the cyclotomic-Kummer extensions $\Q\bigl( \sqrt[\ell]{a},\zeta_{m\ell} \bigr)$ with~$\ell$ squarefree, the degree was computed by the third author~\cite{Peringuey}. In this section we carry out the analogous calculation for arbitrary~$\ell$ and $m$. To do so, we first need to generalize a result of Schinzel~\cite[Lemma~4]{Schinzel}.

\begin{lemma}\label{lem_schinzel}
Fix positive integers~$m$ and~$\schinzeln$.
A rational number~$\beta$ is of the form $\gamma^\schinzeln$ for some $\gamma\in\Q(\zeta_\schinzelm)$ if and only if one of the following sets of conditions is satisfied:
 \begin{itemize}
 \item $\schinzeln\equiv 1\mod{2}$ and $\beta=c^\schinzeln$ for some $c\in\Q$.
 \item$\schinzeln\equiv 0\mod{2}$, $\beta>0$, and $\beta=c^{\schinzeln/2}$ for some $c\in\Q$ such that $\sqrt{c}\in\Q(\zeta_\schinzelm)$.
 \item$\schinzeln\equiv 0\mod{2}$, $\beta<0$, $2\mid m$, and $\beta=-c^{\schinzeln/2}$ for some $c\in\Q$ such that $\sqrt{c}\in\Q(\zeta_\schinzelm)$.
 \item$\schinzeln\equiv 2\mod{4}$, $\beta<0$, $2\nmid m$, and $\beta=c^{\schinzeln/2}$ for some $c\in\Q$ such that $\sqrt{c}\in\Q(\zeta_\schinzelm)$.
 \item$\schinzeln\equiv 4\mod{8}$, $\beta<0$, $2\nmid m$, and $\beta=-(2c)^{\schinzeln/2}$ for some $c\in\Q$ such that $\sqrt{c}\in\Q(\zeta_\schinzelm)$.
 \end{itemize}
\end{lemma}

\begin{proof}
It is easy to check that each set of conditions does imply that~$\beta$ has the required form. For the converse, we suppose that~$\beta$ is indeed of the form $\gamma^\schinzeln$ with $\gamma\in\Q(\zeta_\schinzelm)$. We write $\beta=\sgn(\beta) \beta_1^t$ where~$\beta_1$ a positive rational number that is not a perfect power.

First suppose that $\sgn(\beta)=1$. Then $\beta_1^{t/\schinzeln}\in\Q(\zeta_\schinzelm)$ and thus $\Q(\beta_1^{t/\schinzeln})$ is normal over~$\Q$. Furthermore, the polynomial
$X^{\schinzeln/(\schinzeln,t)}-\beta_1^{t/(\schinzeln,t)}$ is irreducible over~$\Q$ and has $\beta_1^{t/\schinzeln}$ as one of its roots, and therefore all its $\schinzeln/(\schinzeln,t)$ roots (which are equally spaced around a circle in~$\C$) must lie in $\Q(\beta_1^{t/\schinzeln})$. But $\Q(\beta_1^{t/\schinzeln})$ is a real subfield of~$\Q(\zeta_\schinzelm)$, which forces $X^{\schinzeln/(\schinzeln,t)}-\beta_1^{t/(\schinzeln,t)}$ to have at most two roots; it follows that $\schinzeln/(\schinzeln,t) \le 2$, which implies that $\schinzeln\mid2t$.
\begin{itemize}
\item If $\schinzeln\equiv 1\mod{2}$, then $\schinzeln\mid t$ and so $c=\beta_1^{t/\schinzeln}\in\Q$ satisfies $\beta=c^\schinzeln$.
\item If $\schinzeln\equiv 0\mod{2}$, then $c=\beta_1^{2t/\schinzeln}\in\Q$ satisfies $\beta=c^{\schinzeln/2}$ and $\sqrt{c} = \beta_1^{t/\schinzeln}\in\Q(\zeta_\schinzelm)$.
\end{itemize}
Now suppose that $\sgn(\beta)=-1$. Then $\zeta_{2\schinzeln}\beta_1^{t/\schinzeln}\in \Q(\zeta_\schinzelm)$ where~$\zeta_{2\schinzeln}$ is a primitive $2\schinzeln$th root of unity.
\begin{itemize}
\item If $\schinzeln\equiv 1\mod{2}$, then $\zeta_{2\schinzeln}\in\Q(\zeta_{\schinzeln})\subset \Q(\zeta_{\schinzelm})$, so $\beta_1^{t/\schinzeln}\in \Q(\zeta_{\schinzelm})$ and hence (as in the first case) $c=\beta_1^{t/\schinzeln}\in\Q$ satisfies $\beta=c^\schinzeln$.
\item
If $\schinzeln\equiv 0\mod{2}$ and $2\mid m$, then $\zeta_{2\schinzeln}\in \Q(\zeta_{\schinzelm})$, so $\beta_1^{t/\schinzeln}\in \Q(\zeta_{\schinzelm})$ and hence (as in the first case) $c=\beta_1^{2t/\schinzeln}\in\Q$ satisfies $\beta=c^{\schinzeln/2}$ and $\sqrt{c} = \beta_1^{t/\schinzeln}\in\Q(\zeta_\schinzelm)$.
\end{itemize}
The remaining case is where $\sgn(\beta)=-1$ and $\schinzeln\equiv 0\mod{2}$ but $2\nmid m$.
Then $\zeta_{2\schinzeln}\notin \Q(\zeta_{\schinzelm})$ so $\beta_1^{t/\schinzeln}\notin \Q(\zeta_{\schinzelm})$. On the other hand. $\beta_1^{t/\schinzeln}\in \Q(\zeta_{2\schinzelm})$, so $\schinzeln\mid2t$ as in the first case; but since $\schinzeln\mid t$ would imply $\beta_1^{t/\schinzeln}\in\Q\subset\Q(\zeta_{\schinzelm})$, we conclude that $t\equiv \frac{\schinzeln}{2}\mod{\schinzeln}$. It follows that $\sqrt{\beta_1}\in \Q(\zeta_{2\schinzelm})$ and $\sqrt{\beta_1}\notin \Q(\zeta_{\schinzelm})$.

Writing $s(\beta_1)$ for the squarefree kernel of~$\beta_1$, we deduce from~\cite[Lemma~3]{Schinzel} that $s(\beta_1)\mid \schinzelm$ and that either $\schinzelm\equiv2\mod{4}$ and $s(\beta_1)\equiv 3\mod{4}$, or $\schinzelm\equiv4\mod{8}$ and $2\mid s(\beta_1)$. By the same lemma:
\begin{itemize}
\item when $\schinzelm\equiv2\mod{4}$ and $s(\beta_1)\equiv 3\mod{4}$, we also have $\schinzeln\equiv2\mod{4}$, and $c=-\beta_1^{2t/\schinzeln}\in\Q$ satisfies $\beta=c^{\schinzeln/2}$ and $\sqrt{c} \in\Q(\zeta_\schinzelm)$;
\item when $\schinzelm\equiv4\mod{8}$ and $2\mid s(\beta_1)$, we also have $\schinzeln\equiv4\mod{8}$, and $c=-(\beta_1)^{2t/\schinzeln}/2\in\Q$ satisfies $\beta=-(2c)^{\schinzeln/2}$ and $\sqrt{c} \in\Q(\zeta_\schinzelm)$.
\qedhere
\end{itemize}
\end{proof}

We are now able to generalize Wagstaff's argument. Recall the notation from Definition~\ref{s and d def} and Notation~\ref{notation beginning}.

\begin{prop}\label{proposition degree}
Let $a\in\Q\setminus\{-1,0,1\}$. Let~$\ell$ and~$m$ be positive integers, and set $\ell'=\ell/(\ell,h)$. Let $\zeta_{m\ell}$ be a primitive $(m\ell)$th root of unity. 
Then $\displaystyle \left[\Q\Big(\sqrt[\ell]{a},\zeta_{m\ell}\Big):\Q\right]=\frac{\ell'\varphi(m\ell)}{\varepsilon_a(m\ell,\ell)}$,
where $\varepsilon_a(m\ell,\ell)$ is defined as follows:
 \begin{itemize}
 \item If $a>0$, then $\varepsilon_a(m\ell,\ell)=
 \begin{cases}
 2,&\text{if } 2\mid\ell'\text{ and }\mathfrak{d}(a_0)\mid m\ell,\\
1,&\text{otherwise.}
 \end{cases}$
 \item If $a<0$ and $\ell$ is odd, then $\varepsilon_a(m\ell,\ell)=1$.
 \item If $a<0$ and $\ell$ is even but $\ell'$ is odd, then $\varepsilon_a(m\ell,\ell)= 
 \begin{cases}
 1,&\text{if } 2\mid m,\\
1/2,&\text{otherwise.}
 \end{cases}$
 \item If $a<0$ and $\ell'\equiv 2\mod{4}$, then $\varepsilon_a(m\ell,\ell)=
 \begin{cases}
 2,&\text{if } 2\mid m \text{ and } \mathfrak{d}(a_0)\mid m\ell,\\
 2,&\text{if } 2\nmid m \text{ and } \ell\equiv 2\mod{4}\text{ and } \mathfrak{d}(-a_0)\mid m\ell,\\
 2,&\text{if } 2\nmid m \text{ and } \ell\equiv 4\mod{8}\text{ and } \mathfrak{d}(2a_0)\mid m\ell,\\
1,&\text{otherwise.}
 \end{cases}$
 \item If $a<0$ and $4\mid \ell'$, then $\varepsilon_a(m\ell,\ell)=
 \begin{cases}
 2,&\text{if }\mathfrak{d}(a_0)\mid m\ell,\\
1,&\text{otherwise.}
 \end{cases}$ 
 \end{itemize}
\end{prop}

\begin{proof}

Set $L=\Q\Big(\zeta_{m\ell}\Big)$.
We start by observing that
\begin{align*}
\left[\Q\Big(\sqrt[\ell]{a},\zeta_{m\ell}\Big):\Q\right] &= \left[\Q\Big(\sqrt[\ell]{a},\zeta_{m\ell}\Big):L\right] \times \left[L:\Q\right] = \left[\Q\Big(\sqrt[\ell]{a},\zeta_{m\ell}\Big):L\right] \varphi({m\ell}).
\end{align*}
Furthermore, if we set $n = \min\Bigl\{j\in\N \colon a^{j/\ell}\in L \Bigr\}$, we claim that $\left[\Q\Big(\sqrt[\ell]{a},\zeta_{m\ell}\Big):L\right] = n$.
To see this, note that $n\mid \ell$ and that~$a$ has order~$n$ in $L^* / \left(L^*\right)^n$, and thus $[L(\sqrt[n]a):L]=n$ by~\cite[Lemma~2.1]{Birch}. But now $[L(\sqrt[\ell]{a}):L]\le n$ since $(\sqrt[\ell]{a})^n \in L$, and also $L(\sqrt[n]{a})\subset L(\sqrt[\ell]{a})$, and therefore $[L(\sqrt[\ell]{a}):L]=n$ as claimed.

It remains to show that~$n$, which is the smallest positive integer such that $a^n=\gamma^\ell$ for some $\gamma\in L$, equals $\ell'/\varepsilon_a(m\ell,\ell)$. To do so we determine the smallest even and smallest odd such numbers~$n$, which we denote by~$n_e$ and~$n_o$, respectively; then $n=n_e$ if~$n_o$ is not defined, while $n = \min\{n_e,n_o\}$ otherwise.
We use Lemma~\ref{lem_schinzel} to carry out these computations. We do the computations completely in the case $a<0$; the case $a>0$ is simpler and analogous.

We first note that $a^{n_e}>0$ since $n_e$ is even. If~$\ell$ is odd then $n_e=2\ell'$. Suppose that~$\ell$ is even, so that $\ell/(\ell,2h)\mid n_e$ and $\sqrt{a_0^{2hn_e/\ell}}\in L$ (by Lemma~\ref{lem_schinzel}).
\begin{itemize}
\item If $2\mid 2hn_e/\ell$, then $n_e=\ell'$ if $\ell'$ is even and $n_e=2\ell'$ if $\ell'$ is odd.
\item If $2\nmid2hn_e/\ell$, then $\sqrt{a_0^{2hn_e/\ell}}\in L$ if and only if $\nu_2(\ell)>\nu_2(h)+1$ and $\mathfrak{d}(a_0)\mid {m\ell}$.
\end{itemize}
We conclude that
\[
n_e=
\begin{cases}
 \ell'/2, & \text{if } \nu_2(\ell)>\nu_2(h)+1 \text{ and } \mathfrak{d}(a_0)\mid {m\ell},\\ 
2\ell', & \text{if } \nu_2(\ell)\leq \nu_2(h),\\
\ell', & \text{otherwise}.
\end{cases}
\]

Now we note that $a^{n_o}<0$ since~$n_o$ is odd. If~$\ell$ is odd then $n_e=\ell'$. Suppose that~$\ell$ is even. If $2\mid m$, then $\ell/(\ell,2h)\mid n_o$ and therefore $\nu_2(\ell)\leq\nu_2(h)+1$ (or else $n_o$ would be even). In other words,
\[
n_o=
\begin{cases}
 \ell', & \text{if } \nu_2(\ell)\leq\nu_2(h),\\
 \ell'/2, & \text{if } \nu_2(\ell)=\nu_2(h)+1 \text{ and } \mathfrak{d}(a_0)\mid {m\ell}.
\end{cases}
\]
If $\ell\equiv 2\mod{4}$ and $2\nmid m$, then
\[
n_o=
 \ell'/2 \text{if } \nu_2(\ell)=\nu_2(h)+1 \text{ and } \mathfrak{d}(-a_0)\mid {m\ell}
\]
and otherwise~$n_0$ is undefined.
Finally, if $\ell\equiv 4\mod{8}$ and $2\nmid m$, then
\[
n_o= \ell'/2 \text{if } \nu_2(\ell)=\nu_2(h)+1 \text{ and } \mathfrak{d}(2a_0)\mid {m\ell}
\]
and otherwise~$n_0$ is undefined.

In all cases we see that $n=\min\{n_e,n_o\}$ (or $n=n_e$ if~$n_0$ is undefined) matches the value of $\ell'/\varepsilon_a(m\ell,\ell)$ that arises from the definition of $\varepsilon_a(m\ell,\ell)$, which completes the proof of the proposition.
\end{proof}

\subsection{Bounding discriminants of cyclotomic-Kummer extensions} \label{discriminant section}
 
Using the formula for the degree computed in Proposition~\ref{proposition degree}, we are now able to provide the following upper bound for the associated discriminant.
 
\begin{prop}\label{propdiscriminant}
Let $m$ and $\schinzeln$ be integers and let $a$ be a nonzero rational number. Let $K=\Q(\sqrt[\schinzeln]{a},\zeta_\schinzelm)$, where $\zeta_\schinzelm$ is a primitive $\schinzelm$th root of unity. Then the discriminant $\Delta_K$ of~$K$ satisfies
\[
\log |\Delta_K| \ll_a [K:\Q]\log \schinzelm.
\]
\end{prop}

\begin{proof}
We know that $[K:\Q]=\varphi(\schinzelm)\schinzeln'$ with $\schinzeln' = [K:\Q(\zeta_\schinzelm)]$. Let $a'$ be the unique $\schinzeln'$-free integer such that $a/a'$ is a perfect $\schinzeln'$th power, and set $\kappa=(a')^{1/\schinzeln'}$.
Then 
\[
\{\kappa^d \zeta_\schinzelm^g\colon 0\leq d<\schinzeln',\, 0\leq g< \schinzelm,\, (g,\schinzelm)=1\}
\]
forms a basis of $K$ over $\Q$.
Moreover, any $\sigma\in\Gal(K;\Q)$ is determined by $\sigma(\kappa)$ and $\sigma(\zeta_\schinzelm)$:
\begin{itemize}
\item $\sigma(\kappa)$ is a root of $X^{\schinzeln'}-a'$, so $\sigma(\kappa)=\kappa\zeta_\schinzelm^{\lambda \schinzelm/\schinzeln'}$ with $0\leq \lambda< \schinzeln'$;
\item $\sigma(\zeta_\schinzelm)$ is a primitive $\schinzelm$th root of unity, so $\sigma(\zeta_\schinzelm)=\zeta_\schinzelm^\mu$ with $0\leq \mu < \schinzelm$ and $(\mu,\schinzelm)=1$.
\end{itemize}
Then~\cite[Proposition 2.6.3]{Samuel} tells us that
\begin{equation} \label{big matrix}
\Delta_K=\det \biggl( (\Tr_{K/\Q}(\kappa^{d_1}\zeta_\schinzelm^{g_1}\kappa^{d_2}\zeta_\schinzelm^{g_2}))_{\substack{0\leq d_1,d_2 < \schinzeln'\\ 0\leq g_1,g_2<\schinzelm\\ (g_1g_2,\schinzelm)=1}} \biggr).
\end{equation}
For given $\delta$ and $\gamma$ with $(\gamma,\schinzelm)=1$, we compute
\begin{align*}
\Tr(\kappa^\delta\zeta_\schinzelm^\gamma)&=\sum\limits_{0\leq\lambda <\schinzeln'}\sum\limits_{\substack{0\leq \mu<\schinzelm\\ (\mu,\schinzelm)=1}}(\kappa\zeta_\schinzelm^{\lambda \schinzelm/\schinzeln'})^\delta(\zeta_\schinzelm^\mu)^\gamma\\
&=\sum\limits_{0\leq\lambda <\schinzeln'}\kappa^\delta\zeta_\schinzelm^{\lambda\delta \schinzelm/\schinzeln'}\sum\limits_{\substack{0\leq \mu<\schinzelm\\ (\mu,\schinzelm)=1}}\zeta_\schinzelm^{\mu\gamma} =\begin{cases}
\schinzeln'\kappa^\delta\sum\limits_{\substack{0\leq \mu<\schinzelm\\ (\mu,\schinzelm)=1}}\zeta_\schinzelm^{\mu\gamma}, & \text{if}\ \schinzeln'\mid \delta,\\
0, & \text{otherwise}.
\end{cases}
\end{align*}
It follows that if we define
\[
T = \Bigl( \Tr_{\Q(\zeta_\schinzelm)/\Q}(\zeta_\schinzelm^{g_1}\zeta_\schinzelm^{g_2}) \Bigr)_{\substack{0\leq g_1,g_2<\schinzelm\\ (g_1g_2,\schinzelm)=1}},
\]
then the matrix in equation~\eqref{big matrix}
is block diagonal (after reordering its rows and columns) with one block equal to $\schinzeln'T$ and $\schinzeln'-1$ blocks equal to $a'\schinzeln'T$.
Consequently,
\begin{align*}
|\Delta_K|&=\Bigl(|a'|^{\schinzeln'-1}\left(\schinzeln'\right)^{\schinzeln'}\Bigr)^{\varphi(\schinzelm)}(\det T)^{\schinzeln'}\leq|a'|^{\varphi(\schinzelm)\schinzeln'}\left(\schinzeln'\right)^{\varphi(\schinzelm)\schinzeln'}\schinzelm^{\varphi(\schinzelm)\schinzeln'}
\leq|a'|^{[K:\Q]}\schinzelm^{2[K:\Q]}
\end{align*}
since $\schinzeln'\mid \schinzeln\mid \schinzelm$. We conclude that
\[
\log |\Delta_K|\leq 2[K:\Q]\log \schinzelm + [K:\Q]\log|a'| \ll_a [K:\Q]\log \schinzelm
\]
as desired.
\end{proof}

\subsection{Asymptotic behavior of $\Palm(x)$} \label{palm asymptotic section}

We now return to the estimation of $\Palm(x)$ from equation~\eqref{palm fields}.

\begin{thm} \label{thm_asymptoticprimesatisfyingstuff}
Assume GRH. Let $a\in\Q\setminus\{-1,0,1\}$, and let~$\ell$ and~$m$ be positive integers. Then
\[
\Palm(x) = \frac{\pi(x)}{\Bigl[ \Q(\sqrt[\ell]{a},\zeta_{m\ell}):\Q \Bigr]} + \bo_a\Bigl( \sqrt{x}\log(x\ell) \Bigr) = \frac{\varepsilon_a(m\ell,\ell)\cdot(\ell,h)}{\ell\varphi(m\ell)} \pi(x) + \bo_a\Bigl( \sqrt{x}\log(x\ell) \Bigr),
\]
where~$\varepsilon_a(m\ell,\ell)$ and~$h$ were defined in Proposition~\ref{proposition degree} and Notation~\ref{notation beginning}, respectively.
\end{thm}

\begin{proof}
Lang~\cite{Lang} established, assuming GRH, the following effective prime ideal theorem for a number field~$K$ (a special case of a Chebotarev density theorem with effective error term): uniformly in~$K$,
\begin{equation} \label{equation_googd_GRH}
\pi_K(x)=\pi(x)+\bo\left(\sqrt{x}\log(x^{[K:\Q]}|\Delta_K|)\right),
\end{equation}
where $\pi_K(x) = \#\{ \mathfrak p\colon N(\mathfrak p) \le x\}$ is the prime ideal counting function for~$K$ (and $\pi(x) = \pi_\Q(x)$ is the ordinary prime counting function) and $\Delta_K$ is the discriminant of~$K$. If we take $K = \Q(\sqrt[\ell]{a},\zeta_{m\ell})$, then Proposition~\ref{propdiscriminant} implies that
\begin{equation} \label{equation_googd_GRH2}
\pi_{\Q(\sqrt[\ell]{a},\zeta_{m\ell})}(x) = \pi(x)+\bo \Bigl( \Bigl[ \Q(\sqrt[\ell]{a},\zeta_{m\ell}):\Q \Bigr] \sqrt{x}\log(x\ell) \Bigr).
\end{equation}
Finally, $\Palm(x)$ is almost exactly $\pi_{\Q(\sqrt[\ell]{a},\zeta_{m\ell})}(x) / \Bigl[ \Q(\sqrt[\ell]{a},\zeta_{m\ell}):\Q \Bigr]$; the precise relationship, with a suitably bounded error term that is majorized by the error term above, works here exactly the same way as in Hooley's work~\cite[Section 4, equation~(17)]{Hooley}. We thus deduce that
\[
\Palm(x) = \frac{\pi(x)}{\Bigl[ \Q(\sqrt[\ell]{a},\zeta_{m\ell}):\Q \Bigr]} + \bo_a\Bigl( \sqrt{x}\log(x\ell) \Bigr),
\]
and the second equality follows from Proposition~\ref{proposition degree}.
\end{proof}

\section{Dealing with error terms} \label{error section}

We first deal with the error terms in Propositions~\ref{proposition omega of quotient}, \ref{proposition Omega of quotient}, and~\ref{proposition difference of omega}; we may concentrate on the error term in Proposition~\ref{proposition Omega of quotient}, as it is an upper bound for the error terms in the other two propositions. Our goal is to lower~$\xi$ enough so that summing the error term of Theorem~\ref{thm_asymptoticprimesatisfyingstuff} is admissible. In order to control this error term, we use the same techniques as used by Hooley~\cite{Hooley}.

First we will see how low we can get $\xi$ without using GRH. Using a combinatorial argument and the Brun--Titchmarsh theorem we obtain the following proposition.

\begin{prop}\label{proposition unconditional small xi}
Let $a\in\Q\setminus\{-1,0,1\}$.
For all $B\ge0$,
\[
\sum_{\frac{\sqrt{x}}{(\log x)^B}<q^k\le x} \#\{ p\le x \colon \nu_p(a)=0,\, q^k \mid (p-1)/\ord_p(a) \}\ll_{a,B} \frac{x\log\log x}{\log^2 x}.
\]
\end{prop}

\begin{proof}
For the largest prime powers, we see by writing $p-1=q^km$ that
\begin{align*}
\sum_{\sqrt{x}\log x<q^k\le x} & \#\{ p\le x \colon \nu_p(a)=0,\, q^k \mid (p-1)/\ord_p(a) \} \\
&= \sum_{\sqrt{x}\log x<q^k\le x} \#\{ p\le x \colon \nu_p(a)=0,\, q^k \mid (p-1),\, \ord_p(a) \mid (p-1)/q^k \} \\
&= \sum_{\sqrt{x}\log x<q^k\le x} \#\{ p\le x \colon \nu_p(a)=0,\, q^k \mid (p-1),\, a^{(p-1)/q^k} \equiv1\mod p \} \\
&\le \sum_{m<\sqrt{x}/\log x} \#\{ p\le x \colon \nu_p(a)=0,\, a^m \equiv1\mod p \} \\
&\le \Omega \bigg( \prod_{m<\sqrt{x}/\log x} (a^m-1) \bigg)
\le \frac1{\log 2} \log \bigg( \prod_{m\le \sqrt{x}/\log x} (a^m-1) \bigg) < \frac{\log a}{\log 2} \sum_{m<\sqrt{x}/\log x} m \ll_a \frac{x}{\log^2 x}.
\end{align*}
This is entirely correct if $a$ is a positive integer; if $a=n/d$, then we should replace $a^m-1$ with $n^m-d^m$, but this only affects the implicit constant in the final bound.

For medium-sized prime powers,
\begin{align*}
\sum_{\frac{\sqrt{x}}{(\log x)^B}<q^k\le \sqrt{x}\log x} \#\{ p\le x \colon \nu_p(a)=0,\, q^k \mid (p-1)/\ord_p(a) \}
&\le \sum_{\frac{\sqrt{x}}{(\log x)^B}<q^k\le \sqrt{x}\log x} \#\{ p\le x \colon q^k \mid (p-1) \} \\
&= \sum_{\frac{\sqrt{x}}{(\log x)^B}<q^k\le \sqrt{x}\log x} \pi(x;q^k,1).
\end{align*}
Applying Brun--Titchmarsh, we obtain
\begin{align*}
\sum_{\frac{\sqrt{x}}{(\log x)^B}<q^k\le \sqrt{x}\log x} \#\{ p\le x \colon \nu_p(a)=0,\, q^k \mid (p-1)/\ord_p(a) \}
&\ll \sum_{\frac{\sqrt{x}}{(\log x)^B}<q^k\le \sqrt{x}\log x} \frac x{\varphi(q^k)\log(x/q^k)} \\
&\ll \frac x{\log x} \sum_{\frac{\sqrt{x}}{(\log x)^B}<q^k\le \sqrt{x}\log x} \frac 1{q^k} \\
&\ll \frac x{\log^2 x} \sum_{\frac{\sqrt{x}}{(\log x)^B}<q^k\le \sqrt{x}\log x} \frac{\log(q^k)}{q^k} \\
&\ll_B \frac x{\log^2 x} \bigl( \log\log x + O(1) \bigr)
\end{align*}
by Mertens's formula.
\end{proof}

Assuming GRH we can get $\xi$ even lower (exactly as in~\cite{Hooley}).

\begin{prop}\label{proposition smaller xi under GRH}
Assume GRH. Let $a\in\Q\setminus\{-1,0,1\}$ and $B\ge2$. For all $\xi\ll(\log x)^{B-1}$,
\[
\sum\limits_{\xi\leq q^k<\frac{\sqrt{x}}{\left(\log x\right)^B}}\#\left(p\leq x\colon \nu_p(a)=0,\, q^k\mid (p-1)/\ord_p(a)\right) \ll_B \frac{x}{\xi\log x}.
\]
\end{prop}

\begin{proof}
We apply Theorem~\ref{thm_asymptoticprimesatisfyingstuff} with $\ell=q^k$ and $m=1$ to obtain
\begin{align*}
\sum_{\xi\leq q^k<\frac{\sqrt{x}}{\left(\log x\right)^B}}\#\left(p\leq x\colon \nu_p(a)=0,\, q^k\mid (p-1)/\ord_p(a)\right) & =\sum_{\xi\leq q^k<\frac{\sqrt{x}}{\left(\log x\right)^B}}\left(\frac{\pi(x)}{q^{2k-1}(q-1)}+\bo\bigl(\sqrt x\log x\bigr)\right) \\
&\ll \frac{x}{\log x}\sum\limits_{\xi\leq q^k<\frac{\sqrt{x}}{\left(\log x\right)^A} }\frac{1}{q^{2k}} + \frac{x}{(\log x)^B} \\
&< \frac{x}{\log x}\sum\limits_{n>\xi}\frac{1}{n^2} + \frac{x}{(\log x)^B} \ll \frac{x}{\xi\log x} + \frac{x}{(\log x)^B},
\end{align*}
and the first error term dominates by the assumption on~$\xi$.
\end{proof}

When we incorporate Theorem~\ref{thm_asymptoticprimesatisfyingstuff} into Propositions~\ref{proposition omega of quotient}, \ref{proposition Omega of quotient}, and~\ref{proposition difference of omega}, a new error term will appear from summing the error term from Theorem~\ref{thm_asymptoticprimesatisfyingstuff} over divisors of~$Q_\xi$, which we recall is the least common multiple of all the integers up to~$\xi$. The following proposition bounds these errors.

\begin{prop}\label{proposition error sum chebotarev error term combined}
Let $x\ge3$ and $A>0$ be real numbers. Uniformly for $\xi\le \frac12\log x$ and $|z|\leq A$,
\begin{align*}
\sum\limits_{\ell\mid Q_\xi}\mu^2(\ell)|z-1|^{\omega(\ell)}\sqrt{x}\log(x\ell) &\le \bigl(\sqrt x\log x\bigr)(2+A)^\xi \\
\sum\limits_{\ell\mid Q_\xi}|z|^{\Omega(\ell)}|1-z^{-1}|^{\omega(\ell)}\sqrt{x}\log(x\ell) &\le \bigl(\sqrt x\log x\bigr) (2+A)^{4\xi} \\
\sum\limits_{\ell\mid Q_\xi}|z-1|^{\omega(\ell)}\sum\limits_{m\mid \ell}\mu^2(m) \sqrt{x}\log(xm\ell) &\le \bigl(\sqrt x\log x\bigr)(3+2A)^\xi.
\end{align*}
\end{prop}

\begin{proof}
We first note the Chebyshev-style inequalities $\psi(t) = \sum_{n\le t} \Lambda(n) < 1.04t$ \cite[Theorem~12]{RS} and $\pi(t) < 1.26 t/\log t$ \cite[Corollary~1]{RS}, both valid for $t>1$. In particular, $Q_\xi = e^{\psi(\xi)} < e^{1.04\cdot\frac12\log x} = x^{0.52}$, and thus $\log(x\ell)\ll\log x$ and $\log(xm\ell)\ll\log x$ in these sums. It remains to estimate the sums over~$\ell$. For the first estimate,
\begin{align*}
\sum_{\ell\mid Q_\xi}\mu^2(\ell)|z-1|^{\omega(\ell)} = \prod_{p\mid Q_\ell} (1+|z-1|) \le \prod_{p\le\xi} (2+A) \le (2+A)^\xi.
\end{align*}
For the second estimate, let $v_p = \lfloor \frac{\log\xi}{\log p} \rfloor$ be the power of~$p$ dividing~$Q_\xi$; then
\begin{align*}
\sum_{\ell\mid Q_\xi} |z|^{\Omega(\ell)}|1-z^{-1}|^{\omega(\ell)} &= \prod_{p\mid Q_\xi} \Bigl( 1 + |z-1| (1 + |z| + \cdots + |z|^{v_p-1}) \Bigr) \\
&\le \prod_{p\le\xi} \biggl( 1+(A+1)\frac{\log\xi}{\log p} \max\{1,A\} \biggr) \\
&\le \prod_{p\le\xi} \biggl( 1+\frac{(A+1)^2}{\log 2} \biggr)^{\log\xi} \\
&\le \biggl( 1+\frac{(A+1)^2}{\log 2} \biggr)^{1.26\xi} \le (2+A)^{4\xi},
\end{align*}
where the last inequality can be verified computationally.
Finally, for the third estimate,
\begin{align*}
\sum_{\ell\mid Q_\xi} |z-1|^{\omega(\ell)}\sum\limits_{m\mid \ell}\mu^2(m) &= \sum_{\ell\mid Q_\xi} (2|z-1|)^{\omega(\ell)} = \prod_{p\mid Q_\ell} (1+2|z-1|) \le \prod_{p\le\xi} (3+2A) \le (3+2A)^\xi.
\qedhere
\end{align*}
\end{proof}

Finally we are left with the sum over divisors of $Q_\xi$ of the main term in Theorem~\ref{thm_asymptoticprimesatisfyingstuff}. But in order to be able to evaluate our different coefficients we would like the sum to be over all integers. It is easy to see that all integers smaller than $\xi$ divide $Q_\xi$; therefore by replacing our sum over divisors of $\xi$ by a sum over all integers, we are introducing an error which is at most the sum over all integers bigger than $\xi$. First a lemma, letting $\sigma(k)$ be the sum of the positive divisors of~$k$.

\begin{lemma} \label{WO lemma}
For $\xi\ge2$, we have $\displaystyle \sum_{k>\xi} \frac{2^{\omega(k)}\sigma(k)}{k^2\varphi(k)} \ll \frac{\log\xi}\xi$.
\end{lemma}

\begin{proof}
Define the nonnegative multiplicative function $f(k) = 2^{\omega(k)}\sigma(k)/\varphi(k)$, and set $F(t) = \sum_{k\le t} f(k)$. We note that $f(p) = 2(p+1)/(p-1) = 2 + O(\frac1p)$ and thus $\sum_{p\le t} f(p) = 2\pi(t) + O(\log\log t)$. Furthermore, $f(p^r) = 2(p^{r+1}-1)/p^{r-1}(p-1)^2 < 2p^2/(p-1)^2 \le 8$ fot all prime powers~$p^r$. It follows from the Wirsing--Odoni method (see for example~\cite[Proposition~4]{FMS}) that $F(t) \sim C_f t\log t$ for some positive constant~$C_f$. Summing by parts, we deduce that
\begin{align*}
\sum_{k>\xi} \frac{2^\omega(k)\sigma(k)}{k^2\varphi(k)} &= \int_\xi^\infty \frac{dF(t)}{t^2} = \frac{F(\xi)}{\xi^2} + 2\int_\xi^\infty \frac{F(t)}{t^3}\,dt \ll \frac{\xi\log\xi}{\xi^2} + \int_\xi^\infty \frac{t\log t}{t^3}\,dt \ll \frac{\log\xi}{\xi}
\end{align*}
as claimed.
\end{proof}

\begin{prop}\label{proposition combined after error terms}
Let $a\in\Q\setminus\{-1,0,1\}$ and $\xi\ge2$. Uniformly for $|z|\le1$,
\begin{align*}
\sum_{\ell\mid Q_\xi} \frac{\mu^2(\ell) (z-1)^{\omega(\ell)}}{[\Q(\sqrt[\ell]{a},\zeta_\ell):\Q]} &= \sum_{\ell=1}^\infty \frac{\mu^2(\ell) (z-1)^{\omega(\ell)}}{[\Q(\sqrt[\ell]{a},\zeta_\ell):\Q]} + O_a\biggl( \frac{\log\xi}\xi \biggr) \\
\sum_{\ell\mid Q_\xi} \frac{z^{\Omega(\ell)} (1-z^{-1})^{\omega(\ell)}}{[\Q(\sqrt[\ell]{a},\zeta_\ell):\Q]} &= \sum_{\ell=1}^\infty \frac{z^{\Omega(\ell)} (1-z^{-1})^{\omega(\ell)}}{[\Q(\sqrt[\ell]{a},\zeta_\ell):\Q]} + O_a\biggl( \frac{\log\xi}\xi \biggr) \\
\sum_{\ell\mid Q_\xi} (z-1)^{\omega(\ell)} \sum_{m\mid \ell} \frac{\mu(m)}{[\Q(\sqrt[\ell]{a},\zeta_{m\ell}):\Q]} &= \sum_{\ell=1}^\infty (z-1)^{\omega(\ell)}\sum_{m\mid \ell} \frac{\mu(m)}{[\Q(\sqrt[\ell]{a},\zeta_{m\ell}):\Q]} + O_a\biggl( \frac{\log\xi}\xi \biggr).
\end{align*}
\end{prop}

\begin{proof}
For the first claim, we start by noting that all integers up to~$\xi$ divide~$Q_\xi$, and thus
\begin{align*}
\sum_{\ell\mid Q_\xi} \frac{\mu^2(\ell) (z-1)^{\omega(\ell)}}{[\Q(\sqrt[\ell]{a},\zeta_\ell):\Q]} &= \sum_{\ell=1}^\infty \frac{\mu^2(\ell) (z-1)^{\omega(\ell)}}{[\Q(\sqrt[\ell]{a},\zeta_\ell):\Q]} + O\biggl( \sum_{\ell\in\N\setminus Q_\xi} \frac{\mu^2(\ell) |z-1|^{\omega(\ell)}}{[\Q(\sqrt[\ell]{a},\zeta_\ell):\Q]} \biggr) \\
&= \sum_{\ell=1}^\infty \frac{\mu^2(\ell) (z-1)^{\omega(\ell)}}{[\Q(\sqrt[\ell]{a},\zeta_\ell):\Q]} + O_a\biggl( \sum_{\ell>\xi} \frac{\mu^2(\ell) 2^{\omega(\ell)}}{\ell\varphi(\ell)} \biggr)
\end{align*}
by Proposition~\ref{proposition degree}. The error term is majorized by the sum in Lemma~\ref{WO lemma} since $\sigma(k)\ge k$, which establishes the first claim.
For the second claim, we similarly write
\begin{align*}
\sum_{\ell\mid Q_\xi} \frac{z^{\Omega(\ell)} (1-z^{-1})^{\omega(\ell)}}{[\Q(\sqrt[\ell]{a},\zeta_\ell):\Q]} &= \sum_{\ell=1}^\infty \frac{z^{\Omega(\ell)} (1-z^{-1})^{\omega(\ell)}}{[\Q(\sqrt[\ell]{a},\zeta_\ell):\Q]} + O\biggl( \sum_{\ell\in\N\setminus Q_\xi} \frac{|z|^{\Omega(\ell)-\omega(\ell)} |z-1|^{\omega(\ell)}}{[\Q(\sqrt[\ell]{a},\zeta_\ell):\Q]} \biggr) \\
&= \sum_{\ell=1}^\infty \frac{z^{\Omega(\ell)} (1-z^{-1})^{\omega(\ell)}}{[\Q(\sqrt[\ell]{a},\zeta_\ell):\Q]} + O_a\biggl( \sum_{\ell>\xi} \frac{2^{\omega(\ell)}}{\ell\varphi(\ell)} \biggr)
\end{align*}
since $|z|\le1$, and again an appeal to Lemma~\ref{WO lemma} establishes the claim.
Finally, for the third claim we write
\begin{align*}
\sum_{\ell\mid Q_\xi} (z-1)^{\omega(\ell)} \sum_{m\mid \ell} \frac{\mu(m)}{[\Q(\sqrt[\ell]{a},\zeta_{m\ell}):\Q]} &= \sum_{\ell=1}^\infty (z-1)^{\omega(\ell)} \sum_{m\mid \ell} \frac{\mu(m)}{[\Q(\sqrt[\ell]{a},\zeta_{m\ell}):\Q]} \\
&\qquad{}+ O\biggl( \sum_{\ell\in\N\setminus Q_\xi} |z-1|^{\omega(\ell)} \sum_{m\mid \ell} \frac{|\mu(m)|}{[\Q(\sqrt[\ell]{a},\zeta_{m\ell}):\Q]} \biggr);
\end{align*}
the error term can be estimated (using Proposition~\ref{proposition degree} again) as
\[
\ll_a \sum_{\ell>\xi} |z-1|^{\omega(\ell)} \sum_{m\mid \ell} \frac{|\mu(m)|}{\ell\varphi(m\ell)} = \sum_{\ell>\xi} \frac{|z-1|^{\omega(\ell)}}{\ell\varphi(\ell)} \sum_{m\mid \ell} \frac{|\mu(m)|}m \le \sum_{\ell>\xi} \frac{2^{\omega(\ell)}}{\ell\varphi(\ell)} \sum_{m\mid \ell} \frac1m = \sum_{\ell>\xi} \frac{2^{\omega(\ell)}}{\ell\varphi(\ell)} \frac{\sigma(\ell)}\ell,
\]
which is again sufficient by Lemma~\ref{WO lemma}.
\end{proof}

\begin{prop} \label{proposition all done but the main terms}
Let $a\in\Q\setminus\{-1,0,1\}$ and $x\ge3$. Uniformly for $|z|\le1$,
\begin{align*}
\sum_{\substack{p\le x \\ \nu_p(a)=0}} z^{\omega((p-1)/\ord_p(a))} &= \pi(x) \sum_{\ell=1}^\infty \frac{\mu^2(\ell) (z-1)^{\omega(\ell)}}{[\Q\left(\sqrt[\ell]{a},\zeta_{\ell}\right):\Q]} + \bo_a\biggl( \frac{x\log\log x}{\log^2 x} \biggr) \\
\sum_{\substack{p\le x \\ \nu_p(a)=0}} z^{\Omega((p-1)/\ord_p(a))} &= \pi(x) \sum_{\ell=1}^\infty \frac{z^{\Omega(\ell)} (1-z^{-1})^{\omega(\ell)}}{[\Q\left(\sqrt[\ell]{a},\zeta_{\ell}\right):\Q]} + \bo_a\biggl( \frac{x\log\log x}{\log^2 x} \biggr) \\
\sum_{\substack{p\le x \\ \nu_p(a)=0}} z^{\omega(p-1) - \omega(\ord_p(a))} &= \pi(x) \sum_{\ell=1}^\infty (z-1)^{\omega(\ell)}\sum\limits_{m\mid \ell} \frac{\mu(m)}{[\Q\left(\sqrt[\ell]{a},\zeta_{m\ell}\right):\Q]} + \bo_a\biggl( \frac{x\log\log x}{\log^2 x} \biggr).
\end{align*}
\end{prop}

\begin{proof}
By adjusting the implicit constants if necessary, we may assume that $x\ge e^{18}$. Set $\xi=\frac19\log x$, so that $\xi\ge2$. We concentrate on establishing the middle formula, beginning with
Proposition~\ref{proposition Omega of quotient} which says that
\begin{align}
\sum_{\substack{p\le x \\ \nu_p(a)=0}} & z^{\Omega((p-1)/\ord_p(a))} = \sum_{\ell\mid Q_\xi} z^{\Omega(\ell)} (1-z^{-1})^{\omega(\ell)} \#\{ p\le x \colon \nu_p(a)=0,\, \ell \mid (p-1)/\ord_p(a) \} \notag \\
&\qquad{}+ O\bigg( \sum_{\xi<q^k\le x} \#\{ p\le x \colon \nu_p(a)=0,\, q^k \mid (p-1)/\ord_p(a) \} \bigg) \notag \\
&= \sum_{\ell\mid Q_\xi} z^{\Omega(\ell)} (1-z^{-1})^{\omega(\ell)} \#\{ p\le x \colon \nu_p(a)=0,\, \ell \mid (p-1)/\ord_p(a) \} + O_a\biggl( \frac{x\log\log x}{\log^2x} \biggr)
\label{concentrate}
\end{align}
by Propositions~\ref{proposition unconditional small xi} and~\ref{proposition smaller xi under GRH} with $B=2$. The main term can be written, using equations~\eqref{palm counts} and~\eqref{palm fields} and Theorem~\ref{thm_asymptoticprimesatisfyingstuff}, as
\begin{align*}
\sum_{\ell\mid Q_\xi} & z^{\Omega(\ell)} (1-z^{-1})^{\omega(\ell)} \#\{ p\le x \colon \nu_p(a)=0,\, \ell \mid (p-1)/\ord_p(a) \} = \sum_{\ell\mid Q_\xi} z^{\Omega(\ell)} (1-z^{-1})^{\omega(\ell)} \Pali(x) \\
&= \sum_{\ell\mid Q_\xi} z^{\Omega(\ell)} (1-z^{-1})^{\omega(\ell)} \biggl( \frac{\pi(x)}{\Bigl[ \Q(\sqrt[\ell]{a},\zeta_{m\ell}):\Q \Bigr]} + \bo_a\Bigl( \sqrt{x}\log(x\ell) \Bigr) \biggr) \\
&= \sum_{\ell\mid Q_\xi} z^{\Omega(\ell)} (1-z^{-1})^{\omega(\ell)} \frac{\pi(x)}{\Bigl[ \Q(\sqrt[\ell]{a},\zeta_{m\ell}):\Q \Bigr]} + \bo_a\Bigl( \sqrt{x}\log x\cdot 3^{4\xi} \Bigr)
\end{align*}
by Proposition~\ref{proposition error sum chebotarev error term combined} with $A=1$. Since $3^{4\xi} = \exp(4\xi\log 3) \le \exp(4\cdot\frac19\log x\cdot\log 3) \le \exp(0.49\log x) = x^{0.49}$, this last error term can be absorbed into the one in equation~\eqref{concentrate}.
Finally, Proposition~\ref{proposition combined after error terms} allows us to rewrite this last main term as
\begin{align*}
\pi(x) \sum_{\ell\mid Q_\xi} \frac{z^{\Omega(\ell)} (1-z^{-1})^{\omega(\ell)}}{[\Q(\sqrt[\ell]{a},\zeta_\ell):\Q]} &= \pi(x) \sum_{\ell=1}^\infty \frac{z^{\Omega(\ell)} (1-z^{-1})^{\omega(\ell)}}{[\Q(\sqrt[\ell]{a},\zeta_\ell):\Q]} + O_a\biggl( \pi(x) \frac{\log\xi}\xi \biggr) \\
&= \pi(x) \sum_{\ell=1}^\infty \frac{z^{\Omega(\ell)} (1-z^{-1})^{\omega(\ell)}}{[\Q(\sqrt[\ell]{a},\zeta_\ell):\Q]} + O_a\biggl( \frac{x\log\log x}{\log^2 x} \biggr).
\end{align*}
Inserting these estimates back into equation~\eqref{concentrate} establishes the middle claim of the proposition. The proofs of the first and last claims are exactly the same, using the other parts of the propositions and theorems referenced above.
\end{proof}

\section{Main terms and the general theorems} \label{hard proofs section}

We are now ready to state and prove our most general results, which are Theorems~\ref{theorem generating function omega of quotient}, \ref{theorem generating function Omega of quotient}, and~\ref{theorem generating function omega minus omega} below. It is straightforward to check that Theorems~\ref{theorem generating function omega of quotient h=1}, \ref{theorem generating function Omega of quotient h=1}, and~\ref{theorem generating function omega minus omega h=1}, respectively, are the special cases of these three general results where $h=1$.

The conclusion of Proposition~\ref{proposition all done but the main terms} already reflects all the work we have done to bound the error terms in the desired asymptotic formulas, and so it remains only to make the main terms more explicit. If the field extension degree $[\Q(\sqrt[\ell]{a},\zeta_{m\ell}):\Q]$ were always equal to $\ell\phi(m\ell)$, then expressing each main term as an Euler product would take only a few lines (and there would be no need for correction factors depending on~$a$ in the statements of the theorems). However, Proposition~\ref{proposition degree} shows that the truth is more complicated, most notably because of the $\varepsilon_a(\ell,m\ell)$ term; the general strategy we adopt to handle these terms appears in equation~\eqref{epsilon-1} below.

In the beginning of this section, we fully state all three general results (which are complicated enough to require several auxiliary pieces of notation). The proofs of Theorems~\ref{theorem generating function omega of quotient}, \ref{theorem generating function Omega of quotient}, and~\ref{theorem generating function omega minus omega} will be given afterwards in Sections~\ref{5.1}, \ref{5.2}, and~\ref{5.3}, respectively. For the purposes of comparison, we first give the analogous general statement for Artin's primitive root conjecture, as established by Hooley~\cite{Hooley}. Throughout this section we fix a rational number $a\notin\{-1,0,1\}$, and we use the quantities $h,a_0,b_0,c_0$ from Notation~\ref{notation beginning} without further comment, as well as the function $\mathfrak{d}(a)$ from Definition~\ref{s and d def}.

\begin{defn} \label{definition general Artin itself}
Given $a\in\Q\setminus\{-1,0,1\}$, define
\[
 H_a=\begin{cases}
 \displaystyle \prod\limits_{q\mid {h}/{2}} \frac{q^2-2q}{q^2-q-1}, &\text{if }a<0\text{ and }2\mid h,\\
 \displaystyle \prod\limits_{q\mid h} \frac{q^2-2q}{q^2-q-1}, & \text{otherwise},
 \end{cases}
\]
and
\[
F_a=\begin{cases}
 \displaystyle 1-\prod\limits_{\substack{q\mid \mathfrak{d}(a) \\ q\nmid h}}\biggl(\frac{-1}{q^2-q-1}\biggr)\prod\limits_{q\mid (\mathfrak{d}(a),h)}\biggl(\frac{-1}{q-2}\biggr),& \text{if }\mathfrak{d}(a)\equiv1\mod{4},\\
 1,& \text{if }\mathfrak{d}(a)\not\equiv1\mod{4}.
\end{cases}
\]
\end{defn}

\noindent
In this notation, Hooley conditionally established Artin's conjecture in the following form:

\begin{thm} \label{theorem general Artin itself}
Assume GRH. Let $a\in\Q\setminus\{-1,0,1\}$ such that~$a$ is not a square. Let $N_a(x)$ be the number of primes up to~$x$ for which~$a$ is a primitive root. Then for $x\ge3$,
\[
\frac1{\pi(x)}\Na(x) = F_aH_a\prod\limits_q \biggl(1-\frac{1}{q(q-1)}\biggr) + \bo_a\biggl(\frac{\log\log x}{\log x}\biggr),
\]
in the notation of Definition~\ref{definition general Artin itself}.
\end{thm}

We now state our three general results in full detail.

\begin{defn} \label{definition generating function omega of quotient}
Given $a\in\Q\setminus\{-1,0,1\}$, define
\[
H^{\omega/}_a(z) = 
\prod_{\substack{q\mid h\\ q\neq 2}} \frac{q(q+z-2)}{q^2-q+z-1}.
\]
Further define
\[
f_a^{\omega/}(z) = \biggl( \prod_{\substack{q\mid 2b_0 \\ q \nmid h}} \frac{z-1}{q^2-q+z-1} \biggr) \biggl( \prod_{{q\mid (b_0,h)}} \frac{z-1}{q+z-2} \biggr)
\]
and
\begin{equation} \label{C omega/ a z definition}
F^{\omega/}_a(z)=\begin{cases}
1, & \text{if }2\nmid h \text{ and } \sgn(a)b_0 \not\equiv 1\mod{4},\\
1+f_a^{\omega/}(z) ,& \text{if }2\nmid h \text{ and } \sgn(a)b_0 \equiv 1\mod{4},\\
2z/(z+1), & \text{if }2\mid h \text{ and } a>0,\\
1, & \text{if }2\mid h \text{ and } a<0.
\end{cases}
\end{equation} 
\end{defn}

\begin{thm}\label{theorem generating function omega of quotient}
Assume GRH. Let $a\in\Q\setminus\{-1,0,1\}$ and $x\ge3$. Uniformly for $|z|\le1$,
\begin{equation} \label{equation generating function omega of quotient}
\frac{1}{\pi(x)} \sum_{\substack{p\le x \\ \nu_p(a)=0}} z^{\omega((p-1)/\ord_p(a))}
=
F^{\omega/}_a(z) H^{\omega/}_a(z) \prod_q \biggl( 1 + \frac{z-1}{q(q-1)} \biggr)
+ O_a\biggl( \frac{\log\log x}{\log x} \biggr)
\end{equation}
in the notation of Definition~\ref{definition generating function omega of quotient}.
\end{thm}

\begin{rmk} \label{z=0 remark}
Note that when $z=0$, the infinite product in Theorem~\ref{theorem generating function omega of quotient} becomes $\prod_q \Bigl( 1-\frac{1}{q(q-1)} \Bigr)\approx 0.3739\ldots$ which is precisely Artin's constant. Indeed, the special case $z=0$ of Theorem~\ref{theorem generating function omega of quotient} is exactly Artin's conjecture as stated in Theorem~\ref{theorem general Artin itself}, since $\omega\Bigl( (p-1)/\ord_p(a) \Bigr) = 0$ precisely when~$a$ is a primitive root modulo~$p$ (and thanks to the convention $0^0=1$ in generating functions).\
The same comments apply to Theorem~\ref{theorem generating function Omega of quotient} below, since $\Omega\Bigl( (p-1)/\ord_p(a) \Bigr) = 0$ if and only if $\omega\Bigl( (p-1)/\ord_p(a) \Bigr) = 0$.
\end{rmk}

\begin{defn}\label{definition generating function Omega of quotient}
Given $a\in\Q\setminus\{-1,0,1\}$, set $\gamma=b_0/(2,b_0)$. Define
\begin{align*}
 H^\Omega_h(z) &= \prod\limits_{q\mid h}\left(\frac{q^4-2q^3+2qz-z^2-(z-1)(q-1)z^{\nu_q(h)}q^{-\nu_q(h)+2}}{(q-z)(q^3-q^2-q+z)}\right) \\
 I^\Omega_h(\gamma,z) &= \frac{(z-1)^{\omega(\gamma)}(\gamma,h)}{\gamma\varphi(\gamma)}\prod\limits_{q\mid\gamma}\left(\frac{(q-1)(q^3-qz-(q-1)z^{\nu_q(h/(h,q))+1}q^{-\nu_q(h/(h,q))+1})}{q^4-2q^3+2qz-z^2-(z-1)(q-1)z^{\nu_q(h)}q^{-\nu_q(h)+2}}\right),
\end{align*}
and note that $I^\Omega_h(1,z) = 1$.
Further define
\begin{align*}
\tau^\Omega_+(a,z) &= \begin{cases}
z/2, &\text{if }b_0\equiv 3\mod{4}\text{ and }2\nmid h,\\
z^2/8, &\text{if }b_0\equiv 2\mod{4}\text{ and }2\nmid h,\\
z^2/4, &\text{if }b_0\equiv 2\mod{4}\text{ and }2\mid\mid h,\\
z^{\nu_2(h)}/2^{\nu_2(h)-1}, &\text{otherwise}
\end{cases} \\
\tau^\Omega_-(a,z) &= \begin{cases}
2, &\text{if }b_0\equiv 3\mod{4}\text{ and }2\nmid h,\\
z^2/8, &\text{if }b_0\equiv 2\mod{4}\text{ and }2\nmid h,\\
z, &\text{if }b_0\equiv 2\mod{4}\text{ and }2\mid\mid h,\\
(z/2)^{\nu_2(h)+1}, &\text{otherwise.}
\end{cases}
\end{align*}
Finally, define
\begin{equation} \label{secretly g(infty) Omega}
F^\Omega_a(h,z) = 1+\frac{(z-1)(z-2)(z-4)}{4z-z^2-4(z-1)(z/2)^{\nu_2(h)}} \begin{cases}
\dfrac{1}{4-z}I^\Omega_h(\gamma,z)\tau^\Omega_+(a,z), &\text{if }a>0,\\
\dfrac{1-(z/2)^{\nu_2(h)}}{z-2}+\dfrac{1}{4-z}I^\Omega_h(\gamma,z)\tau^\Omega_-(a,z), &\text{if }a<0.
\end{cases}
\end{equation}
\end{defn}

\begin{thm} \label{theorem generating function Omega of quotient}
Assume GRH. Let $a\in\Q\setminus\{-1,0,1\}$ and $x\ge3$. Uniformly for $|z|\le1$,
\begin{equation} \label{equation generating function Omega of quotient}
\frac1{\pi(x)} \sum_{\substack{p\le x \\ \nu_p(a)=0}} z^{\Omega\left((p-1)/\ord_p(a)\right)}  = H^\Omega_h(z)F^\Omega_a(h,z)\prod\limits_{q}\biggl(1+\frac{q(z-1)}{(q-1)(q^2-z)}\biggl) + \bo_a\biggl( \frac{\log\log x}{\log x} \biggr)
\end{equation}
in the notation of Definition~\ref{definition generating function Omega of quotient}.
\end{thm}

\begin{defn}\label{definition generating function omega minus omega}
Given $a\in\Q\setminus\{-1,0,1\}$, set $\gamma={b_0}/{(2,b_0)}$.
Define
\begin{align*}
 H_h^{\omega-}(z)&=\prod\limits_{\substack{q|h}}\left(\frac{q^2-1+(z-1)(q+1-q^{-\nu_q(h)+1})}{q^2+z-2}\right)\\
 I_h^{\omega-}(\gamma,z)&=\frac{(z-1)^{\omega(\gamma)}(\gamma,h)}{\gamma}\prod\limits_{\substack{ q\mid \gamma}}\left(\frac{q+1-q^{-\nu_q(h/(h,q))}}{q^2-1+(z-1)(q+1-q^{-\nu_q(h)+1})}\right).
\end{align*}
Further define
\begin{align*}
\tau_+^{\omega-}(a) &= \begin{cases}
-1/2, &\text{if }b_0\equiv 3\mod{4}\text{ and }2\nmid h,\\
-1/8, &\text{if }b_0\equiv 2\mod{4}\text{ and }2\nmid h,\\
-1/4, &\text{if }b_0\equiv 2\mod{4}\text{ and }2\mid\mid h,\\
1/2^{\nu_2(h)}, &\text{otherwise}
\end{cases} \\
\tau_-^{\omega-}(a) &= \begin{cases}
1, &\text{if }b_0\equiv 3\mod{4}\text{ and }2\nmid h,\\
-1/8, &\text{if }b_0\equiv 2\mod{4}\text{ and }2\nmid h,\\
1/2, &\text{if }b_0\equiv 2\mod{4}\text{ and }2\mid\mid h,\\
-1/2^{\nu_2(h)+1}, &\text{otherwise.}
\end{cases}
\end{align*}
Finally, define
\begin{equation} \label{secretly g(infty)}
F_a^{\omega-}(h,z) = 1+\frac{z-1}{1+(z-1)( 1-2^{1-\nu_2(h)}/3)} \begin{cases}
\frac{1}{3}I_h^{\omega-}(\gamma,z)\tau_+^{\omega-}(a), &\text{if }a>0,\\
2^{-\nu_2(h)}-1+\frac{1}{3}I_h^{\omega-}(\gamma,z)\tau_-^{\omega-}(a), &\text{if }a<0.
\end{cases}
\end{equation}
\end{defn}

\begin{thm}\label{theorem generating function omega minus omega}
Assume GRH. Let $a\in\Q\setminus\{-1,0,1\}$ and $x\ge3$. Uniformly for $|z|\le1$,
\begin{equation} \label{equation generating function omega minus omega}
\frac1{\pi(x)} \sum_{\substack{p\le x \\ \nu_p(a)=0}} z^{\omega(p-1)-\omega(\ord_p(a))}  = H_h^{\omega-}(z)F_a^{\omega-}(h,z)\prod\limits_{\substack{q}}\biggl(1+\frac{z-1}{q^2-1}\biggl) + \bo_a\biggl( \frac{\log\log x}{\log x} \biggr)
\end{equation}
in the notation of Definition~\ref{definition generating function omega minus omega}.
\end{thm}

In the proofs that follow,
we will use a strategy spurred by the form of Proposition~\ref{proposition degree}, namely writing
\begin{equation} \label{epsilon-1}
\frac1{[\Q(\sqrt[\ell]{a},\zeta_{m\ell}):\Q]} = \frac{(\ell,h)\varepsilon_a(m\ell,\ell)}{\ell\varphi(m\ell)} = \frac{(\ell,h)}{\ell\varphi(m\ell)} + \frac{(\ell,h)}{\ell\varphi(m\ell)}(\varepsilon_a(m\ell,\ell)-1);
\end{equation}
the first summand leads to the main term, while the second summand leads to the correction factors depending on~$a$.

\subsection{The statistic $\omega\Bigl( (p-1)/\ord_p(a) \Bigr)$} \label{5.1}

This statistic is by far the simplest of the three, as having~$\ell$ squarefree discards many cases in Proposition~\ref{proposition degree} and makes the products easier to compute. Throughout we will freely use the notation from Definition~\ref{s and d def}, Notation~\ref{notation beginning}, and Definition~\ref{definition generating function omega of quotient}.

\begin{lemma} \label{lemma main term omega of quotient}
Define
$\displaystyle
S_0^{\omega/}(n) = 
\sum_{\substack{\ell\ge1 \\ n\mid\ell}} \mu^2(\ell)(z-1)^{\omega(\ell)} \frac{(\ell,h)}{\ell \varphi(\ell)}.
$
When $n$ is squarefree,
\begin{equation*}
S_0^{\omega/}(n) 
= 
\prod_{\substack{q\mid n \\ q\nmid h}} \frac{z-1}{q^2-q+z-1} 
\prod_{q\mid(n,h)} \frac{z-1}{(q-1)+(z-1)}
\prod_q \frac{q^2-q+z-1}{q^2-q}
H_a^{\omega/}(z)
\begin{cases}
2z/(z+1), &\text{if } 2\mid h, \\ 1, &\text{if } 2\nmid h.
\end{cases}
\end{equation*}
\end{lemma}

\begin{proof}
If we write $\ell=nk$, then $nk$ must be squarefree to contribute to the sum due to the $\mu^2(\ell)$ factor, which means that $(n,k)=1$. Therefore
\begin{align*}
S_0^{\omega/}(n) &= \sum_{k=1}^\infty \mu^2(nk)(z-1)^{\omega(nk)}\frac{(nk,h)}{nk\varphi(nk)} \\
 &= (z-1)^{\omega(n)} \frac{(n,h)}{n\varphi(n)} \sum_{(k,n)=1} \mu^2(k)(z-1)^{\omega(k)}\frac{(k,h)}{k\varphi(k)} \\
 &= (z-1)^{\omega(n)} \frac{(n,h)}{n\varphi(n)} \prod_{q\nmid n} \biggl (1+\frac{(z-1)(q,h)}{q(q-1)} \biggr).
\end{align*}
The product can be written as
\begin{align*}
\prod_{q\nmid n} \biggl( 1+\frac{(z-1)(q,h)}{q(q-1)} \biggr) 
&= \prod_q \biggl( 1+\frac{z-1}{q(q-1)} \biggr) 
\prod_{q\mid h} \biggl( 1+\frac{z-1}{q(q-1)} \biggr)^{-1} \biggl( 1+\frac{z-1}{q-1} \biggr) \\
&\qquad{}\times \prod_{\substack{q\mid n \\ q\nmid h}} \biggl( 1+\frac{z-1}{q(q-1)} \biggr)^{-1} \prod_{q\mid(n,h)} \biggl( 1+\frac{z-1}{q-1} \biggr)^{-1} .
\end{align*}
We complete the proof by noting that
\[
\prod_{q\mid h} \biggl( 1+\frac{z-1}{q(q-1)} \biggr)^{-1} \biggl( 1+\frac{z-1}{q-1} \biggr)
= H_a^{\omega/}(z) \begin{cases}
2z/(z+1), &\text{if } 2\mid h, \\ 1, &\text{if } 2\nmid h
\end{cases}
\]
and
\begin{align*}
(z-1)^{\omega(n)} \frac{(n,h)}{n\varphi(n)}
\prod_{\substack{q\mid n \\ q\nmid h}} \biggl( 1+\frac{z-1}{q(q-1)} \biggr)^{-1} \prod_{q\mid(n,h)} \biggl( 1+\frac{z-1}{q-1} \biggr)^{-1} 
&= \prod_{\substack{q\mid n \\ q\nmid h}} \frac{z-1}{q^2-q+z-1} 
\prod_{q\mid(n,h)} \frac{z-1}{(q-1)+(z-1)} .
\end{align*}
\end{proof}

The following proposition immediately implies Theorem~\ref{theorem generating function omega of quotient} when combined with Proposition~\ref{proposition all done but the main terms}.

\begin{prop}\label{proposition main term omega of quotient}
Let $a\in\Q\setminus\{-1,0,1\}$. For $|z|\le1$,
\[
\sum_{\ell=1}^\infty \frac{\mu^2(\ell) (z-1)^{\omega(\ell)}}{[\Q(\sqrt[\ell]{a},\zeta_{\ell}):\Q]} = F^{\omega/}_a(z)H^{\omega/}_a(z)
\prod_q \biggl( 1 + \frac{z-1}{q(q-1)} \biggr)
\]
in the notation of Definition~\ref{definition generating function omega of quotient}.
\end{prop}

\begin{proof}
Using the notation of Lemma~\ref{lemma main term omega of quotient},
$\sum_\ell {\mu^2(\ell)(z-1)^{\omega(\ell)}}/{[\Q(\sqrt[\ell]{a},\zeta_{\ell}):\Q]} = S_0^{\omega/}(1) + S_1$ where
\begin{equation} \label{5.10 S1 def}
S_1 = \sum\limits_\ell \mu^2(\ell)(z-1)^{\omega(\ell)}\frac{(\ell,h)}{\ell \varphi(\ell)}\left(\varepsilon_a(\ell,\ell)-1\right)
\end{equation}
in the notation of Proposition~\ref{proposition degree}. By Lemma~\ref{lemma main term omega of quotient} we have simply
\[
S_0^{\omega/}(1) = 
\prod_q \frac{q^2-q+z-1}{q^2-q}
H_a^{\omega/}(z) \begin{cases}
2z/(z+1), &\text{if } 2\mid h, \\ 1, &\text{if } 2\nmid h.
\end{cases}
\]
The evaluation of~$S_1$ depends upon the sign of~$a$ and the parity of~$h$. Throughout this proof it will be helpful to remember that $4\nmid\ell$ due to the presence of the $\mu^2(\ell)$ factor.

First assume that $a>0$. By Proposition~\ref{proposition degree},
\begin{equation} \label{5.10 S_1}
S_1 = \sum\limits_{\substack{\ell\ge1 \\ 2\mid \ell'\\ \mathfrak{d}(a_0)\mid \ell}}\mu^2(\ell)(z-1)^{\omega(\ell)}\frac{(\ell,h)}{\ell\varphi(\ell)} = \id(2\nmid h) \sum\limits_{\substack{2\mid \ell\\ \mathfrak{d}(a_0)\mid \ell}}\mu^2(\ell)(z-1)^{\omega(\ell)}\frac{(\ell,h)}{\ell\varphi(\ell)}.
\end{equation}
In particular, if $a>0$ and $2\mid h$, then $S_1=0$ and therefore
\[
S_0^{\omega/}(1) + S_1 
= \frac{2z}{z+1} H_a^{\omega/}(z) \prod_q \frac{q^2-q+z-1}{q^2-q}
= F_a^{\omega/}(z) H_a^{\omega/}(z) \prod_q \frac{q^2-q+z-1}{q^2-q}
\]
as required.

We examine $S_1$ further in the case $a>0$ and $2\nmid h$.
Since $\ell$ is squarefree, the only way for $\mathfrak{d}(a_0) = \mathfrak{d}(b_0)$ to divide~$\ell$ is for $b_0\equiv1\mod{4}$; thus in this case,
\begin{align} \label{b0+}
S_1 &= \id(b_0\equiv1\mod{4}) \sum_{2b_0\mid \ell} \mu^2(\ell)(z-1)^{\omega(\ell)}\frac{(\ell,h)}{\ell\varphi(\ell)} = \id \Bigl( \sgn(a)b_0\equiv1\mod{4} \Bigr) S_0^{\omega/}(2b_0).
\end{align}
In particular, if $\sgn(a)b_0\not\equiv1\mod{4}$ then again $S_1=0$ and
\[
S_0^{\omega/}(1) + S_1 
= H_a^{\omega/}(z) \prod_q \frac{q^2-q+z-1}{q^2-q}
= F_a^{\omega/}(z) H_a^{\omega/}(z) \prod_q \frac{q^2-q+z-1}{q^2-q}
\]
as required. Finally, if $\sgn(a)b_0\equiv1\mod{4}$ then by Lemma~\ref{lemma main term omega of quotient},
\begin{align*}
S_0^{\omega/}(1) + S_1 
&= H_a^{\omega/}(z) \prod_q \frac{q^2-q+z-1}{q^2-q}
F_a^{\omega/}(z)
\end{align*}
as required.

Next, when $a<0$ and $2\nmid h$, the only case in Proposition~\ref{proposition degree} that results in $\varepsilon_a(\ell,\ell)\ne0$ is when $\ell\equiv\ell'\equiv2\mod4$ and $\mathfrak{d}(-a_0) \equiv 1\mod 4$. This argument yields the same expression for~$S_1$ as in equation~\eqref{b0+} except with~$b_0$ replaced by~$-b_0$; the analogue of the argument following equation~\eqref{b0+} verifies the proposition in this case as well.

Finally, when $a<0$ and $2\mid h$, the only case in Proposition~\ref{proposition degree} that results in $\varepsilon_a(\ell,\ell)\ne0$ is when~$\ell$ is even but~$\ell'$ is odd, in which case $\varepsilon_a(\ell,\ell)-1=-\frac12$. Therefore in this case,
\begin{align*}
S_1 &= -\frac{1}{2} \sum_{\substack{\ell\ge1\\ 2\mid \ell}} \mu^2(\ell)(z-1)^{\omega(\ell)}\frac{(\ell,h)}{\ell\varphi(\ell)} = -\frac12S_0^{\omega/}(2),
\end{align*}
and so by Lemma~\ref{lemma main term omega of quotient},
\begin{align*}
S_0^{\omega/}(1) + S_1 
&= H_a^{\omega/}(z) \frac{2z}{z+1} \prod_q \frac{q^2-q+z-1}{q^2-q}
\cdot \biggl( 1 - \frac12 \prod_{q\mid(2,h)} \biggl( 1 - \frac1{1 + \frac{z-1}{q-1}} \biggr) \biggr) \\
&= H_a^{\omega/}(z) \prod_q \frac{q^2-q+z-1}{q^2-q} \cdot
\frac{2z}{z+1} \biggl( 1 - \frac12 \biggl( 1 - \frac1{1-2(1-z)/2(2-1)} \biggr) \biggr) \\
&= H_a^{\omega/}(z) \prod_q \frac{q^2-q+z-1}{q^2-q} 
\cdot 1
\end{align*}
as required.
\end{proof}

\subsection{The statistic $\Omega\Bigl( (p-1)/\ord_p(a) \Bigr)$} \label{5.2}

The Kummer extensions relevant to this statistic now include non-squarefree values of~$\ell$, which makes their degrees somewhat more complicated. However, we are still using Proposition~\ref{proposition degree} only for $m=1$, which limits this case to one of intermediate difficulty.
Throughout we will freely use the notation from Definition~\ref{s and d def}, Notation~\ref{notation beginning}, and Definition~\ref{definition generating function Omega of quotient}.

\begin{defn} \label{definition main term Omega}
For any positive integer~$\lambda$, define
\begin{align*}
 g^\Omega(\lambda) &= 1+2(1-z^{-1})\sum\limits_{k=1}^{\lambda}z^k \frac{(2^k,h)}{2^{2k}} \\
 &= \begin{cases}
1+\dfrac{2(z-1) \Bigl( 1-(z/2)^{\lambda} \Bigr)}{2-z}, & \text{if } \lambda\leq \nu_2(h), \\
\dfrac{4z-z^2+2(z-1)\left((z-2)2^{\nu_2(h)}(z/4)^\lambda-z^{\nu_2(h)}2^{-\nu_2(h)+1}\right)}{(2-z)(4-z)}, & \text{otherwise}
 \end{cases} \\
g^\Omega(\infty) &= \lim_{\lambda\to\infty} g^\Omega(\lambda) = \frac{4z-z^2-4(z-1)(z/2)^{\nu_2(h)}}{(z-2)(z-4)}.
\end{align*}
(Note that $g^\Omega(\infty)$ appears as a denominator in equation~\eqref{secretly g(infty) Omega}.) 
\end{defn}

\begin{lemma} \label{lemma main term Omega}
Let~$\newa$ be an integer and let $\newb\in\{\alpha,\alpha+1,\alpha+2,\ldots\}\cup\{\infty\}$. For any any positive integer~$h$ and any odd squarefree integer~$\gamma$, write 
\begin{equation*}
 S^\Omega_h(\alpha,\beta,\gamma,z)=\sum\limits_{\substack{\ell\geq 1\\ \gamma\mid\ell\\ \alpha\leq\nu_2(\ell)\leq\beta}}z^{\Omega(\ell)}\Bigl(1-z^{-1}\Bigr)^{\omega(\ell)}\frac{(\ell,h)}{\ell\varphi(\ell)}.
\end{equation*}
Then
\begin{equation*}
 S^\Omega_h(\alpha,\beta,\gamma,z)=\prod\limits_{q}\left(1+\frac{q(z-1)}{(q-1)(q^2-z)}\right)\frac{H^\Omega_h(z)I^\Omega_h(\gamma,z)}{g^\Omega(\infty)} \begin{cases}
g^\Omega(\beta),&\text{if }\alpha=0,\\ g^\Omega(\beta)-g^\Omega(\alpha-1),&\text{otherwise} \end{cases}
\end{equation*}
in the notation of Definitions~\ref{definition generating function Omega of quotient} and~\ref{definition main term Omega}.
\end{lemma}

\begin{proof}
First we compute the sum over odd $\ell$:
\begin{align*}
 S^\Omega_h(0,0,\gamma,z)&=\sum\limits_{\substack{\ell\geq 1\\ \gamma\mid\ell\\ 2\nmid\ell}}z^{\Omega(\ell)}\Bigl(1-z^{-1}\Bigr)^{\omega(\ell)}\frac{(l,h)}{l\varphi(\ell)}\\
 &=\frac{z^{\Omega(\gamma)}\Bigl(1-z^{-1}\Bigr)^{\omega(\gamma)}(\gamma,h)}{\gamma\varphi(\gamma)}\sum\limits_{\substack{\ell\geq 1\\ 2\nmid\ell}}z^{\Omega(\ell)}\Bigl(1-z^{-1}\Bigr)^{\omega(\ell)-\omega((\ell,\gamma))}\frac{(\ell,\frac{h}{(h,\gamma)})}{\ell \frac{\varphi(\gamma \ell)}{\varphi(\gamma)}}\\
 &=\frac{(z-1)^{\omega(\gamma)}(\gamma,h)}{\gamma\varphi(\gamma)}\prod\limits_{\substack{q\geq 3\\ q\nmid \gamma}}\biggl(1+\frac{(z-1)q}{z(q-1)}\sum\limits_{k=1}^{\infty}\frac{z^k(q^k,h)}{q^{2k}}\biggr)\prod\limits_{\substack{q\geq 3\\ q\mid \gamma}}\biggl(1+\sum\limits_{k=1}^{\infty}\frac{z^k(q^k,\frac{h}{(h,\gamma)})}{q^{2k}}\biggr)\\
  &=\frac{(z-1)^{\omega(\gamma)}(\gamma,h)}{\gamma\varphi(\gamma)g^\Omega(\infty)}\prod\limits_{\substack{q}}\biggl(1+\frac{(z-1)q}{z(q-1)}\sum\limits_{k=1}^{\infty}\frac{z^k(q^k,h)}{q^{2k}}\biggr)\prod\limits_{\substack{q\geq 3\\ q\mid \gamma}} \frac{1+\sum\limits_{k=1}^{\infty}\frac{z^k(q^k,\frac{h}{(h,\gamma)})}{q^{2k}}}{1+\frac{(z-1)q}{z(q-1)}\sum\limits_{k=1}^{\infty}\frac{z^k(q^k,h)}{q^{2k}}} \\
    &=\prod\limits_{q}\biggl(1+\frac{q(z-1)}{(q-1)(q^2-z)}\biggr)\frac{H^\Omega_h(z)I^\Omega_h(\gamma,z)}{g^\Omega(\infty)}.
\end{align*}
Then taking into account the $2$-adic valuation of~$\ell$ in general,
\begin{align*}
S^\Omega_h(\alpha,\beta,\gamma,z) = \sum_{k=\alpha}^\beta S^\Omega_h(k,k,\gamma,z) &= \sum_{k=\alpha}^\beta S^\Omega_h(0,0,\gamma,z) z^k (1-z^{-1}) \frac{(2^k,h)}{2^{2k-1}} \\
&= S^\Omega_h(0,0,\gamma,z) \begin{cases}
g^\Omega(\beta),&\text{if }\alpha=0,\\ g^\Omega(\beta)-g^\Omega(\alpha-1),&\text{otherwise.} \end{cases}
\qedhere
\end{align*}
\end{proof}

The following proposition immediately implies Theorem~\ref{theorem generating function Omega of quotient} when combined with Proposition~\ref{proposition all done but the main terms}.

\begin{prop}\label{proposition main term Omega}
Let $a\in\Q\setminus\{-1,0,1\}$. For $|z|\le1$,
\begin{equation*}
\sum\limits_{\ell\geq 1} \frac{z^{\Omega(\ell)}(1-z^{-1})^{\omega(\ell)}}{[\Q(\sqrt[\ell]{a},\zeta_{\ell}):\Q]} = H^\Omega_h(z)F^\Omega_a(h,z)\prod\limits_{q}\biggl(1+\frac{q(z-1)}{(q-1)(q^2-z)}\biggl)
\end{equation*}
in the notation of Definition~\ref{definition generating function Omega of quotient}.
\end{prop}

\begin{proof}
In the notation of Proposition~\ref{proposition degree}, define
\[
S^\Omega_1=\sum\limits_{\ell\geq 1}\left(\varepsilon_a(\ell,\ell)-1\right) \frac{z^{\Omega(\ell)}(1-z^{-1})^{\omega(\ell)}(\ell,h)}{\ell\varphi(\ell)},
\]
so that
\begin{align*}
\sum\limits_{\ell\geq 1} \frac{z^{\Omega(\ell)}(1-z^{-1})^{\omega(\ell)}}{[\Q(\sqrt[\ell]{a},\zeta_{\ell}):\Q]} &= \sum\limits_{\ell\geq 1} \frac{z^{\Omega(\ell)}(1-z^{-1})^{\omega(\ell)}(\ell,h)}{\ell\varphi(\ell)} + S^\Omega_1 \\
&= S^\Omega_h(0,\infty,1,z) + S_1 = H^\Omega_h(z) \prod_q \biggl( 1+\frac{q(z-1)}{(q-1)(q^2-z)} \biggr) + S^\Omega_1
\end{align*}
by Lemma~\ref{lemma main term Omega}, since $I^\Omega_h(1,z) = 1$. The first term on the right-hand side accounts for the initial~$1$ in the definition~\eqref{secretly g(infty) Omega} of $F_a^\Omega(h,z)$; we now show that~$S^\Omega_1$ accounts for the complicated remainder of that definition.

We first consider the case $a>0$. In this case,
\begin{align*}
 S^\Omega_1&=\sum\limits_{\substack{\ell\geq 1\\ \nu_2(\ell)>\nu_2(h)\\ \mathfrak{d}(a_0)\mid\ell}} \frac{z^{\Omega(\ell)}(1-z^{-1})^{\omega(\ell)}(\ell,h)}{\ell\varphi(\ell)}\\
 &=S^\Omega_h(\max(\nu_2(h)+1,\eta),\infty,\gamma,z)\\
 &=\prod\limits_{q}\left(1+\frac{q(z-1)}{(q-1)(q^2-z)}\right)\frac{H^\Omega_h(z)I^\Omega_h(\gamma,z)}{g^\Omega(\infty)} \begin{cases}
g^\Omega(\infty)-g^\Omega(\eta-1),&\text{if }\eta> \nu_2(h)+1,\\g^\Omega(\infty)-g^\Omega(\nu_2(h)),& \text{otherwise}\end{cases}\\
 &=\prod\limits_{q}\left(1+\frac{q(z-1)}{(q-1)(q^2-z)}\right)\frac{H^\Omega_h(z)I^\Omega_h(\gamma,z)(z-1)}{g^\Omega(\infty)(4-z)} \begin{cases}
2^{-2\eta+\nu_2(h)+3}z^{\eta-1},&\text{if }\eta> \nu_2(h)+1,\\ 2^{-\nu_2(h)+1}z^{\nu_2(h)},& \text{otherwise}\end{cases}\\
 &=\prod\limits_{q}\left(1+\frac{q(z-1)}{(q-1)(q^2-z)}\right)\frac{H^\Omega_h(z)I^\Omega_h(\gamma,z)(z-1)}{g^\Omega(\infty)(4-z)}\tau^\Omega_+(a,z).
\end{align*}

Assume now that $a<0$. We split $S^\Omega_1$ into three parts according to the $2$-adic valuation of $\ell$, to fit the different cases of Proposition~\ref{proposition degree}:
 \begin{align*}
  S^\Omega_{1,1}&=\sum\limits_{\substack{\ell\ge1\\\nu_2(\ell)\leq\nu_2(h)}}\left(\varepsilon_a(\ell,\ell)-1\right) \frac{z^{\Omega(\ell)}(1-z^{-1})^{\omega(\ell)}(\ell,h)}{\ell\varphi(\ell)}\\
  S^\Omega_{1,2}&=\sum\limits_{\substack{\ell\ge1\\\nu_2(\ell)=\nu_2(h)+1}}\left(\varepsilon_a(\ell,\ell)-1\right) \frac{z^{\Omega(\ell)}(1-z^{-1})^{\omega(\ell)}(\ell,h)}{\ell\varphi(\ell)}\\
  S^\Omega_{1,3}&=\sum\limits_{\substack{\ell\ge1\\\nu_2(\ell)\geq\nu_2(h)+2}}\left(\varepsilon_a(\ell,\ell)-1\right) \frac{z^{\Omega(\ell)}(1-z^{-1})^{\omega(\ell)}(\ell,h)}{\ell\varphi(\ell)}.
 \end{align*}
 Starting with $S^\Omega_{1,1}$, we use Proposition~\ref{proposition degree} to write
 \begin{align*}
  S^\Omega_{1,1}&=\frac{-\mathbbm{1}(2\mid h)}{2}\sum\limits_{\substack{\ell\ge1\\1\leq \nu_2(\ell)\leq\nu_2(h)}}\frac{z^{\Omega(\ell)}(1-z^{-1})^{\omega(\ell)}(\ell,h)}{\ell\varphi(\ell)}\\
  &=\frac{-\mathbbm{1}(2\mid h)}{2}S^\Omega_h(1,\nu_2(h),1,z)\\
  &=\frac{-\mathbbm{1}(2\mid h)}{2}\prod\limits_{q}\left(1+\frac{q(z-1)}{(q-1)(q^2-z)}\right)\frac{H^\Omega_h(z)I^\Omega_h(1,z)}{g^\Omega(\infty)}(g^\Omega(\nu_2(h))-1)\\
  &=-\prod\limits_{q}\left(1+\frac{q(z-1)}{(q-1)(q^2-z)}\right)\frac{H^\Omega_h(z)}{g^\Omega(\infty)}(z-1)\left(\frac{1-(z/2)^{\nu_2(h)}}{2-z}\right)
 \end{align*}
 by Lemma~\ref{lemma main term Omega}.
 Similarly, recalling that $\gamma=b_0/(2,b_0)$, Proposition~\ref{proposition degree} implies
\begin{align*}
S^\Omega_{1,2}&=\sum\limits_{\substack{\ell\ge1\\\nu_2(\ell)=\nu_2(h)+1}}\left(\varepsilon_a(\ell,\ell)-1\right) \frac{z^{\Omega(\ell)}(1-z^{-1})^{\omega(\ell)}(\ell,h)}{\ell\varphi(\ell)}\\
  &=\begin{cases}     S^\Omega_h(1,1,\gamma,z),&\text{if } 2\nmid h \text{ and }b_0\equiv 3\mod{4},\\    S^\Omega_h(2,2,\gamma,z),&\text{if } 2\mid\mid h \text{ and }b_0\equiv 2\mod{4},\\      0,&\text{otherwise} \end{cases} \\
  &=\prod\limits_{q}\left(1+\frac{q(z-1)}{(q-1)(q^2-z)}\right)\frac{H^\Omega_h(z)I^\Omega_h(\gamma,z)(z-1)}{g^\Omega(\infty)} \begin{cases}     1/2,&\text{if } 2\nmid h \text{ and }b_0\equiv 3\mod{4},\\    z/4,&\text{if } 2\mid\mid h \text{ and }b_0\equiv 2\mod{4},\\      0,&\text{otherwise.} \end{cases}
 \end{align*}
 Finally,
  \begin{align*}
  S^\Omega_{1,3}&=\sum\limits_{\substack{\ell\ge1\\\nu_2(\ell)\geq\nu_2(h)+2\\ 2^{\nu_2(\mathfrak{d}(a_0))}\mid\ell\\ \gamma\mid \ell}} \frac{z^{\Omega(\ell)}(1-z^{-1})^{\omega(\ell)}(\ell,h)}{\ell\varphi(\ell)}\\
    &=\begin{cases}S^\Omega_h(3,\infty,\gamma,z),&\text{if }2\nmid h \text{ and }b_0\equiv 2\mod{4},\\ S^\Omega_h\Bigl(\nu_2(h)+2,\infty,\gamma,z\Bigr),&\text{otherwise}\end{cases}\\
    &=\prod\limits_{q}\left(1+\frac{q(z-1)}{(q-1)(q^2-z)}\right)\frac{H^\Omega_h(z)I^\Omega_h(\gamma,z)(z-1)}{g^\Omega(\infty)(4-z)} \begin{cases}2^{-3}z^2,&\text{if }2\nmid h \text{ and }b_0\equiv 2\mod{4},\\ 2^{-\nu_2(h)-1}z^{\nu_2(h)+1},&\text{otherwise.}\end{cases}
 \end{align*}
 Adding and regrouping these evaluations results in
 \begin{align*}
  S^\Omega_1&=\prod\limits_{q}\left(1+\frac{q(z-1)}{(q-1)(q^2-z)}\right)\frac{H^\Omega_h(z)(z-1)}{g^\Omega(\infty)}\left(\frac{1-(z/2)^{\nu_2(h)}}{z-2}+\frac{I^\Omega_h(\gamma,z)}{4-z}\tau^\Omega_-(a,z)\right),
 \end{align*}
 which completes the proof of the proposition.
\end{proof}

\subsection{The statistic $\omega(p-1) - \omega\Bigl( \ord_p(a) \Bigr)$} \label{5.3}

This statistic is the most complicated to address, since the cyclotomic-Kummer extensions involved now require the full strength of Proposition~\ref{proposition degree} for all values of~$\ell$ and~$m$.
Throughout we will freely use the notation from Definition~\ref{s and d def}, Notation~\ref{notation beginning}, and Definition~\ref{definition generating function omega minus omega}. Since this is the last case to consider, we decline to include $\omega-$ superscripts in the new notation we introduce along the way.

\begin{defn} \label{definition main term omega of minus}
For any positive integer~$\lambda$, define
\begin{align*}
g(\lambda) &= 1+(z-1)\sum_{k=1}^\lambda \frac{(2^k,h)}{2^{2k}} = \begin{cases}
1+(z-1)\Bigl(1-\frac13(2^{1-\nu_2(h)}+2^{\nu_2(h)-2\lambda}) \Bigr), 
&\text{if }\nu_2(h)<\lambda,\\
1+(z-1)(1-2^{-\lambda}), &\text{if }\nu_2(h)\geq\lambda
\end{cases} \\
g(\infty) &= \lim_{\lambda\to\infty} g(\lambda) = 1+(z-1)( 1-2^{1-\nu_2(h)}/3).
\end{align*}
(Note that $g(\infty)$ appears as a denominator in equation~\eqref{secretly g(infty)}.) Also define
\[
G_h(\newa,\newb,\eta,\kappa,z)=\frac{1}{g(\infty)} \begin{cases}
(-1)^\kappa g(\newb)+\kappa, &\text{if }\eta=\newa=0,\\
(-1)^\kappa (g(\newb)-g(\newa-1)), &\text{if }\eta\leq \newa \text{ and }\newa>0,\\
g(\eta-2)-g(\eta-1)+(-1)^\kappa \left(g(\newb)-g(\eta-1)\right), &\text{if }\eta> \newa \text{ and }\eta\leq \newb+1, \\
0, &\text{otherwise.}
\end{cases}
\]
\end{defn}

\begin{lemma} \label{lemma main term omega of minus}
Let~$\newa$ be an integer and let $\newb\in\{\alpha,\alpha+1,\alpha+2,\ldots\}\cup\{\infty\}$. Fix $\eta\in\{0,2,3\}$ and $\kappa\in\{0,1\}$. For any any positive integer~$h$ and any odd squarefree integer~$\gamma$, write 
\begin{align*}
 S_h(\newa,\newb,\gamma,\eta,\kappa,z)=\sum\limits_{\substack{\ell \ge1\\ \gamma\mid\ell\\ \newa\leq\nu_2(\ell)\leq \newb}}(z-1)^{\omega(\ell)}\frac{(\ell,h)}{\ell\varphi(\ell)} \sum_{\substack{m\mid \ell\\ 2^\kappa\mid m\\ \frac{2^{\eta}}{(2^\eta,\ell)}\mid m}}\frac{\mu(m)}{m}
\end{align*}
Then 
\begin{equation*}
  S_h(\newa,\newb,\gamma,\eta,\kappa,z)=\prod\limits_{\substack{q}}\biggl(1+\frac{z-1}{q^2-1}\biggr)H_h^{\omega-}(z)I_h^{\omega-}(\gamma,z)G_h(\newa,\newb,\eta,\kappa,z)
\end{equation*}
in the notation of Definitions~\ref{definition generating function omega minus omega} and~\ref{definition main term omega of minus}.
\end{lemma}

\begin{proof}
If $\eta-\newb\geq 2$ then trivially $S_h(\newa,\newb,\gamma,\eta,\kappa,z)=G_h(\newa,\newb,\eta,\kappa,z)=0$. We may therefore assume that $\eta-\newb\leq 1$, and we write $\theta(\ell)=\max\{\kappa, \eta-\nu_2(\ell)\}$.

First we compute the sum over divisors of $\ell$:
\begin{align*}
 \sum_{\substack{m\mid \ell\\ 2^\kappa\mid m\\ \frac{2^{\eta}}{(2^\eta,\ell)}\mid m}}\frac{\mu(m)}{m} = \sum_{\substack{m\mid \ell\\ 2^{\theta(\ell)}\mid m}}\frac{\mu(m)}{m}&= (-1)^{\theta(\ell)}\mathbbm{1}\Bigl( \theta(\ell)\leq 1 \Bigr)\mathbbm{1}\Bigl(2^{\theta(\ell)}\mid \ell\Bigr)\prod\limits_{\substack{q\mid \ell}}\biggr(1-\frac{1}{q}\biggr).
\end{align*}
Now $\theta(\ell)>1$ if $\eta-\nu_2(\ell)>1$, so we split our problem according to whether $\eta\leq \newa$ or $\newa<\eta\leq \newb+1$.
First, if $\eta\leq \newa$, then
\begin{align*}
  S_h(\newa,\newb,\gamma,\eta,\kappa,z)&=(-1)^{\kappa}\sum\limits_{\substack{\ell\ge1\\ \gamma\mid\ell \\ \newa\leq\nu_2(\ell)\leq \newb}}\mathbbm{1}(2^{\kappa}\mid \ell)(z-1)^{\omega(\ell)}\frac{(\ell,h)}{\ell\varphi(\ell)}\prod\limits_{\substack{q\mid \ell}}\biggl(1-\frac{1}{q}\biggr)\\
  &=(-1)^{\kappa}\sum\limits_{\substack{\ell\ge1\\ \gamma\mid\ell \\ \newa\leq\nu_2(\ell)\leq \newb}}(z-1)^{\omega(\ell)}\frac{(\ell,h)}{\ell\varphi(\ell)}\prod\limits_{\substack{q\mid \ell}}\biggl(1-\frac{1}{q}\biggr)\\
  &\qquad{}+\mathbbm{1}\Bigl(\kappa=1,\newa=0\Bigr)\sum\limits_{\substack{\ell\ge1\\ \gamma\mid\ell \\ 2\nmid \ell}}(z-1)^{\omega(\ell)}\frac{(\ell,h)}{\ell\varphi(\ell)}\prod\limits_{\substack{q\mid \ell}}\biggl(1-\frac{1}{q}\biggr).
\end{align*}
On the other hand, if $\eta> \newa$ and $\eta \leq \newb+1$, then
\begin{align*}
  S_h(\newa,\newb,\gamma,\eta,\kappa,z)&=\sum\limits_{\substack{\ell\ge1 \\ \gamma\mid\ell\\ \newa\leq\nu_2(\ell)\leq \newb}}(z-1)^{\omega(\ell)}\frac{(\ell,h)}{\ell\varphi(\ell)}(-1)^{\theta(\ell)}\mathbbm{1}\Bigl(2^{\theta(\ell)}\mid \ell\Bigr)\prod\limits_{\substack{q\mid \ell}}\biggl(1-\frac{1}{q}\biggr)\\
  &=-\mathbbm{1}(\eta\geq2)\sum\limits_{\substack{\ell\ge1 \\ \gamma\mid\ell\\ \nu_2(\ell)=\eta-1}}(z-1)^{\omega(\ell)}\frac{(\ell,h)}{\ell\varphi(\ell)}\prod\limits_{\substack{q\mid \ell}}\biggl(1-\frac{1}{q}\biggr)\\
  &\qquad{}+\mathbbm{1}(\eta\leq \newb)(-1)^{\kappa}\sum\limits_{\substack{\ell\ge1 \\ \gamma\mid\ell\\ \eta\leq\nu_2(\ell)\leq \newb}}(z-1)^{\omega(\ell)}\frac{(\ell,h)}{\ell\varphi(\ell)}\prod\limits_{\substack{q\mid \ell}}\biggl(1-\frac{1}{q}\biggr).
  \end{align*}
In both cases, the sum can be expressed in terms of the function
\begin{align*}
 M(\lambda)&=\sum\limits_{\substack{\ell\ge1\\ \gamma\mid\ell \\ \nu_2(\ell)\leq \lambda}}(z-1)^{\omega(\ell)}\frac{(\ell,h)}{\ell\varphi(\ell)}\prod\limits_{\substack{q\mid \ell}}\biggl(1-\frac{1}{q}\biggr)\\
 &=\biggl(1+(z-1)\sum\limits_{k=1}^\lambda \frac{(2^k,h)}{2^{2k}}\biggr)\frac{(z-1)^{\omega(\gamma)}(\gamma,h)}{\gamma^2}\sum\limits_{\substack{\ell\ge1 \\2\nmid\ell}}(z-1)^{\omega(\ell)-\omega((\ell,\gamma))}\frac{(\ell,\frac{h}{(h,\gamma)})}{\ell^2}\\
  &=g(\lambda)\frac{(z-1)^{\omega(\gamma)}(\gamma,h)}{\gamma^2}\prod_{q\nmid 2\gamma}\biggl(1+(z-1)\sum\limits_{k=1}^{\infty}\frac{(q^k,h)}{q^{2k}}\biggr)\prod_{q\mid \gamma}\biggl(1+\sum\limits_{k=1}^{\infty}\frac{(q^k,\frac{h}{(h,q)})}{q^{2k}}\biggr);
\end{align*}
more precisely,
\begin{align*}
 S_h(\newa,\newb,\gamma,\eta,\kappa,z)&=\begin{cases}
(-1)^\kappa M(\newb)+\kappa M(0),&\text{if }\eta=\newa=0,\\
(-1)^\kappa (M(\newb)-M(\newa)),&\text{if } \eta\leq \newa \text{ and } \newa>0,\\
-(M(\eta-1)-M(\eta-2))+(-1)^\kappa (M(\newb)-M(\eta-1)),&\text{if } \eta> \newa \text{ and } \eta\leq \newb+1,\\
0,&\text{otherwise}
\end{cases} \\
&=\frac{(z-1)^{\omega(\gamma)}(\gamma,h)}{\gamma^2g(\infty)}\prod\limits_{\substack{q}}\left(1+(z-1)\sum\limits_{k=1}^{\infty}\frac{(q^k,h)}{q^{2k}}\right)\\
 &\qquad{}\times\prod\limits_{\substack{q\mid \gamma}}\left[\left(1+(z-1)\sum\limits_{k=1}^{\infty}\frac{(q^k,h)}{q^{2k}}\right)^{-1}\left(1+\sum\limits_{k=1}^{\infty}\frac{(q^k,\frac{h}{(h,q)})}{q^{2k}}\right)\right]\\
 &\qquad{}\times\begin{cases}
(-1)^\kappa g(\newb)+\kappa, &\text{if }\eta=\newa=0,\\
(-1)^\kappa (g(\newb)-g(\newa-1)),&\text{if }\eta\leq \newa \text{ and }\newa>0,\\
g(\eta-2)-g(\eta-1)+(-1)^\kappa \left(g(\newb)-g(\eta-1)\right),&\text{if }\eta> \newa \text{ and }\eta\leq \newb+1,\\
0,&\text{otherwise.}
\end{cases}
\end{align*}
After calculating the eventually geometric series
\begin{align*}
 \sum\limits_{k=1}^{\infty}\frac{(p^k,h)}{p^{2k}}&=\frac{p+1-p^{-\nu_p(h)+1}}{p^2-1},
\end{align*}
we obtain
\begin{align*}
 S_0(\newa,\newb,\gamma,\eta,\kappa,z)&=\frac{(z-1)^{\omega(\gamma)}(\gamma,h)}{\gamma g(\infty)}\prod\limits_{\substack{q}}\biggl(1+\frac{z-1}{q^2-1}\biggr)\prod\limits_{\substack{q|h}}\left(\frac{q^2-1+(z-1)(q+1-q^{-\nu_q(h)+1})}{q^2+z-2}\right)\\
 &\qquad{}\times\prod\limits_{\substack{q\mid \gamma}}\left(\frac{q+1-q^{-\nu_q(h/(h,q))}}{q^2-1+(z-1)(q+1-q^{-\nu_q(h)+1})}\right)\\
 &\qquad{}\times\begin{cases}
                (-1)^\kappa g(\newb)+\kappa, &\text{if }\eta=\newa=0,\\
                (-1)^\kappa (g(\newb)-g(\newa-1)),&\text{if }\eta\leq \newa \text{ and }\newa>0,\\
                g(\eta-2)-g(\eta-1)+(-1)^\kappa \left(g(\newb)-g(\eta-1)\right),&\text{if }\eta> \newa \text{ and }\eta\leq \newb+1,\\
                0,&\text{otherwise}
               \end{cases}
\end{align*}
which completes the proof of the lemma.
\end{proof}

The following proposition immediately implies Theorem~\ref{theorem generating function omega minus omega} when combined with Proposition~\ref{proposition all done but the main terms}.

\begin{prop}\label{proposition main term omega minus omega}
Let $a\in\Q\setminus\{-1,0,1\}$. For $|z|\le1$,
\[ \sum\limits_{\ell\ge1}(z-1)^{\omega(\ell)}\frac{(\ell,h)}{\ell\varphi(\ell)}\sum\limits_{m\mid \ell}\varepsilon_a(m\ell,\ell)\frac{\mu(m)}{m}=\prod\limits_{q}\biggl(1+\frac{z-1}{q^2-1}\biggr)H_h^{\omega-}(z)F_a^{\omega-}(h,z),\]
in the notation of Definition~\ref{definition generating function omega minus omega}.
\end{prop}

\begin{proof}
In the notation of Proposition~\ref{proposition degree}, define
\[
S_1 = \sum\limits_{\ell\ge1}(z-1)^{\omega(\ell)}\frac{(\ell,h)}{\ell\varphi(\ell)}\sum\limits_{m\mid \ell}(\varepsilon_a(m\ell,\ell)-1)\frac{\mu(m)}{m},
\]
so that
\begin{align*}
\sum_{\ell=1}^\infty (z-1)^{\omega(\ell)}\sum\limits_{m\mid \ell} \frac{\mu(m)}{[\Q\left(\sqrt[\ell]{a},\zeta_{m\ell}\right):\Q]} &= \sum\limits_{\ell\ge1}(z-1)^{\omega(\ell)}\frac{(\ell,h)}{\ell\varphi(\ell)}\sum\limits_{m\mid \ell}\frac{\mu(m)}{m} + S_1 \\
&= S_h(0,\infty,1,0,0,z) + S_1 = H_h^{\omega-}(z) \prod\limits_{\substack{q}}\biggl(1+\frac{z-1}{q^2-1}\biggr) + S_1
\end{align*}
by Lemma~\ref{lemma main term omega of minus}, since $I^{\omega-}_h(1,z) = 1$. The first term on the right-hand side accounts for the initial~$1$ in the definition~\eqref{secretly g(infty)} of $F_a^\Omega(h,z)$; we now show that~$S_1$ accounts for the complicated remainder of that definition.

First assume $a>0$. By Proposition~\ref{proposition degree},~$S_1$ is nonzero only when $2\mid\ell'$ and $\mathfrak{d}(a_0)\mid m\ell$. Write $\mathfrak{d}(a_0)=2^\eta \gamma$ where~$\gamma$ is odd and squarefree. Then
 \begin{align*}
  S_1&=S_h(\nu_2(h)+1,\infty,\gamma,\eta,0,z)=\prod\limits_{\substack{q}}\biggl(1+\frac{z-1}{q^2-1}\biggr)H_h^{\omega-}(z)I_h^{\omega-}(\gamma,z)G_h(\nu_2(h)+1,\infty,\eta,0,z)\\
  &=\prod\limits_{\substack{q}}\biggl(1+\frac{z-1}{q^2-1}\biggr)\frac{H_h^{\omega-}(z)I_h^{\omega-}(\gamma,z)}{g(\infty)}
  \begin{cases}
g(\infty)-g(\nu_2(h)),&\text{if }\eta\leq \nu_2(h)+1,\\
g(\eta-2)-2g(\eta-1)+ g(\infty),&\text{if }\eta> \nu_2(h)+1
\end{cases}\\
    &=\prod\limits_{\substack{q}}\biggl(1+\frac{z-1}{q^2-1}\biggr)\frac{H_h^{\omega-}(z)I_h^{\omega-}(\gamma,z)(z-1)}{g(\infty)}
  \begin{cases}
{2^{-\nu_2(h)}}/{3},&\text{if }\eta\leq \nu_2(h)+1,\\
\displaystyle\sum\limits_{k=\eta}^\infty \frac{(2^k,h)}{2^{2k}}-\frac{(2^{\eta-1},h)}{2^{2\eta-2}},&\text{if }\eta> \nu_2(h)+1
\end{cases}\\
      &=\prod\limits_{\substack{q}}\biggl(1+\frac{z-1}{q^2-1}\biggr)\frac{H_h^{\omega-}(z)I_h^{\omega-}(\gamma,z)(z-1)}{g(\infty)}
  \begin{cases}
{2^{-\nu_2(h)}}/{3},&\text{if }\eta\leq \nu_2(h)+1,\\
-{2^{\nu_2(h)-2\eta+3}}/{3},&\text{if }\eta> \nu_2(h)+1
\end{cases}\\
        &=\prod\limits_{\substack{q}}\biggl(1+\frac{z-1}{q^2-1}\biggr)\frac{H_h^{\omega-}(z)I_h^{\omega-}(\gamma,z)(z-1)}{3g(\infty)}
  \begin{cases}
-2^{-1},&\text{if }b_0\equiv 3\mod{4} \text{ and }2\nmid h,\\
-2^{-3},&\text{if }b_0\equiv 2\mod{4} \text{ and }2\nmid h,\\
-2^{-2},&\text{if }b_0\equiv 2\mod{4} \text{ and }2\mid\mid h,\\
2^{-\nu_2(h)},&\text{otherwise.}
\end{cases}
 \end{align*}
 
 Assume now that $a<0$. We split $S_1 = S_{1,1} + S_{1,2} + S_{1,3}$ into three parts according to the $2$-adic valuation of $\ell$, to fit the different cases of Proposition~\ref{proposition degree}:
 \begin{align*}
  S_{1,1}&=\sum\limits_{\substack{\ell\ge1\\\nu_2(\ell)\leq\nu_2(h)}}(z-1)^{\omega(\ell)}\frac{(\ell,h)}{\ell\varphi(\ell)}\sum\limits_{m\mid \ell}(\varepsilon_a(m\ell,\ell)-1)\frac{\mu(m)}{m}\\
  S_{1,2}&=\sum\limits_{\substack{\ell\ge1\\\nu_2(\ell)=\nu_2(h)+1}}(z-1)^{\omega(\ell)}\frac{(\ell,h)}{\ell\varphi(\ell)}\sum\limits_{m\mid \ell}(\varepsilon_a(m\ell,\ell)-1)\frac{\mu(m)}{m}\\
  S_{1,3}&=\sum\limits_{\substack{\ell\ge1\\\nu_2(\ell)\geq\nu_2(h)+2}}(z-1)^{\omega(\ell)}\frac{(\ell,h)}{\ell\varphi(\ell)}\sum\limits_{m\mid \ell}(\varepsilon_a(m\ell,\ell)-1)\frac{\mu(m)}{m}.
 \end{align*}
 Starting with $S_{1,1}$, we use Proposition~\ref{proposition degree} to write:
 \begin{align*}
  S_{1,1}&=\frac{-1}{2}\sum\limits_{\substack{\ell\ge1\\ 2\mid\ell \\\nu_2(\ell)\leq\nu_2(h)}}(z-1)^{\omega(\ell)}\frac{(\ell,h)}{\ell\varphi(\ell)}\sum\limits_{\substack{m\mid \ell\\ 2\nmid m}}\frac{\mu(m)}{m}\\
  &=\frac{-\mathbbm{1}(2\mid h)}{2}\left(S_h(1,\nu_2(h),1,0,0,z)-S_h(1,\nu_2(h),1,0,1,z)\right)\\
  &=\prod\limits_{\substack{q}}\biggl(1+\frac{z-1}{q^2-1}\biggr)\frac{H_h^{\omega-}(z)(z-1)}{g(\infty)}(2^{-\nu_2(h)}-1).
 \end{align*}

Next we turn to $S_{1,2}$. Note that $\gamma={b_0}/{(2,b_0)}$ is an odd squarefree number. Choose $\eta_+,\eta_-,\eta_2\in\{0,2,3\}$ such that $\mathfrak{d}(a_0)=2^{\eta_+}\gamma$, $\mathfrak{d}(-a_0)=2^{\eta_-}\gamma$, and $\mathfrak{d}(2a_0)=2^{\eta_2}\gamma$, where $\mathfrak{d}$ is as in Definition~\ref{s and d def}.
Then, according to Proposition~\ref{proposition degree}, we can split $S_{1,2}$ into three parts as follows:
\begin{align*}
S_{1,2,1}&=\sum\limits_{\substack{\ell\ge1\\ \gamma\mid\ell\\\nu_2(\ell)=\nu_2(h)+1}}(z-1)^{\omega(\ell)}\frac{(\ell,h)}{\ell\varphi(\ell)}\sum\limits_{\substack{m\mid \ell\\ 2\mid m\\ \frac{2^{\eta_+}}{(2^{\eta_+},\ell)}\mid m}}\frac{\mu(m)}{m},\\
S_{1,2,2}&=\sum\limits_{\substack{\ell\ge1\\ \gamma\mid\ell\\\nu_2(\ell)=\nu_2(h)+1\\ \nu_2(\ell)=1}}(z-1)^{\omega(\ell)}\frac{(\ell,h)}{\ell\varphi(\ell)}\sum\limits_{\substack{m\mid \ell\\ 2\nmid m\\ \frac{2^{\eta_-}}{(2^{\eta_-},\ell)}\mid m}}\frac{\mu(m)}{m},\\
S_{1,2,3}&=\sum\limits_{\substack{\ell\ge1\\ \gamma\mid\ell\\\nu_2(\ell)=\nu_2(h)+1\\ \nu_2(\ell)=2}}(z-1)^{\omega(\ell)}\frac{(\ell,h)}{\ell\varphi(\ell)}\sum\limits_{\substack{m\mid \ell\\ 2\nmid m\\ \frac{2^{\eta_2}}{(2^{\eta_2},\ell)}\mid m}}\frac{\mu(m)}{m}.
\end{align*}
These sums can be computed using Lemma~\ref{lemma main term omega of minus}. For $S_{1,2,1}$, we obtain
\begin{align*}
 S_{1,2,1}&=S_h(\nu_2(h)+1,\nu_2(h)+1,\gamma,\eta_+,1,z)\\
 &=\prod\limits_{\substack{q}}\biggl(1+\frac{z-1}{q^2-1}\biggr)H_h^{\omega-}(z)I_h^{\omega-}(\gamma,z)G_h(\nu_2(h)+1,\nu_2(h)+1,\eta_+,1,z)\\
 &=\prod\limits_{\substack{q}}\biggl(1+\frac{z-1}{q^2-1}\biggr)\frac{H_h^{\omega-}(z)I_h^{\omega-}(\gamma,z)(z-1)}{g(\infty)}\begin{cases}
- 2^{-\nu_2(h)-2},&\text{if }\eta_+\leq \nu_2(h)+2,\\
0,&\text{otherwise}\\
 \end{cases}\\
  &=-\mathbbm{1}(\nu_2(h)\geq \eta_+-2)\prod\limits_{\substack{q}}\biggl(1+\frac{z-1}{q^2-1}\biggr)\frac{H_h^{\omega-}(z)I_h^{\omega-}(\gamma,z)(z-1)}{2^{\nu_2(h)+2}g(\infty)}.
\end{align*}
For $S_{1,2,2}$ we obtain
\begin{align*}
 S_{1,2,2}&=\mathbbm{1}(2\nmid h)\left(S_h(1,1,\gamma,\eta_-,0,z)-S_h(1,1,\gamma,\eta_-,1,z)\right)\\
 &=\mathbbm{1}(2\nmid h)\prod\limits_{\substack{q}}\biggl(1+\frac{z-1}{q^2-1}\biggr)H_h^{\omega-}(z)I_h^{\omega-}(\gamma,z)\left(G_h(1,1,\eta_-,0,z)-G_h(1,1,\eta_-,1,z)\right)\\
  &=\mathbbm{1}(2\nmid h)\prod\limits_{\substack{q}}\biggl(1+\frac{z-1}{q^2-1}\biggr)\frac{H_h^{\omega-}(z)I_h^{\omega-}(\gamma,z)}{g(\infty)}\begin{cases}
\frac{z-1}{2},&\text{if }\eta_-\leq 1,\\
0,&\text{otherwise}
\end{cases}\\
  &=\mathbbm{1}(2\nmid h,\eta_-=0)\prod\limits_{\substack{q}}\biggl(1+\frac{z-1}{q^2-1}\biggr)\frac{H_h^{\omega-}(z)I_h^{\omega-}(\gamma,z)(z-1)}{2g(\infty)}.
\end{align*}
Finally, for $S_{1,2,3}$ we obtain
\begin{align*}
S_{1,2,3}&=\mathbbm{1}(2\mid\mid h)\left(S_h(2,2,\gamma,\eta_2,0,z)-S_h(2,2,\gamma,\eta_2,1,z)\right)\\
 &=\mathbbm{1}(2\mid\mid h)\prod\limits_{\substack{q}}\biggl(1+\frac{z-1}{q^2-1}\biggr)H_h^{\omega-}(z)I_h^{\omega-}(\gamma,z)\left(G_h(2,2,\eta_2,0,z)-G_h(2,2,\eta_2,1,z)\right)\\
 &=\mathbbm{1}(2\mid\mid h)\prod\limits_{\substack{q}}\biggl(1+\frac{z-1}{q^2-1}\biggr)\frac{H_h^{\omega-}(z)I_h^{\omega-}(\gamma,z)}{g(\infty)}\begin{cases}
\frac{z-1}{4},&\text{if }\eta_2\leq 2,\\
0,&\text{otherwise}
  \end{cases}\\
 &=\mathbbm{1}(2\mid\mid h,\eta_2\leq 2)\prod\limits_{\substack{q}}\biggl(1+\frac{z-1}{q^2-1}\biggr)\frac{H_h^{\omega-}(z)I_h^{\omega-}(\gamma,z)(z-1)}{4g(\infty)}.
\end{align*}
Adding these last three evaluations together yields
\begin{equation*}
 S_{1,2}=\prod\limits_{\substack{q}}\biggl(1+\frac{z-1}{q^2-1}\biggr)\frac{H_h^{\omega-}(z)I_h^{\omega-}(\gamma,z)(z-1)}{g(\infty)}\begin{cases}
 0,&\text{if }\eta_+=3,\ 2\nmid h,\\
 2^{-3},&\text{if }\eta_+=3,\ 2\mid\mid h,\\
  2^{-2},&\text{if }\eta_+=2,\ 2\nmid h,\\
-2^{-\nu_2(h)-2},&\text{otherwise.}
  \end{cases}
\end{equation*}

Last, we compute that
\begin{align*}
 S_{1,3}&=\sum\limits_{\substack{\ell\ge1\\ \gamma\mid\ell\\ \nu_2(\ell)\geq\nu_2(h)+2}}(z-1)^{\omega(\ell)}\frac{(\ell,h)}{\ell\varphi(\ell)}\sum\limits_{\substack{m\mid \ell\\ \frac{2^{\eta_+}}{(2^{\eta_+},\ell)}\mid m}}\frac{\mu(m)}{m}=S_h(\nu_2(h)+2,\infty,\gamma,\eta_+,0,z)\\
 &=\prod\limits_{\substack{q}}\biggl(1+\frac{z-1}{q^2-1}\biggr)\frac{H_h^{\omega-}(z)I_h^{\omega-}(\gamma,z)}{g(\infty)}\begin{cases}
g(\infty)-2g({\eta_+}-1)+g({\eta_+}-2),&\text{if }\eta_+> \nu_2(h)+2,\\
g(\infty)-g(\nu_2(h)+1),&\text{if }\eta_+\leq \nu_2(h)+2
\end{cases}\\
 &=\prod\limits_{\substack{q}}\biggl(1+\frac{z-1}{q^2-1}\biggr)\frac{H_h^{\omega-}(z)I_h^{\omega-}(\gamma,z)(z-1)}{g(\infty)}\begin{cases}
\sum\limits_{k=3}^\infty {2^{=2k}}-2^{-4},&\text{if }\eta_+=3 \text{ and }2\nmid h,\\
2^{\nu_2(h)}\sum\limits_{k=\nu_2(h)+2}^\infty {2^{-2k}},&\text{otherwise}
\end{cases}\\
 &=\prod\limits_{\substack{q}}\biggl(1+\frac{z-1}{q^2-1}\biggr)\frac{H_h^{\omega-}(z)I_h^{\omega-}(\gamma,z)(z-1)}{g(\infty)}\begin{cases}
-2^{-3}/3,&\text{if }\eta_+=3 \text{ and }2\nmid h,\\
{2^{-\nu_2(h)-2}}/{3},&\text{otherwise.}
\end{cases}
\end{align*}
In total, we have found that
\begin{equation*}
 S_1=\prod\limits_{\substack{q}}\biggl(1+\frac{z-1}{q^2-1}\biggr)\frac{H_h^{\omega-}(z)(z-1)}{g(\infty)}\biggl(2^{-\nu_2(h)}-1+\frac{I_h^{\omega-}(\gamma,z)}{3}\tau_-^{\omega-}(a,z)\biggr),
\end{equation*}
where
\begin{align*}
\tau_-^{\omega-}(a,z)&=3\times\begin{cases}
-2^{-\nu_2(h)-2}+{2^{-\nu_2(h)-2}}/{3},&\text{if }b_0\equiv1\mod{4}\\
2^{-2}+{2^{-2}}/{3},&\text{if }b_0\equiv3\mod{4} \text{ and }2\nmid h\\
-2^{-\nu_2(h)-2}+{2^{-\nu_2(h)-2}}/{3},&\text{if }b_0\equiv3\mod{4} \text{ and }2\mid h\\
\frac{-1}{3\times 2^3},&\text{if }b_0\equiv2\mod{4} \text{ and }2\nmid h\\
2^{-3}+{2^{-3}}/{3},&\text{if }b_0\equiv2\mod{4} \text{ and }2\mid\mid h\\
-2^{-\nu_2(h)-2}+{2^{-\nu_2(h)-2}}/{3},&\text{if }b_0\equiv2\mod{4} \text{ and }4\mid h\\
 \end{cases}\\
 &=\begin{cases}
1,&\text{if }b_0\equiv3\mod{4} \text{ and }2\nmid h\\
-2^{-3},&\text{if }b_0\equiv2\mod{4} \text{ and }2\nmid h\\
2^{-1},&\text{if }b_0\equiv2\mod{4} \text{ and }2\mid\mid h\\
-2^{-\nu_2(h)-1},&\text{otherwise.}
 \end{cases}
\end{align*}
This calculation completes the proof of the proposition.
\end{proof}

\section{Numerical values} \label{numerical section}

Theorems~\ref{theorem generating function omega of quotient}, \ref{theorem generating function Omega of quotient}, and~\ref{theorem generating function omega minus omega} establish the densities of the sets of primes for which $\ord_p(a)$ has a particular relationship with $p-1$.
For example, $D_a^{\omega/}(n)$ is the density of the set of primes~$p$ with $\omega((p-1)/\ord_p(a)) = n$; in light of
equations~\eqref{defining coefficients} and~\eqref{definition_D_a(n)}, we can restate 
Theorem~\ref{theorem generating function omega of quotient} as
\[
\sum_{n=0}^\infty D_a^{\omega/}(n) z^n = F^{\omega/}_a(z) H^{\omega/}_a(z) \prod_q \biggl(1+\frac{z-1}{q(q-1)}\biggr).
\]
In particular, this identity provides us with a way to extract information about the individual densities $D_a^{\omega/}(n)$ from the formula for their generating function on the right-hand side.
The first two factors are simply rational functions of~$z$ that depend on~$a$, while the third factor is an infinite product that takes some attention to understand. We therefore define ``uncorrected'' coefficients in Notation~\ref{f P D notation} below so that, for example,
\[
\sum_{n=0}^\infty D^{\omega/}(n) z^n = \prod_q \biggl(1+\frac{z-1}{q(q-1)}\biggr)
\]
and concentrate on them initially.

In Section~\ref{subsection Marcus Lai} we provide explicit recursive formulas for these uncorrected coefficients. In Section~\ref{subsection numerical data} we use these recursive formulas to calculate several examples of numerical values, both for the uncorrected coefficients and the true coefficients such as $D_a^{\omega/}(n)$. In
Section~\ref{subsection non-recursive} we record a non-recursive formula for the uncorrected coefficients, in the hope that readers might find it useful to adapt to other situations.

\subsection{Derivatives of infinite products and recursive formulas for their coefficients} \label{subsection Marcus Lai}

This approach we take in this section is based on observations of Marcus Lai (private communication), and a version of it appears in~\cite{EM} as well.

\begin{prop} \label{recursive general}
Let $P(z)$ be any function with Maclaurin series
\[
P(z) = \sum_{n=0}^\infty C(n) z^n,
\]
so that $C(n) = P^{(n)}(0)/n!$ for every $n\ge0$. Define
\[
S(0,z) = \log P(z) \quad\text{and}\quad S(n,z) = \frac{d^n}{dz^n} S(0,z)
\]
for all $n\ge1$; and define $S(n) = S(n,0)$ and $\tilde S(n) = S(n)/(n-1)!$. Then for any $n\ge1$,
\begin{align*}
P^{(n)}(z) &= \sum_{k=0}^{n-1} \binom{n-1}k P^{(k)}(z) S(n-k,z).
\end{align*}
In particular, for $n\ge1$,
\[
C(n) = \frac1n \sum_{k=0}^{n-1} C(k) \tilde S(n-k).
\]
\end{prop}

\begin{proof}
We first verify that
\[
P'(z) = P(z)\frac{P'(z)}{P(z)} = P(z) \frac d{dz} \log P(z) = P(z) \frac d{dz} S(0,z) = P(z) S(1,z),
\]
which is the case $n=1$ of the first identity. The general case of the first identity now follows from using the product rule $n-1$ times in a row on this initial identity $P'(z) = P(z) S(1,z)$. The second identity follows by plugging in $z=0$ to the first and using the definition $C(n) = \frac1{n!} P^{(n)}(0)$.
\end{proof}

We now apply this result to the (uncorrected) infinite products appearing in Theorems~\ref{theorem generating function omega of quotient}, \ref{theorem generating function Omega of quotient}, and~\ref{theorem generating function omega minus omega}. We caution the reader not to conflate the functions we now define, such as $f_{\omega/}(q,z)$, with earlier similarly named functions such as $f^{\omega/}(z)$ from Definition~\ref{definition generating function omega of quotient}.

\begin{notn} \label{f P D notation}
Define
\[
f_\Omega(q,z) = \frac{q^3-q^2-q+z}{(q-1)(q^2-z)}
\quad\text{and}\quad
f_{\omega/}(q,z) = \bigg( 1 + \frac{z-1}{q(q-1)} \bigg)
\quad\text{and}\quad
f_{\omega-}(q,z) = \bigg( 1 + \frac{z-1}{q^2-1} \bigg).
\]
Then the infinite products appearing in Theorems~\ref{theorem generating function omega of quotient}, \ref{theorem generating function Omega of quotient}, and~\ref{theorem generating function omega minus omega} are
\[
P_\Omega(z) = \prod_q f_\Omega(q,z)
\quad\text{and}\quad
P_{\omega/}(z) = \prod_q f_{\omega/}(q,z)
\quad\text{and}\quad
P_{\omega-}(z) = \prod_q f_{\omega-}(q,z).
\]
We are interested in the coefficients of their Maclaurin series, which we name by
\[
P_\Omega(z) = \sum_{n=0}^\infty D^\Omega(n) z^n
\quad\text{and}\quad
P_{\omega/}(z) = \sum_{n=0}^\infty D^{\omega/}(n) z^n
\quad\text{and}\quad
P_{\omega-}(z) = \sum_{n=0}^\infty D^{\omega-}(n) z^n,
\]
so that
\[
D^\Omega(n) = \frac1{n!} P^{(n)}_\Omega(z)
\quad\text{and}\quad
D^{\omega/}(n) = \frac1{n!} P^{(n)}_{\omega/}(z)
\quad\text{and}\quad
D^{\omega-}(n) = \frac1{n!} P^{(n)}_{\omega-}(z).
\]
(These Maclaurin coefficients are closely related to the ones defined in equation~\eqref{defining coefficients}, except that those prior coefficients also include the effect of the correction factors depending on~$a$.)
Note that
\begin{equation} \label{D at 0}
D^{\omega-}(0) = \prod_q \biggl( 1 - \frac1{q^2-1} \biggr),
\quad\text{and that }
D^\Omega(0) = D^{\omega/}(0) = \prod_q \biggl( 1 - \frac1{q(q-1)} \biggr)
\end{equation}
both equal Artin's constant; Moree~\cite{MoreeWebpage} provides precise numerical approximation of these constants.
\end{notn}

We derive recursive formulas (Proposition~\ref{calc the cs} below) for $D^\Omega(n)$, $D^{\omega/}(n)$, and $D^{\omega-}(n)$ for all $n\ge1$, given in terms of certain quantities $\tilde S_\Omega(n)$, $\tilde S_{\omega/}(n)$, and $\tilde S_{\omega-}(n)$ defined in equation~\eqref{tilde S defns} and evaluated in Lemma~\ref{S values}.

\begin{notn} \label{notation including S tilde}
Define
\[
g_\Omega(0,q,z) = \log f_\Omega(q,z)
\quad\text{and}\quad
g_{\omega/}(0,q,z) = \log f_{\omega/}(q,z)
\quad\text{and}\quad
g_{\omega-}(0,q,z) = \log f_{\omega-}(q,z),
\]
and for $n\ge1$ define
\[
g_\Omega(n,q,z) = \frac{d^n}{dz^n} g_\Omega(0,q,z)
\quad\text{and}\quad
g_{\omega/}(n,q,z) = \frac{d^n}{dz^n} g_{\omega/}(0,q,z)
\quad\text{and}\quad
g_{\omega-}(n,q,z) = \frac{d^n}{dz^n} g_{\omega-}(0,q,z).
\]
Then define $g_\Omega(n,q) = g_\Omega(n,q,0)$ and $g_{\omega/}(n,q) = g_{\omega/}(n,q,0)$ and $g_{\omega-}(n,q) = g_{\omega-}(n,q,0)$.

Further define
\[
S_\Omega(n,z) = \sum_q g_\Omega(n,q,z)
\quad\text{and}\quad
S_{\omega/}(n,z) = \sum_q g_{\omega/}(n,q,z)
\quad\text{and}\quad
S_{\omega-}(n,z) = \sum_q g_{\omega-}(n,q,z),
\]
so that $S_\Omega(0,z) = \log P_\Omega(z)$ and $S_{\omega/}(0,z) = \log P_{\omega/}(z)$ and $S_{\omega-}(0,z) = \log P_{\omega-}(z)$, and also
\[
S_\Omega(n,z) = \frac{d^n}{dz^n} S_\Omega(0,z)
\quad\text{and}\quad
S_{\omega/}(n,z) = \frac{d^n}{dz^n} S_{\omega/}(0,z)
\quad\text{and}\quad
S_{\omega-}(n,z) = \frac{d^n}{dz^n} S_{\omega-}(0,z).
\]
Finally, define $S_\Omega(n) = S_\Omega(n,0) = \sum_q g_\Omega(n,q)$ and $S_{\omega/}(n) = S_{\omega/}(n,0) = \sum_q g_{\omega/}(n,q)$ and $S_{\omega-}(n) = S_{\omega-}(n,0) = \sum_q g_{\omega-}(n,q)$; and define
\begin{equation} \label{tilde S defns}
\tilde S_\Omega(n) = \frac1{(n-1)!} S_\Omega(n)
\quad\text{and}\quad
\tilde S_{\omega/}(n) = \frac1{(n-1)!} S_{\omega/}(n)
\quad\text{and}\quad
\tilde S_{\omega-}(n) = \frac1{(n-1)!} S_{\omega-}(n).
\end{equation}
\end{notn}

\begin{prop} \label{calc the cs}
For any $n\ge1$,
\begin{align*}
P^{(n)}_\Omega(z) &= \sum_{k=0}^{n-1} \binom{n-1}k P^{(k)}_\Omega(z) S_\Omega(n-k,z) \\
P^{(n)}_{\omega/}(z) &= \sum_{k=0}^{n-1} \binom{n-1}k P^{(k)}_{\omega/}(z) S_{\omega/}(n-k,z) \\
P^{(n)}_{\omega-}(z) &= \sum_{k=0}^{n-1} \binom{n-1}k P^{(k)}_{\omega-}(z) S_{\omega-}(n-k,z).
\end{align*}
In particular, for $n\ge1$,
\[
D^\Omega(n) = \frac1n \sum_{k=0}^{n-1} D^\Omega(k) \tilde S_\Omega(n-k)
\text{ and }
D^{\omega/}(n) = \frac1n \sum_{k=0}^{n-1} D^{\omega/}(k) \tilde S_{\omega/}(n-k)
\text{ and }
D^{\omega-}(n) = \frac1n \sum_{k=0}^{n-1} D^{\omega-}(k) \tilde S_{\omega-}(n-k).
\]
\end{prop}

\begin{proof}
These are particular instances of Proposition~\ref{recursive general}.
\end{proof}

The values of these coefficients for $n=0$ are recorded in equation~\eqref{D at 0}; to use these recursive formulas for the coefficients, it remains to exhibit formulas for the~$\tilde S$ quantities.

\begin{lemma} \label{S values}
For $n\ge1$,
\begin{align*}
g_\Omega(n,q,z) &= \frac{(n-1)!}{(q^2-z)^n} - \frac{(n-1)!}{(q+q^2-q^3-z)^n} \\
g_{\omega/}(n,q,z) &= -\frac{(n-1)!}{(1+q-q^2-z)^n} \\
g_{\omega-}(n,q,z) &= -\frac{(n-1)!}{(2-q^2-z)^n}.
\end{align*}
In particular,
\[
\tilde S_\Omega(n) = \sum_q \biggl( \frac1{q^{2n}} - \frac1{(q+q^2-q^3)^n} \biggr) \text{\, and \,}
\tilde S_{\omega/}(n) = - \sum_q \frac1{(1+q-q^2)^n} \text{\, and \,}
\tilde S_{\omega-}(n) = - \sum_q \frac1{(2-q^2)^n}.
\]
\end{lemma}

\begin{proof}
The primary calculation is to verify that
\begin{align*}
g_\Omega(1,q,z) &= \frac d{dz} g_\Omega(0,q,z) = \frac1{q^2-z} - \frac1{q+q^2-q^3-z} \\
g_{\omega/}(1,q,z) &= \frac d{dz} g_{\omega/}(0,q,z) = -\frac1{1+q-q^2-z} \\
g_{\omega-}(1,q,z) &= \frac d{dz} g_{\omega-}(0,q,z) = -\frac1{2-q^2-z};
\end{align*}
the rest follow trivially by induction on~$n$. The expressions for the $\tilde S$ quantities are direct consequences of Notation~\ref{notation including S tilde}.
\end{proof}

\subsection{Numerical calculation of the coefficients} \label{subsection numerical data}

In this section we record some numerical values of the Maclaurin coefficients defined in equation~\eqref{defining coefficients}, conditional on GRH. The main part of these calculations is computing the ``uncorrected'' coefficients $D^\Omega(n)$, $D^{\omega/}(n)$, and $D^{\omega-}(n)$ using Proposition~\ref{calc the cs}, which uses the formulas for $\tilde S_\Omega(n)$, $\tilde S_{\omega/}(n)$, and $\tilde S_{\omega-}(n)$ given in Lemma~\ref{S values}. In computing approximations to these infinite sums, we used the first $10^6$ primes in the case $n=1$ and the first $10^5$ primes in the cases $n\ge2$, which gave us sufficient precision to report the decimal places appearing in Figure~\ref{D values figure}.
(As mentioned earlier, extremely precise values of $D^{\omega-}(0)$, $D^\Omega(0)$, and $D^{\omega/}(0)$ have been computed by Moree~\cite{MoreeWebpage}.) The final row of Figure~\ref{D values figure}, labeled~$E$, records the expectation of the sequence in question. For example, the bottom-right entry denotes the expectation of the random variable that takes the value~$n$ with probability $D^{\omega-}(n)$ for each $n\ge0$, which is equal to $\sum_{n=0}^\infty nD^{\omega-}(n)$. See Section~\ref{subsection expectations} for more details about these expectations.

\begin{figure}[ht]
$\displaystyle
\begin{array}{c|ccc}
n & D^\Omega(n) & D^{\omega/}(n) & D^{\omega-}(n) \\ \hline
0 & 0.373955 & 0.373955 & 0.530711 \\
1 & 0.387002 & 0.489828 & 0.391986 \\
2 & 0.167049 & 0.125687 & 0.072349 \\
3 & 0.052465 & 0.010164 & 0.004800 \\
4 & 0.014466 & 0.000356 & 0.000147 \\
5 & 0.003774 & 0.000006 & 0.000002 \\
\hline
E & 0.96337 & 0.77315 & 0.55169
\end{array}
$
\caption{Coefficients and expectations for the ``uncorrected'' infinite products}\label{D values figure}
\end{figure}

In light of Theorems~\ref{theorem generating function omega of quotient}, \ref{theorem generating function Omega of quotient}, and~\ref{theorem generating function omega minus omega}, the coefficients in equation~\eqref{defining coefficients} can be written as
\begin{align*}
 \sum_{n=0}^\infty D_a^{\omega/}(n)z^n&=F^{\omega/}_a(z) H^{\omega/}_a(z) \prod_q \biggl( 1 + \frac{z-1}{q(q-1)} \biggr)\\
 \sum_{n=0}^\infty D_a^{\Omega}(n)z^n&=F^{\Omega}_a(z)H^{\Omega}_a(z)\prod\limits_{\substack{q}}\biggl(1+\frac{(z-1)q}{(q-1)(q^2-z)}\biggr)\\
 \sum_{n=0}^\infty D_a^{\omega-}(n)z^n &= F^{\omega-}_a(z)H^{\omega-}_a(z)\prod\limits_{\substack{q}}\biggl(1+\frac{z-1}{q^2-1}\biggr).
\end{align*}
We have already calculated the Maclaurin coefficients of these infinite products in Figure~\ref{D values figure}. Since the correction factors~$F_a$ and~$H_a$ are simply rational functions, it is an easy matter to use those previous calculations to determine the Maclaurin coefficients of these right-hand sides for any particular value of~$a$. Figures~\ref{D3 values figure}--\ref{D5 values figure} provide three example sets of coefficients with $a=3,4,5$. (In Figure~\ref{D4 values figure}, note that the values $D_4^\Omega(0) = D_4^{\omega/}(0) = 0$ reflect the fact that the perfect square $a=4$ cannot be a primitive root modulo any prime; however, it is still possible for $\omega(\ord_p(4))$ to equal $\omega(p-1)$, which is reflected in the positive value for $D_4^{\omega-}(0)$.) Again the final rows labeled~$E$ record the expectations of these three statistics; for example, the bottom-right entry of Figure~\ref{D3 values figure} records the average value of $\omega(p-1) - \omega(\ord_p(3))$ as $p$ ranges over all primes exceeding~$3$, which is equal to $\sum_{n=0}^\infty nD_3^{\omega-}(n)$.

\begin{rmk}
We calculated the coefficients in Figures~\ref{D values figure}--\ref{D5 values figure} for values of~$n$ beyond $n=5$; while the associated computational errors became large enough (relative to the small coefficients) that we didn't want to report possibly incorrect decimal places, we did observe that the $\Omega$ column decayed significantly less rapidly than the other two columns in each case, consistent with Corollary~\ref{decay rate corollary}.
\end{rmk}

\begin{figure}[ht]
$\displaystyle
\begin{array}{c|ccc}
n & D_3^\Omega(n) & D_3^{\omega/}(n) & D_3^{\omega-}(n) \\ \hline
0 & 0.373955 & 0.373955 & 0.511757 \\
1 & 0.405700 & 0.489828 & 0.428079 \\
2 & 0.138409 & 0.125687 & 0.056962 \\
3 & 0.056421 & 0.010164 & 0.003112 \\
4 & 0.018447 & 0.000356 & 0.000085 \\
5 & 0.005215 & 0.000006 & 0.000001 \\
\hline
E & 0.96337 & 0.77315 & 0.55169
\end{array}
$
\caption{Coefficients and expectations for the case $a=3$}\label{D3 values figure}
\end{figure}

\begin{figure}[ht]
$\displaystyle
\begin{array}{c|ccc}
n & D_4^\Omega(n) & D_4^{\omega/}(n) & D_4^{\omega-}(n) \\ \hline
0 & 0 & 0 & 0.331694 \\
1 & 0.560933 & 0.747911 & 0.543517 \\
2 & 0.253293 & 0.231746 & 0.116448 \\
3 & 0.122297 & 0.019629 & 0.008083 \\
4 & 0.045042 & 0.000700 & 0.000252 \\
5 & 0.013501 & 0.000012 & 0.000004 \\
\hline
E & 1.71337 & 1.27315 & 0.80169
\end{array}
$
\caption{Coefficients and expectations for the case $a=4$}\label{D4 values figure}
\end{figure}

\begin{figure}[ht]
$\displaystyle
\begin{array}{c|ccc}
n & D_5^\Omega(n) & D_5^{\omega/}(n) & D_5^{\omega-}(n) \\ \hline
0 & 0.393637 & 0.393637 & 0.542249 \\
1 & 0.357959 & 0.455527 & 0.371163 \\
2 & 0.169510 & 0.135494 & 0.079483 \\
3 & 0.056907 & 0.014732 & 0.006858 \\
4 & 0.016216 & 0.000596 & 0.000241 \\
5 & 0.004294 & 0.000011 & 0.000004 \\
\hline
E & 0.96337 & 0.77315 & 0.55169
\end{array}
$
\caption{Coefficients and expectations for the case $a=5$}\label{D5 values figure}
\end{figure}

\subsection{A non-recursive formula} \label{subsection non-recursive}

Note that Proposition~\ref{recursive general} gives a recursive formula for the $C(n)$ in terms of $C(0)$ and the $\tilde S(n)$. We take this opportunity to record an explicit (non-recursive) version, which can be derived using Fa\`a di Bruno's formula generalizing the chain rule.

\begin{prop} \label{non recursive general}
For $n\ge1$,
\[
C(n) = P(0) \sum_{\substack{0\le m_1,m_2,\dots,m_n\le n \\ m_1+2m_2+\cdots+nm_n = n}} \frac{S(1)^{m_1}S(2)^{m_2}\cdots S(n)^{m_n}}{m_1!m_2!\cdots m_n! \cdot 2^{m_2}3^{m_3}\cdots n^{m_n}}.
\]
\end{prop}

\begin{proof}
Define the complete exponential Bell polynomial
\[
B_n(x_1,\dots,x_n) = \sum_{\substack{0\le m_1,m_2,\dots,m_n\le n \\ m_1+2m_2+\cdots+nm_n = n}} \frac{n!}{m_1!m_2!\cdots m_n!} \biggl( \frac{x_1}{1!} \biggr)^{m_1} \biggl( \frac{x_2}{2!} \biggr)^{m_2} \cdots \biggl( \frac{x_n}{n!} \biggr)^{m_n},
\]
as well as the incomplete exponential Bell polynomials
\[
B_{n,k}(x_1,\dots,x_n) = \sum_{\substack{0\le m_1,m_2,\dots,m_n\le n \\ m_1+2m_2+\cdots+nm_n = n \\ m_1+m_2+\cdots+m_n=k}} \frac{n!}{m_1!m_2!\cdots m_n!} \biggl( \frac{x_1}{1!} \biggr)^{m_1} \biggl( \frac{x_2}{2!} \biggr)^{m_2} \cdots \biggl( \frac{x_n}{n!} \biggr)^{m_n},
\]
so that $B_n(x_1,\dots,x_n) = \sum_{k=1}^n B_{n,k}(x_1,\dots,x_n)$. We can check that $B_{n,k}(x_1,\dots,x_n)$ depends only on its first $n-k+1$ variables.

We begin with Fa\`a di Bruno's formula for the high-order derivatives of a composition of two functions:
\[
\frac{d^n}{dz^n} h\bigl( g(z) \bigr) = \sum_{k=1}^n h^{(k)}(g(z)) B_{n,k}\bigl( g'(z),g''(z),\dots,g^{(n)}(z) \bigr).
\]
When applied with $h(u)=e^u$, this becomes
\[
\frac{d^n}{dz^n} e^{g(z)} = \sum_{k=1}^n e^{g(z)} B_{n,k}\bigl( g'(z),g''(z),\dots,g^{(n)}(z) \bigr) = e^{g(z)} B_n\bigl( g'(z),g''(z),\dots,g^{(n)}(z) \bigr).
\]
We apply this with $g(z) = S(0,z)$, so that $e^{g(z)} = P(z)$ and $g^{(k)}(z) = S(k,z)$:
\[
\frac{d^n}{dz^n} P(z) = P(z) B_n\bigl( S(1,z), S(2,z), \dots, S(n,z) \bigr).
\]
In particular, since $S(k,z) = (k-1)! \tilde S(k)$,
\begin{align*}
\frac{d^n}{dz^n} P(z) \biggr|_{z=0} &= P(0) B_n\bigl( S(1), S(2), \dots, S(n) \bigr) \\
&= P(0) \sum_{\substack{0\le m_1,m_2,\dots,m_n\le n \\ m_1+2m_2+\cdots+nm_n = n}} \frac{n!}{m_1!m_2!\cdots m_n!} \biggl( \frac{\tilde S(1)}{1} \biggr)^{m_1} \biggl( \frac{\tilde S(2)}{2} \biggr)^{m_2} \cdots \biggl( \frac{\tilde S(n)}{n} \biggr)^{m_n}
\end{align*}
(without the factorials in the denominators), and the proposition follows.
\end{proof}

\section{Decay rates and expectations}

Our generating functions allow us to gather even more information about the coefficients $D_a^{\omega/}(n)$, $D_a^\Omega(n)$, and $D_a^{\omega-}(n)$. The decay rates for these sequences of coefficients, as described in Corollary~\ref{decay rate corollary}, are proved in Section~\ref{decay section}; the expectations of those sequences, which can be thought of as the average values of the statistics $\omega\Bigl( (p-1)/\ord_p(a) \Bigr)$ and $\Omega\Bigl( (p-1)/\ord_p(a) \Bigr)$ and $\omega(p-1) - \omega\Bigl( \ord_p(a) \Bigr)$ when~$p$ is chosen at random, are determined in Section~\ref{subsection expectations}.

\subsection{Decay rates} \label{decay section}

We begin with a precise statement of the well-known phenomenon that the size of the Maclaurin coefficients of an analytic function is determined by the distance from the origin to the nearest singularity of the function.

\begin{lemma} \label{two poles lemma}
Let $F(z) = \sum_{n=0}^\infty a_n z^n$ be a meromorphic function on a neighbourhood of the closed disk of radius~$R$ centred at the origin. Suppose that there are two complex numbers~$z_1$ and~$z_2$ with $0<|z_1|<|z_2|=R$ such that $F(z)$ has simple poles at $z=z_1$ and $z=z_2$, with residues~$r_1$ and~$r_2$ respectively, and no other poles in $\{|z|\le R\}$. Then $a_n = -r_1 z_1^{-n-1} + O_F( R^{-n} )$.
\end{lemma}

\begin{proof}
If we define $G(z) = F(z) - r_1/(z-z_1) - r_2/(z-z_2)$, then $G(z)$ is analytic on a neighbourhood of $\{|z|\le R\}$. When we write $G(z) = \sum_{n=0}^\infty b_n z^n$, the standard Cauchy integral formula calculation gives
\[
b_n = \frac{G^{(n)}}{n!} = \frac1{2\pi i} \oint_{|z|=R} \frac{G(z)}{z^{n+1}} \,dz \ll_G \oint_{|z|=R} \frac1{R^{n+1}} \,d|z| \ll_G R^{-n}.
\]
But $F(z) = G(z) - r_1/z_1(1-z/z_1) - r_2/z_2(1-z/z_2)$ means that
\[
\sum_{n=0}^\infty a_n z^n = - \frac{r_1}{z_1} \sum_{n=0}^\infty \biggl( \frac z{z_1} \biggr)^n - \frac{r_2}{z_2} \sum_{n=0}^\infty \biggl( \frac z{z_2} \biggr)^n + \sum_{n=0}^\infty b_n z^n,
\]
and therefore
\[
a_n = - \frac{r_1}{z_1} \biggl( \frac 1{z_1} \biggr)^n - \frac{r_2}{z_2} \biggl( \frac 1{z_2} \biggr)^n + b_n = -\frac{r_1}{z_1^{n+1}} + O_F(R^{-{n+1}}) + O_F(R^{-n})
\]
as claimed.
\end{proof}

This lemma is relevant to the generating function appearing in Theorem~\ref{theorem generating function Omega of quotient}, whose poles we take some care to describe exactly.

\begin{lemma} \label{4 and 9}
The main term in equation~\eqref{equation generating function Omega of quotient} is a meromorphic function on all of~$\C$, with at most simple poles at $z=q^2$ for every prime~$q$ and no other poles.
\end{lemma}

\begin{proof}
We refer freely to the notation from Definition~\ref{definition generating function Omega of quotient}. The main term in question includes the product
\[
\prod_q \biggl( 1 + \frac{(z-1)q}{(q-1)(q^2-z)} \biggr) = \prod_q \biggl( 1 + \frac{z-q}{q^2(q-1)} \biggr) \bigg/ \prod_q \biggl( 1 - \frac{z}{q^2} \biggr).
\]
Since $\sum_q 1/q(q-1)$ and $\sum_q 1/q^2$ converge, both products converge absolutely and uniformly on compact sets and are therefore entire functions; therefore their quotient is a meromorphic function on all of~$\C$ with at most simple poles at $z=q^2$ for all primes~$q$. We also note that
\begin{align}
H^\Omega_h(z) & \prod_q \biggl( 1 + \frac{(z-1)q}{(q-1)(q^2-z)} \biggr) = H^\Omega_h(z) \prod_q \frac{q^3-q^2-q+z}{(q-1)(q^2-z)} \notag \\
&= \prod_{q\mid h} \biggl( \frac{q^4-2q^3+2qz-z^2-(z-1)(q-1)z^{\nu_q(h)}q^{-\nu_q(h)+2}}{(q-z)(q-1)(q^2-z)} \biggr) \prod_{q\nmid h} \frac{q^3-q^2-q+z}{(q-1)(q^2-z)} \label{schwartz's}
\end{align}
has no additional poles---the factor $q-z$ in the first denominator on the right-hand side does not produce a pole since the corresponding numerator vanishes at $z=q$.

Next we consider what happens when equation~\eqref{schwartz's} is multiplied by the factor $I_h^\Omega(\gamma,z)$ that is present in both cases of $F_a^\Omega(\gamma,z)$. Again both its numerator and denominator vanish to first order at $z=q$, so no pole will be introduced there. For any $q\mid\gamma$, if $q\mid h$ as well then the denominator of $I_h^\Omega(\gamma,z)$ is canceled by the identical numerator of $H_h^\Omega(\gamma,z)$; on the other hand, if $q\nmid h$ then the denominator of $I_h^\Omega(\gamma,z)$ equals $(q-z)(q^3-q^2-q+z)$, and the second factor is canceled by the second numerator on the right-hand side of equation~\eqref{schwartz's}.

We now consider what happens when equation~\eqref{schwartz's} is multiplied by the first fraction in the definition of $F_a^\Omega(h,z)$. When $2\mid h$, its denominator is identical to the numerator of the $q=2$ term in the definition of $H_h^\Omega(z)$ and is thus canceled. (Note that~$2$ never divides~$\gamma$ by definition, since~$b_0$ is squarefree, and so we are not using this numerator to cancel two different denominators.) On the other hand, when $2\nmid h$ then its denominator equals $4-z^2=-(z-2)(z+2)$; the $z-2$ is canceled by the numerator above it, while the $z+2$ is canceled by the numerator $2^3-2^2-2+z=2+z$ of the $q=2$ term in the infinite product in equation~\eqref{schwartz's}.

We are almost done examining the various ways that poles might have been introduced when multplying equation~\eqref{schwartz's} by $F_a^\Omega(h,z)$ to obtain the main term in equation~\eqref{equation generating function Omega of quotient}. The functions $\tau_\pm^\Omega(a,z)$ are entire in all cases. Both occurrences of $1/(4-z)$ in the definition of $F_a^\Omega(h,z)$ are canceled by the $z-4$ factor in the large fraction. Finally, the term $(1-(\frac z2)^{\nu_2(h)})/(z-2)$ in the second case of $F_a^\Omega(h,z)$ has a removable singularity at $z=2$. These observations collectively establish the lemma.
\end{proof}

\begin{proof}[Proof of Corollary~\ref{decay rate corollary}(a)]
By Lemma~\ref{4 and 9}, we may apply Lemma~\ref{two poles lemma} to the main term in equation~\eqref{equation generating function Omega of quotient}, whose Maclaurin coefficients are the $D_a^\Omega(n)$. The conclusion is that $D_a^\Omega(n) = -r_1\cdot 4^{-n-1} + O_a(9^{-n})$, where~$r_1$ is the residue of the simple pole at $z=4$ of the main term in equation~\eqref{equation generating function Omega of quotient}. The infinite product in that main term is
\[
\prod_q \biggl( 1 + \frac{q(z-1)}{(q-1)(q^2-z)} \biggr) = -\frac{2+z}{z-4} \prod_{q\ge3} \biggl( 1 + \frac{q(z-1)}{(q-1)(q^2-z)} \biggr),
\]
and so its residue at $z=4$ equals
\[
-6 \prod_{q\ge3} \biggl( 1 + \frac{3q}{(q-1)(q^2-4)} \biggr).
\]
While $F_a^\Omega(h,z)$ has a removable singularity at $z=4$, it can be checked that in every case its limiting value there equals $1+I_h^\Omega(\gamma,4)$. Therefore the residue of the main term in equation~\eqref{equation generating function Omega of quotient} at $z=4$ equals
\[
H_h^\Omega(4) \Bigl( 1+I_h^\Omega(\gamma,4) \Bigr) \cdot -6 \prod_{q\ge3} \biggl( 1 + \frac{3q}{(q-1)(q^2-4)} \biggr),
\]
which completes the proof.
\end{proof}

As it turns out, the generating functions in Theorems~\ref{theorem generating function omega of quotient} and~\ref{theorem generating function omega minus omega} are entire functions, which necessitates a different approach to determining the size of their Maclaurin coefficients.

\begin{lemma} \label{growth rate lemma}
The main terms in equations~\eqref{equation generating function omega of quotient} and~\eqref{equation generating function omega minus omega} are entire functions, and their logarithms are $\asymp_a z^{1/2}\log z$ as $z\to\infty$ through real numbers.
\end{lemma}

\begin{proof}
We include the proof only for the main term in equation~\eqref{equation generating function omega of quotient}; the proof for the main term in equation~\eqref{equation generating function omega minus omega} follows exactly the same structure (but with messier algebraic details).

First note that since $\sum_q 1/q(q-1)$ converges, the product
\begin{equation} \label{omega/ product}
\prod_q \biggl( 1 + \frac{z-1}{q(q-1)} \biggr) = \prod_q \frac{z+q^2-q-1}{q(q-1)}
\end{equation}
converges absolutely and uniformly on compact sets and is therefore an entire function of~$z$. Note also that
\begin{equation} \label{H omega/ product}
H^{\omega/}_a(z) \prod_q \biggl(1+\frac{z-1}{q(q-1)}\biggr) = \frac{z+1}2 \prod_{\substack{q\mid h\\ q\neq 2}} \frac{z+q-2}{q-1} \prod_{q\nmid 2h} \frac{z+q^2-q-1}{q(q-1)}.
\end{equation}
is also entire. In the case $2\mid h$, we then have
\begin{equation} \label{CH omega/ product even h}
F^{\omega/}_a(z) H^{\omega/}_a(z) \prod_q \biggl(1+\frac{z-1}{q(q-1)}\biggr) = 
\prod_{\substack{q\mid h\\ q\neq 2}} \frac{z+q-2}{q-1} \prod_{q\nmid 2h} \frac{z+q^2-q-1}{q(q-1)}
\begin{cases}
z, &\text{if } a>0, \\
(z+1)/2, &\text{if } a<0,
\end{cases}
\end{equation}
which remains entire. In the case $2\nmid h$, we have
\begin{multline} \label{CH omega/ product odd h}
F^{\omega/}_a(z) H^{\omega/}_a(z) \prod_q \biggl(1+\frac{z-1}{q(q-1)}\biggr) \\
= H^{\omega/}_a(z) \prod_q \biggl(1+\frac{z-1}{q(q-1)}\biggr) +
\begin{cases}
0 , &\text{if } \sgn(a)b_0 \not\equiv 1\mod4, \\
f^{\omega/}_a(z) H^{\omega/}_a(z) \prod_q \bigl(1+\frac{z-1}{q(q-1)}\bigr) , &\text{if } \sgn(a)b_0 \equiv 1\mod4;
\end{cases}
\end{multline}
we have already seen in equation~\eqref{H omega/ product} that the first summand is entire. Therefore it remains only to examine
\begin{multline} \label{fH omega/ product odd h}
f^{\omega/}_a(z) H^{\omega/}_a(z) \prod_q \biggl(1+\frac{z-1}{q(q-1)}\biggr) \\
= \frac{z-1}2
\prod_{q\mid (b_0,h)} \frac{z-1}{q-1}
\prod_{\substack{q\mid h \\ q\nmid b_0}} \frac{z+q-2}{q-1}
\prod_{\substack{q\mid b_0 \\ q\nmid 2h}} \frac{z-1}{q(q-1)}
\prod_{q\nmid 2hb_0} \frac{z+q^2-q-1}{q(q-1)}
\end{multline}
which is also entire.

In addition, when~$z$ is real and large, the logarithm of the product in equation~\eqref{omega/ product} is $\asymp z^{1/2}/\log z$ by~\cite[Proposition~3.2]{EM}. Since the other terms are all rational functions of~$z$, the logarithms of their values are $\ll \log z$ when~$z$ is large and real, which does not change the growth rate of the right-hand side.
\end{proof}

\begin{proof}[Proof of Corollary~\ref{decay rate corollary}(b)]
It is a classical fact from complex analysis (see~\cite[Definition~2.1.1 and Theorem~2.2.2]{Boas}) that the rate of growth of an entire function $f(z) = \sum_{m=0}^\infty b_m z^m$ is connected to the decay rate of its Maclaurin coefficients~$b_m$ by
\[
\limsup_{r\to \infty} \frac{\log\log \bigl( \max_{|z|=r} |f(z)| \bigr)}{\log r} = \limsup_{\substack{m\to \infty \\ b_m\ne0}} \frac{m\log m}{\log(1/|b_m|)}
\]
(both sides, if finite, equal the ``order'' of $f(z)$). It is straightforward to deduce the proposition from this identity and the growth rates in Lemma~\ref{growth rate lemma}.
(See~\cite[Proof of Corollary~4]{EM} for the details of deriving such an upper bound for the coefficients of an analogous function, including the argument that the growth rate for~$z$ real and positive is indeed the maximal growth rate on the circle $|z|=r$.)
\end{proof}

\subsection{Expectations}\label{subsection expectations}

Our functions $D_a^{\omega/}(n;x)$, $D_a^{\omega-}(n;x)$, and $D_a^{\Omega}(n;x)$ (and their limits $D_a^{\omega/}(n)$, $D_a^{\omega-}(n)$, and $D_a^{\Omega}(n)$ if they exist) represent proportions of primes having certain properties. We can view these proportions as probabilities associated with discrete random variables, and study the expectations of these random variables.

\begin{notn}
Write
\begin{align*}
\mathbb{E}_{a}^{\omega/}(x)&=\sum_{n=0}^\infty nD_a^{\omega/}(n;x),
\end{align*}
which is the expectation of random variable $X$ with Prob$(X=n)=D_a^{\omega/}(n;x)$ for all $n\ge0$.
(Technically, for this to be a random variable, in equation~\eqref{defining coefficients} we should be dividing not by $\pi(x)$ but by the number of primes $p\le x$ for which $\nu_p(a)=0$, which is $\pi(x) + O_a(1)$; but we shall ignore this technicality as it is insignificant in the limit.)
Similarly define
\begin{align*}
\mathbb{E}_{a}^{\omega-}(x)=\sum_{n=0}^\infty nD_a^{\omega-}(n;x)
\quad\text{and}\quad
\mathbb{E}_{a}^{\Omega}(x)=\sum_{n=0}^\infty nD_a^{\Omega}(n;x).
\end{align*}
Furthermore, if the appropriate limits exist, define
\begin{align*}
\mathbb{E}_{a}^{\omega/}=\sum_{n=0}^\infty nD_a^{\omega/}(n),
\quad
\mathbb{E}_{a}^{\omega-}=\sum_{n=0}^\infty nD_a^{\omega-}(n),
\quad\text{and}\quad
\mathbb{E}_{a}^{\Omega}=\sum_{n=0}^\infty nD_a^{\Omega}(n).
\end{align*}
\end{notn}

\begin{thm} \label{Theorem conditional expectations}
Under GRH, the expectations $\mathbb{E}_{a}^{\omega/}$, $\mathbb{E}_{a}^{\omega-}$, and $\mathbb{E}_{a}^{\Omega}$ exist, and
\begin{align*}
\mathbb{E}_{a}^{\omega/}&=\frac{dF^{\omega/}_a}{dz}(1) + \frac{dH^{\omega/}_a}{dz}(1)+\sum\limits_q\frac{1}{q(q-1)} \\
\mathbb{E}_{a}^{\omega-}&=\frac{dF^{\omega-}_a}{dz}(1) + \frac{dH^{\omega-}_a}{dz}(1)+\sum\limits_q\frac{1}{q^2-1} \\
\mathbb{E}_{a}^{\Omega}&=\frac{dF^{\Omega}_a}{dz}(1) + \frac{dH^{\Omega}_a}{dz}(1)+\sum\limits_q\frac{q}{q^3-q^2-q+1}.
\end{align*}
\end{thm}

\begin{proof}
In this proof we use $D_a$ and $\mathbb{E}_a$ to examine the three cases simultaneously.
The formulas in Theorems~\ref{theorem generating function omega of quotient}, \ref{theorem generating function Omega of quotient}, and~\ref{theorem generating function omega minus omega} imply that
\begin{equation*}
    \sum\limits_n D_a(n)z^n=F_a(z)H_a(z)\prod\limits_q P_a(q,z)
\end{equation*}
for certain functions $P_a(p,z)$ depending on the statistic being examined; therefore 
\begin{equation*}
\mathbb{E}_a = \biggl( \frac{d}{dz}F_a(z)H_a(z)\prod\limits_q P_a(q,z) \biggr) \bigg|_{z=1}.
\end{equation*}
But $F_a(1)=H_a(1)=P_a(p,1)=1$ always, and so the product rule yields
\begin{equation*}
\mathbb{E}_a=\frac{dF_a}{dz}(1) + \frac{dH_a}{dz}(1)+ \biggl( \frac{d}{dz}\prod\limits_q P_a(q,z) \biggr) \bigg|_{z=1},
\end{equation*}
where the last derivative can be easily computed in each of our cases using Lemma~\ref{S values}.
\end{proof}

We apply this theorem to the special cases treated in Corollaries~\ref{corollary generating function omega of quotient a=3,4,5}, \ref{corollary generating function Omega of quotient a=3,4,5}, and~\ref{corollary generating function omega minus omega a=3,4,5}.

\begin{cor}
Assuming GRH,
\begin{align*}
\mathbb{E}_{3}^{\omega/} = \mathbb{E}_{5}^{\omega/} = \sum_q \frac1{q(q-1)}
&\quad\text{and}\quad
\mathbb{E}_{4}^{\omega/} = \frac12 + \sum_q \frac1{q(q-1)} \\
\mathbb{E}_{3}^{\Omega} = \mathbb{E}_{5}^{\Omega} = \sum_q \frac{q}{q^3-q^2-q+1}
&\quad\text{and}\quad
\mathbb{E}_{4}^{\Omega} = \frac34 + \sum_q \frac{q}{q^3-q^2-q+1} \\
\mathbb{E}_{3}^{\omega-} = \mathbb{E}_{5}^{\omega-} = \sum_q \frac1{q^2-1}
&\quad\text{and}\quad
\mathbb{E}_{4}^{\omega-} = \frac14 + \sum_q \frac1{q^2-1}.
\end{align*}
\end{cor}

\section{Unconditional bounds}

Our conditional results can actually be converted into unconditional bounds, of which Corollary~\ref{unconditional inequality corollary k=0} is an example. We concentrate on the first statistic $\omega\Bigl( (p-1)/\ord_p(a) \Bigr)$, but all our results have immediate analogues for the other two statistics $\Omega\Bigl( (p-1)/\ord_p(a) \Bigr)$ and $\omega(p-1) - \omega\Bigl( \ord_p(a) \Bigr)$, with the same proofs.

In the notation of Definition~\ref{definition generating function omega of quotient}, define the Maclaurin series coefficients $\tilde D_a^{\omega/}(n)$ by
\begin{equation*}
\sum_{n=0}^\infty \tilde D_a^{\omega/}(n) z^n = F^{\omega/}_a(z) H^{\omega/}_a(z) \prod_q \biggl(1+\frac{z-1}{q(q-1)}\biggr),
\end{equation*}
so that Theorem~\ref{theorem generating function omega of quotient} is the statement (conditional on GRH) that $\lim_{x\to\infty} D_a^{\omega/}(n;x) = \tilde D_a^{\omega/}(n)$ for all $n\ge0$. We can modify the proof to establish:

\begin{thm} \label{welcome both lim sup and tilde notation. sigh}
Unconditionally, for all $k\ge0$,
\[
\limsup_{x\to\infty} \sum_{n=0}^k D_a^{\omega/}(n;x) \le \sum_{n=0}^k \tilde D_a^{\omega/}(n).
\]
\end{thm}

\noindent Note that the first assertion of Corollary~\ref{unconditional inequality corollary k=0} is the special case $k=0$.

\begin{thm}\label{Theorem unconditional lower bound expectation}
Unconditionally,
\begin{align*}
\liminf_{x\to\infty} \sum\limits_{n=0}^\infty nD_a^{\omega/}(n;x)&\ge\frac{dF^{\omega/}_a}{dz}(1) + \frac{dH^{\omega/}_a}{dz}(1)+\sum\limits_q\frac{1}{q(q-1)},
\end{align*}
where $H^{\omega/}_a(z)$ and $F^{\omega/}_a(z)$ are as in Definition~\ref{definition generating function omega of quotient}.
\end{thm}

\noindent The analogues of these two theorems for the other statistics appear in Theorems~\ref{welcome both lim sup and tilde notation. no sigh} and~\ref{Corollary unconditional lower bound expectation} below.

To approach these theorems, we define
\begin{equation*}
\omega_\xi(n)=\#\{p\mid n\colon p\leq\xi\}
\end{equation*}
so that $\omega_\xi(n)$ is an increasing function of~$\xi$ and $\omega_n(n) = \omega(n)$.
For all $k\in\N$ and all $\xi\le n$, we thus have the trivial inequality
\begin{align}
\#\biggl\{p\leq x\colon \omega\biggl(\frac{p-1}{\ord_p(a)}\biggr)\le k\biggr\} &\le \#\biggl\{p\leq x\colon \omega_\xi\biggl(\frac{p-1}{\ord_p(a)}\biggr)\le k\biggr\}.
\label{eqn_remarkunconditional2}
\end{align}
Analogously to \eqref{defining coefficients}, we write
\begin{equation} \label{D a xi omega quotient def}
\frac{1}{\pi_a(x)}\sum_{\substack{p\le x \\ \nu_p(a)=0}} z^{\omega_\xi((p-1)/\ord_p(a))} = \sum\limits_{n=0}^\infty D_{a,\xi}^{\omega/}(n;x) z^n,
\end{equation}
so that the inequality~\eqref{eqn_remarkunconditional2} becomes
\begin{equation} \label{cumulative xi inequality2}
\sum_{n=0}^k D_a^{\omega/}(n;x) \le \sum_{n=0}^k D_{a,\xi}^{\omega/}(n;x)
\end{equation}
for any $x\ge\xi\ge1$ and $k\in\N$.

To prove Theorem~\ref{welcome both lim sup and tilde notation. sigh}, therefore, it suffices to show that for a suitable function $\xi=\xi(x)$,
\begin{equation} \label{oral exams take long}
\lim_{x\to\infty} D_{a,\xi}^{\omega/}(n;x) = \tilde D_a^{\omega/}(n)
\end{equation}
for all $n\ge0$. We accomplish this goal in the next section. These tools can also be used to prove Theorem~\ref{Theorem unconditional lower bound expectation}, which we do in Section~\ref{Unconditional lower bound for the expectation}.

\subsection{Unconditional asymptotics for suitable $\xi$}

We have already seen (in the proof of Theorem~\ref{thm_asymptoticprimesatisfyingstuff}) that versions of the prime ideal theorem can be seen as special cases of the Chebotarev density theorem. In particular, we may adapt an unconditional Chebotarev density theorem by Lagarias and Odlyzko~\cite{LagariasOdlyzko} to obtain the following prime ideal theorem:

\begin{thm} \label{Lagarias, Odlyzko}
Let~$K$ be a number field with discriminant~$\Delta_K$ and $[K:\Q] = n_K$. There exist absolute positive constants~$c_1$ and~$c_2$ such that the prime ideal counting function $\pi_K(x)$ satisfies
\[
\pi_K(x) = \pi(x) + \bo\Bigl(\pi(x^{1-1/(4\log|\Delta_K|)}) + \pi(x^{1-c_1|\Delta_k|^{-1+1/n_K}}) + x\exp(-c_2n_K^{-1/2}(\log x)^{1/2})\Bigr)
\]
uniformly in~$K$.
\end{thm}

\noindent
We can then derive an unconditional version of Theorem~\ref{thm_asymptoticprimesatisfyingstuff}:

\begin{thm}\label{thm_asymptoticprimesatisfyingstuff_unconditional}
Let $a\in\Q\setminus\{-1,0,1\}$, and let $m$ and $\ell$ be positive integers. Unconditionally,
\begin{align*}
\#\Big\{ &p\le x \colon \nu_p(a)=0,\, \ell\mid (p-1)/\ord_p(a),\ m\ell\mid (p-1)\Big\} =\frac{\pi(x)}{[\Q(\sqrt[\ell]{a},\zeta_m):\Q]} \\
&\qquad{}+\bo\biggl(\frac{\log\log m\ell}{m\ell^2}\biggl(\pi(x^{1-1/(c_3 m\ell^2\log m\ell)})+\pi(x^{1-c_1(m\ell)^{c_4(1-m\ell^2)}})+x\exp\biggl(-\frac{c_2(\log x)^{1/2}}{m^{1/2}\ell}\biggr)\biggr)\biggr),
\end{align*}
where the implied constant depends only on~$a$.
\end{thm}

\begin{proof}
The proof is the same as the proof of Theorem~\ref{thm_asymptoticprimesatisfyingstuff}, except that in place of equation~\eqref{equation_googd_GRH2} we instead use
\[
\pi_{\Q(\sqrt[\ell]{a},\zeta_{m\ell})}(x) = \pi(x) + \bo\Bigl(\pi(x^{1-1/(c_3 m\ell^2\log m\ell)}) + \pi(x^{1-c_1m^{c_4(1-m\ell^2)}}) + x\exp(-c_2(m\ell^2)^{-1/2}(\log x)^{1/2})\Bigr)
\]
which follows from Theorem~\ref{Lagarias, Odlyzko} once we use our upper bound (Proposition~\ref{propdiscriminant}) for the discriminant of $\Q(\sqrt[\ell]{a},\zeta_{m\ell})$.
\end{proof}

\begin{cor} \label{prop_asymptoticprimesatisfyingstuff_unconditional_loglogloglogx}
Let $a\in\Q\setminus\{-1,0,1\}$, and let $m$ and $\ell$ be positive integers with $m\leq\ell\leq\log\log\log x$. Unconditionally,
\begin{align*}
\#\Big\{ p\le x \colon \nu_p(a)=0,\, \ell\mid (p-1)/\ord_p(a),\ m\ell\mid (p-1)\Big\} &=\frac{\pi(x)}{[\Q(\sqrt[\ell]{a},\zeta_{m\ell}):\Q]}+\bo_\varepsilon\bigl(x\exp(-(\log x)^{1/2-\varepsilon})\bigr)
\end{align*}
for any $\varepsilon>0$.
\end{cor}

We can obtain the following unconditional result, which establishes equation~\eqref{oral exams take long} and thus Theorem~\ref{welcome both lim sup and tilde notation. sigh}.

\begin{thm}\label{theorem generating functions unconditional}
Let $a\in\Q\setminus\{-1,0,1\}$, let~$x$ be a sufficiently large real number, and set $\xi=\log\log\log\log x$. Unconditionally,
\begin{align*}
\frac1{\pi(x)} \sum_{\substack{p\le x \\ \nu_p(a)=0}} z^{\omega_\xi((p-1)/\ord_p(a))} &= F^{\omega/}_a(z)H^{\omega/}_a(z) \prod\limits_{\substack{q}}\biggl(1+\frac{z-1}{q(q-1)}\biggr)+\bo_\varepsilon\biggl(\frac{1}{(\log\log\log\log x)^{1-\varepsilon}}\biggr)
\end{align*}
uniformly for all $z\in\C$ with $|z|\leq 1$ and all $\varepsilon >0$,
where $F^{\omega/}_a(z)$ and $H^{\omega/}_a(z)$
are as in Definition~\ref{definition generating function omega of quotient}.
\end{thm}

\begin{proof}
First we compute the sum over $\ell$ of the error term of Corollary~\ref{prop_asymptoticprimesatisfyingstuff_unconditional_loglogloglogx}: it follows directly from Proposition~\ref{proposition error sum chebotarev error term combined} (with both sides divided by $\sqrt x\log x$) that for all $|z|\le 2$,
\begin{align}
 \sum\limits_{\ell\mid Q_\xi}\mu^2(\ell)|z-1|^{\omega(\ell)}\left(x\exp(-\ell^{-1}(\log x)^{1/2-\varepsilon})\right)&\ll\left(x\exp(-(\log x)^{1/2-\varepsilon})\right)\sum\limits_{\ell\mid Q_\xi}\mu^2(\ell)|z-1|^{\omega(\ell)} \notag \\
 &\ll\left(x\exp(-(\log x)^{1/2-\varepsilon})\right) 4^\xi \notag \\
 &\ll \left(x\exp(-(\log x)^{1/2-\varepsilon})\right).
\label{later z'll exceed 1}
\end{align}
Then we need to take into account the error arising from completing the sum over $\ell$ for the main term as in Proposition~\ref{proposition combined after error terms}, which in this case gives an error term of $\bo_\varepsilon\left( (\log\log\log\log x)^{1-\varepsilon}\right)$. Finally the computations for the main term are exactly the same as for the conditional result.
\end{proof}

\begin{rmk}
These error terms are certainly not best possible: with more work one should be able to obtain something similar to Wiertelak's result~\cite{Wiertelak} for the upper bound in the classical Artin's conjecture setting. Other results on unconditional upper bounds are discussed in~\cite{Moree} as well.
\end{rmk}

We state without further proof the analogous results for the other two statistics.
In the notation of Definitions~\ref{definition generating function Omega of quotient} and~\ref{definition generating function omega minus omega}, define the Maclaurin series coefficients $\tilde D_a^{\Omega}(n)$ and $\tilde D_a^{\omega-}(n)$ by
\begin{equation} \label{other two tildes}
\begin{split}
\sum_{n=0}^\infty \tilde D_a^{\Omega}(n) z^n &= F^{\Omega}_a(z)H^{\Omega}_a(z)\prod_{\substack{q}}\biggl(1+\frac{(z-1)q}{(q-1)(q^2-z)}\biggr) \\
\sum_{n=0}^\infty \tilde D_a^{\omega-}(n) z^n &= F^{\omega-}_a(z)H^{\omega-}_a(z)\prod\limits_{\substack{q}}\biggl(1+\frac{z-1}{q^2-1}\biggl),
\end{split}
\end{equation}
so that Theorems~\ref{theorem generating function Omega of quotient} and~\ref{theorem generating function omega minus omega} are the statements (conditional on GRH) that $\lim_{x\to\infty} D_a^{\Omega}(n;x) = \tilde D_a^{\Omega}(n)$ and $\lim_{x\to\infty} D_a^{\omega-}(n;x) = \tilde D_a^{\omega-}(n)$ for all $n\ge0$.

\begin{thm} \label{welcome both lim sup and tilde notation. no sigh}
Unconditionally, for all $k\ge0$,
\[
\limsup_{x\to\infty} \sum_{n=0}^k D_a^{\Omega}(n;x) \le \sum_{n=0}^k \tilde D_a^{\Omega}(n)
\quad\text{and}\quad
\limsup_{x\to\infty} \sum_{n=0}^k D_a^{\omega-}(n;x) \le \sum_{n=0}^k \tilde D_a^{\omega-}(n).
\]
\end{thm}

\noindent Note that the two assertions in Corollary~\ref{unconditional inequality corollary k=0} are the special case $k=0$ of these two inequalities.

Furthermore, these unconditional results lead to an unconditional lower bound for the expected number of prime factors of $(p-1)/\ord_p(a)$ (and similarly for the other two statistics) described in Section~\ref{subsection expectations}, which is the subject of the next section.

\subsection{Unconditional lower bound for the expectation} \label{Unconditional lower bound for the expectation}

Recall that if $X$ is a nonnegative discrete random variable with $\mathbb P(X=n)=a_n$, and $f(z)=\sum_{n=0}^\infty a_n z^n$ is the generating function for these probabilities, then
\begin{equation} \label{Cauchy}
    \mathbb E(X) = \sum_{n=0}^\infty na_n=f'(1)=\int_C \frac{f(z)}{(z-1)^2} \,dz
\end{equation}
by the Cauchy integral formula, where~$C$ is any simple closed curve around~$1$ on which $f$ is analytic.
We apply this identity in the following result, for which we need the following analogue of Theorem~\ref{theorem generating functions unconditional} for~$z$ on an open set containing~$z=1$.

\begin{thm}\label{theorem generating functions unconditional around 1}
Let $a\in\Q\setminus\{-1,0,1\}$, let~$x$ be a sufficiently large real number, and set $\xi=\log\log\log\log x$. Unconditionally,
\begin{align*}
\sum_{n=0}^\infty nD_{a,\xi}^{\omega/}(n;x)&=\frac{dF^{\omega/}_a}{dz}(1) + \frac{dH^{\omega/}_a}{dz}(1)+\sum\limits_q\frac{1}{q(q-1)}+\bo_\varepsilon\biggl(\frac{1}{(\log\log\log\log x)^{1-\varepsilon}}\biggr)
\end{align*}
uniformly for all $\varepsilon >0$,
where $F^{\omega/}_a(z)$ and $H^{\omega/}_a(z)$ are as in Definitions~\ref{definition generating function omega of quotient}.
\end{thm}

\begin{proof}
Our strategy is to use equation~\eqref{Cauchy} for the function
\[
\sum_{n=0}^\infty nD_{a,\xi}^{\omega/}(n;x) = \frac{1}{\pi(x)} \sum_{\substack{p\le x \\ \nu_p(a)=0}} z^{\omega_\xi((p-1)/\ord_p(a))}.
\]
Set $A=1/\log\log\log\log\log x$.
By equation~\eqref{later z'll exceed 1}, the error term from Theorem~\ref{theorem generating functions unconditional} is admissible for all $z\in\C$ with $|z-1|\le A$.
Therefore in this range,
\begin{equation*}
    \frac{1}{\pi(x)}\sum_{\substack{p\le x \\ \nu_p(a)=0}} z^{\omega_\xi((p-1)/\ord_p(a))}=F^{\omega/}_a(z)H^{\omega/}_a(z)\prod\limits_{\substack{q}}\biggl(1+\frac{z-1}{q(q-1)}\biggr)+\bo_\varepsilon\left(\frac{1}{(\log\log\log\log x)^{1-\varepsilon}}\right).
\end{equation*}
Now applying equation~\eqref{Cauchy}, and doing the same computation for the main term as in Theorem~\ref{Theorem conditional expectations}, we obtain
\begin{align*}
\sum_{n=0}^\infty nD_{a,\xi}^{\omega/}(n;x)&=\frac{1}{2i\pi}\int_{|z-1|=A/2} \frac{1}{\pi(x)}\sum_{\substack{p\le x \\ \nu_p(a)=0}} z^{\omega_\xi((p-1)/\ord_p(a))} \cdot \frac{dz}{(z-1)^2}\\
&=\frac{dF^{\omega/}_a}{dz}(1) + \frac{dH^{\omega/}_a}{dz}(1)+\sum\limits_q\frac{1}{q(q-1)}+\bo_\varepsilon\left(\frac{1}{(\log\log\log\log x)^{1-\varepsilon}}\int_{|z-1|=A/2}\frac{dz}{|z-1|^2}\right)\\
&=\frac{dF^{\omega/}_a}{dz}(1) + \frac{dH^{\omega/}_a}{dz}(1)+\sum\limits_q\frac{1}{q(q-1)}+\bo_\varepsilon\left(\frac{1}{A^2(\log\log\log\log x)^{1-\varepsilon}}\right)\\
&=\frac{dF^{\omega/}_a}{dz}(1) + \frac{dH^{\omega/}_a}{dz}(1)+\sum\limits_q\frac{1}{q(q-1)}+\bo_\varepsilon\left(\frac{1}{(\log\log\log\log x)^{1-\varepsilon}}\right)
\end{align*}
as claimed.
\end{proof}

To derive a lower bound for $\sum_{n=0}^\infty nD_{a}(n;x)$ from this result we will need the following simple result on positive discrete random variables.

\begin{lemma} \label{compare expectations lemma}
Let $a_n,b_n\ge0$ with $\sum_{n=0}^\infty a_n = \sum_{n=0}^\infty b_n = 1$. Suppose that $\sum_{n=0}^j a_n \le \sum_{n=0}^j b_n$ for all $j\ge0$. Then $\mathbb E(a_n) \ge \mathbb E(b_n)$.
\end{lemma}

\begin{proof}
We have
\begin{equation} \label{from expectation to double sum}
\mathbb E(a_n) = \sum_{n=0}^\infty na_n = \sum_{n=0}^\infty a_n \sum_{k=1}^n 1 = \sum_{k=1}^\infty \sum_{n=k}^\infty a_n = \sum_{k=1}^\infty \biggl( 1 - \sum_{n=0}^{k-1} a_n \biggr).
\end{equation}
Therefore
\[
\mathbb E(b_n) = \sum_{k=1}^\infty \biggl( 1 - \sum_{n=0}^{k-1} b_n \biggr) \le \sum_{k=1}^\infty \biggl( 1 - \sum_{n=0}^{k-1} a_n \biggr) = \mathbb E(a_n).
\qedhere
\]
\end{proof}

\begin{proof}[Proof of Theorem~\ref{welcome both lim sup and tilde notation. sigh}]
From the inequality~\eqref{cumulative xi inequality2} and Lemma~\ref{compare expectations lemma},
\begin{equation}
\sum_{n=0}^\infty nD^{\omega/}_{a}(n;x)\geq \sum_{n=0}^\infty nD^{\omega/}_{a,\xi}(n;x),
\end{equation}
and taking the $\liminf$ on both sides and using Theorem~\ref{theorem generating functions unconditional around 1} yields the result.
\end{proof}

The method immediately yields analogous statements for the expectations of the other two statistics:

\begin{thm}\label{Corollary unconditional lower bound expectation}
Unconditionally,
\begin{align*}
\liminf_{x\to\infty} \sum_{n=0}^\infty nD_a^{\Omega}(n;x)&\ge\frac{dF^{\Omega}_a}{dz}(1) + \frac{dH^{\Omega}_a}{dz}(1)+\sum\limits_q\frac{q}{q^3-q^2-q+1} \\
\liminf_{x\to\infty} \sum_{n=0}^\infty nD_a^{\omega-}(n;x)&\ge\frac{dF^{\omega-}_a}{dz}(1) + \frac{dH^{\omega-}_a}{dz}(1)+\sum\limits_q\frac{1}{q^2-1},
\end{align*}
where $F^{\Omega}_a(z)$, $H^{\Omega}_a(z)$, $F^{\omega-}_a(z)$ and $H^{\omega-}_a(z)$ are as in Definitions~\ref{definition generating function Omega of quotient} and~\ref{definition generating function omega minus omega}.
\end{thm}

\printbibliography

\end{document}